\documentclass[arxiv]{agn_article}

\usepackage[utf8]{inputenc}

\usepackage{csquotes}

\usepackage{amssymb,amsmath,amsthm,agn_bib}

\usepackage{agn_macros,agn_authors}
\usepackage{mathrsfs}
\usepackage{graphicx}
\usepackage{multicol}
\usepackage[inline]{enumitem}
\usepackage{array,multirow}
\usepackage[normalem]{ulem}
\usepackage{arydshln}
\usepackage[super]{nth}

\newcounter{alphasect}
\def\alphainsection{0}

\let\oldsection=\section
\def\section{%
  \ifnum\alphainsection=1%
    \addtocounter{alphasect}{1}
  \fi%
\oldsection}%
\renewcommand\thesection{%
  \ifnum\alphainsection=1%
    \Alph{alphasect}%
  \else%
    \arabic{section}%
  \fi%
}%

\newenvironment{alphasection}{%
  \ifnum\alphainsection=1%
    \errhelp={Let other blocks end at the beginning of the next block.}
    \errmessage{Nested Alpha section not allowed}
  \fi%
  \setcounter{alphasect}{0}
  \def\alphainsection{1}
}{%
  \setcounter{alphasect}{0}
  \def\alphainsection{0}
}%

\AtBeginBibliography{\footnotesize}

\newtheorem{theorem}{Theorem}[section]
\newtheorem{proposition}[theorem]{Proposition}
\newtheorem{lemma}[theorem]{Lemma}
\newtheorem{observation}[theorem]{Observation}
\newtheorem{corollary}[theorem]{Corollary}

\theoremstyle{definition}

\newtheorem{remark}[theorem]{Remark}

\usepackage{xcolor,hyperref}
\definecolor{darkblue}{rgb}{0,0,0.6}
\definecolor{fg}{RGB}{34,139,34}  
\newcommand{\corrected}[1]{#1}

\hypersetup{
    colorlinks=true,       
    linkcolor=darkblue,          
    citecolor=darkblue,        
    filecolor=darkblue,      
    urlcolor=darkblue           
}

\usepackage[colorinlistoftodos,prependcaption,textsize=tiny]{todonotes}
\newcommand{\skalarProd}[2]{\big\langle#1,#2\big\rangle}

\newcommand{\inc}{\boldsymbol{\operatorname{inc}}\,}
\renewcommand{\D}{\operatorname{D}\hspace{-1pt}}
\newcommand{\curl}{\operatorname{curl}}

\newcommand{\Ce}{\mathbb{C}_{\mathrm{e}}}
\newcommand{\Cc}{\mathbb{C}_{\mathrm{c}}}
\newcommand{\Cmicro}{\mathbb{C}_{\mathrm{micro}}}

\newcommand{\Lc}{L_{\mathrm{c}}}

\newcommand{\Anti}{\operatorname{Anti}}

\newcommand{\T}{\mathbb{T}}
\newcommand{\komplexI}{\mathrm{i}}

\newlist{thmenum}{enumerate}{1}
\setlist[thmenum]{label=\upshape(\alph*)}

\defineaffiliation{hcm}{%
	 Hausdorff Center for Mathematics and Institute for Applied Mathematics, Universit\"at Bonn, Endenicher Allee 60, 53115 Bonn, Germany
}
\defineauthor{mueller}{Stefan~M\"uller}{hcm}{stefan.mueller@hcm.uni-bonn.de}

\usepackage{tikz}

\begin{document}
\begin{tikzpicture}[remember picture, overlay]
 \node [xshift=-1cm,yshift=15cm,rotate=-90] at (current page.south east)
 {Calculus of Variations and Partial Differential Equations (2021), doi: \href{https://doi.org/10.1007/s00526-021-02000-x}{10.1007/s00526-021-02000-x}.
 };
\end{tikzpicture}
\numberwithin{equation}{section}

 \title{Korn inequalities for incompatible tensor fields in three space dimensions with conformally invariant dislocation energy}
\knownauthors[lewintan]{lewintan,mueller,neff}

\maketitle

\begin{abstract} Let $\Omega \subset \R^3$ be an open and bounded set with Lipschitz boundary and outward unit normal $\nu$. For $1<p<\infty$ we establish an improved version of the generalized $L^p$-Korn inequality for incompatible tensor fields $P$ in the new Banach space 
\begin{align*}
  W^{1,\,p,\,r}_0&(\dev\sym\Curl; \Omega,\R^{3\times3}) \\
 = & \{ P \in L^p(\Omega; \R^{3 \times 3}) \mid \dev \sym \Curl P \in L^r(\Omega; \R^{3 \times 3}),\   \dev \sym (P \times \nu) = 0 \text{ on $\partial \Omega$}\}
 \end{align*}
where 
 $$ r \in [1, \infty), \qquad \frac1r \le \frac1p + \frac13, \qquad r >1 \quad \text{if $p = \frac32$.}$$
 
 Specifically, there exists a constant $c=c(p,\Omega,r)>0$ such that the inequality
 \[
\norm{ P }_{L^p(\Omega,\R^{3\times3})}\leq c\,\left(\norm{\sym P }_{L^p(\Omega,\R^{3\times3})} + \norm{ \dev\sym \Curl P }_{L^{r}(\Omega,\R^{3\times3})}\right)
\]
holds for all tensor fields $P\in  W^{1,\,p, \, r}_0(\dev\sym\Curl; \Omega,\R^{3\times3})$.
Here, $\dev X \coloneqq X -\frac13 \tr(X)\,\id$ denotes the deviatoric (trace-free) part of a $3 \times 3$ matrix $X$ and the boundary condition is
understood in a suitable weak sense. This estimate also holds true if the boundary condition is only satisfied on  a relatively open, non-empty subset $\Gamma \subset \partial \Omega$.
If no boundary conditions are imposed then the estimate holds after taking the quotient with the finite-dimensional space $K_{S,dSC}$ which is determined by the conditions $\sym P =0$ and $\dev \sym \Curl P = 0$. In that case one can replace $\norm{ \dev\sym \Curl P }_{L^r(\Omega,\R^{3\times3})} $ by $\norm{ \dev\sym \Curl P }_{W^{-1,p}(\Omega,\R^{3\times3})}$. The new $L^p$-estimate implies  a classical Korn's inequality with weak boundary conditions by choosing $P=\D u$ and 
a deviatoric-symmetric generalization of Poincar\'{e}'s inequality by choosing $P=A\in\so(3)$.

The proof relies on a representation of the third derivatives $\D^3 P$ in terms of $\D^2 \dev \sym \Curl P$ combined with the Lions lemma and the Ne\v{c}as estimate.

We also discuss applications of the new inequality to the relaxed micromorphic model, to Cosserat models with the weakest form of the curvature 
energy,  to gradient plasticity with plastic spin and to incompatible linear elasticity.
\end{abstract}

\par


\section{Introduction}

\subsection{Overview}
Korn's second inequality provides an $L^p$-estimate  of a gradient vector field
(modulo a constant) in terms of the symmetric part of the derivative. This can be generalized to general fields $P$ if one adds a term  in $\Curl P$
on the right hand side \cite{agn_lewintan2019KornLp,agn_lewintan2019KornLpN,agn_neff2015poincare}. For recent refined estimates which involve only the deviatoric part of $\sym P$ and $\Curl P$, see \cite{agn_lewintan2020KornLp_tracefree,agn_bauer2013dev,agn_lewintan2020KornLpN_tracefree}.

Here, we show that $P$  can be estimated  in dimension $n=3$ in terms of $\sym P$ and
 $\sym \Curl P$ or even $\sym P$ and $\dev \sym \Curl P$.  The difference is that we  need to subtract not only constants but also certain affine  or quadratic skew-symmetric fields  in  the kernel of  the operators $\sym \Curl$ and $\dev \sym \Curl$, respectively.

\corrected{To set the stage  we recall the notation for the relevant Lie groups used in this paper
and their Lie algebras and indicate  how our new inequalities relate to  (infinitesimal) conformal invariance.
We denote the space of $(n \times n)$-matrices by $\R^{n \times n}$ and we denote  the groups of 
proper orthogonal matrices, and matrices with determinant $1$ by
\begin{subequations}
\begin{align}
 \SO(n) &= \{ Q \in \R^{n\times n} \mid Q^T Q = \id\},&  \SL(n) &= \{ B \in \R^{n \times n} \mid \det B = 1\}.
\intertext{The corresponding Lie algebras  of skew-symmetric and trace-free matrices are denoted by
}
\so(n) &= \{  A \in \R^{n\times n} \mid A^T = -A\},& \sl(n) &= \{D \in \R^{n \times n} \mid \tr D = 0\}.
\end{align}
\end{subequations}

Let $\Omega \subset \R^3$ be open, bounded and simply connected. A $C^1$-map $\varphi :\Omega \to \R^3$ is conformal if its differential preserves the scalar product up to dilations, i.e., if 
for all $x$ there exist $\lambda(x) \ge 0$ and $Q(x) \in \SO(3)$ such that $\D \varphi(x) = \lambda(x) Q(x)$. 
It is well-known that conformal maps are smooth and the non-constant conformal maps form a finite-dimensional
manifold. The vector fields in the tangent space of the identity map are called \emph{conformal Killing fields} (or \emph{infinitesimally conformal maps})
and are characterized by the condition
\begin{equation}\label{eq:confcond}
 \dev \sym \D u  = 0.
\end{equation}
In fact the solutions $u$ of \eqref{eq:confcond} are certain quadratic polynomials, see \eqref{eq:infinitesimalconfis} below for an explicit formula.

In the Cosserat theory, the curvature expression
\begin{subequations}
\begin{equation}
 \norm{\dev\sym\Curl A}^2 \quad \text{for } A:\Omega\to\so(3)
\end{equation}
can be expressed equivalently as
\begin{equation}\label{eq:explanation}
 \norm{\dev\sym\D \axl(A)}^2.
\end{equation}
and the latter expression has been termed ``conformal curvature'', consistent with \eqref{eq:confcond} for $u=\axl(A)$. Therefore, we call the generalized curvature expression 
\begin{equation}
 \norm{\dev\sym\Curl P}^2
\end{equation}
\emph{conformal dislocation energy}. Upon restricting $P\in\so(3)$ we recover \eqref{eq:explanation}, see also \eqref{eq:conform}.
\end{subequations}
}

\subsection{The classical Korn's inequalities}
First inequalities of this type were identified by Arthur Korn more than hundred years ago, cf.~\cite{Korn1,Korn2,Korn3}, where they were derived for applications in linear elasticity. It is worth mentioning that after his graduation in 1890 Korn studied in Paris under the supervision of Henri Poincar\'{e}. For Korn's biography including his pioneering work in telephotography we refer to \cite{Litten,Korn-bio} but also \cite[p.~182f]{Korn-bio-DMV}.

We start by summarizing the inequalities which bear Korn's name. In the following, let $n\ge2$ and $\Omega\subset\R^n$ be a bounded Lipschitz domain. Korn's first inequality (in $L^p$) with vanishing boundary values\footnote{In fact, the estimate is also true for functions with vanishing boundary values on a relatively open (non-empty) subset of the boundary.} reads
\begin{align}
\label{eq:Korn1}
\norm{\D u}_{L^p(\Omega,\R^{n\times n})} &\leq c\, \norm{\sym \D u}_{L^{p}(\Omega,\R^{n\times n})} \qquad &&\forall\, u\in W^{1,\,p}_{0}(\Omega,\R^n).
\intertext{
It can be deduced from Korn's second inequality (in $L^p$), which does not require boundary conditions:
}
\label{eq:Korn2}
\norm{u}_{W^{1,\,p}(\Omega,\R^n)}&\leq c\, \left(\norm{u}_{L^p(\Omega,\R^n)} + \norm{\sym \D u}_{L^{p}(\Omega,\R^{n\times n})} \right)  &&\forall\, u\in W^{1,\,p}(\Omega,\R^n).
\intertext{
From the latter inequality also the following version follows}
\label{eq:KornQuant}
\inf_{A\in\so(n)}\norm{\D u- A}_{L^p(\Omega,\R^{n\times n})}&\leq c\, \norm{\sym \D u}_{L^{p}(\Omega,\R^{n\times n})}  &&\forall\, u\in W^{1,\,p}(\Omega,\R^n).
\intertext{For $n\geq 3$ these inequalities can be improved to inequalities which only require the trace-free part of $\sym \D u$ on the right hand side. One has}
\label{eq:Korn1-dev}
 \norm{\D u}_{L^p(\Omega,\R^{n\times n})}&\leq c\, \norm{\dev_n\sym \D u}_{L^{p}(\Omega,\R^{n\times n})} &&\forall\, u\in W^{1,\,p}_{0}(\Omega,\R^n),
\intertext{where $\dev_n X\coloneqq X -\frac1n\tr(X)\cdot \id$ denotes the deviatoric (trace-free) part of the square matrix $X\in\R^{n\times n}$. Moreover,}
\label{eq:Korn2-dev}
 \norm{u}_{W^{1,\,p}(\Omega,\R^n)}&\leq c\, \left(\norm{u}_{L^p(\Omega,\R^n)} + \norm{\dev_n\sym \D u}_{L^{p}(\Omega,\R^{n\times n})} \right) &&\forall\, u\in W^{1,\,p}(\Omega,\R^n),
\intertext{as well as}
\label{eq:KornQuant-dev}
\norm{u- \Pi u}_{W^{1,\,p}(\Omega,\R^{n\times n})}&\leq c\, \norm{\dev_n\sym \D u}_{L^{p}(\Omega,\R^{n\times n})}  &&\forall\, u\in W^{1,\,p}(\Omega,\R^n)
\intertext{where $\Pi$ is an arbitrary projection from $W^{1,\,p}(\Omega,\R^n)$ onto the space of \textit{conformal Killing vectors}  (or infinitesimal conformal mappings), i.e.,~the finite-dimensional kernel of $\dev_n\sym\D$, which is given by quadratic polynomials of the form
\begin{equation}\label{eq:infinitesimalconfis}
\varphi_C(x)=\skalarProd{a}{x}\,x-\frac12a\norm{x}^2+A\,x +\beta\,x +b, \qquad \text{with $A\in\so(n)$, $a,b\in\R^n$ and $\beta\in\R$,}
\end{equation}
cf. \cite{Reshetnyak1970,agn_neff2009subgrid,agn_jeong2008existence,Dain2006tracefree,Reshetnyak1994, Schirra2012tracefreenD}. The situation is quite different in the planar case $n=2$, since the condition $\dev_2\sym\D u\equiv0$ becomes the system of Cauchy-Riemann equations and the corresponding kernel is infinite-dimensional, so that an adequate quantitative version of the trace-free classical Korn's inequality does not hold true. However, in \cite{FuchsSchirra2009tracefree2D} it is proved that}
 \norm{\D u}_{L^p(\Omega,\R^{2\times2})}&\le c\,\norm{\dev_2\sym\D u}_{L^p(\Omega,\R^{2\times2})}  &&\forall \, u\in W^{1,\,p}_0(\Omega,\R^2),
\end{align}
but this result ceases to be valid if the homogeneous Dirichlet conditions are prescribed only on a part of the boundary, cf.~the counterexample in \cite[sec. 6.6]{agn_bauer2013dev}.

For the limiting cases $p=1$ and $p=\infty$ Korn-type inequalities fail, since from the counterexamples traced back in \cite{CFM2005counterex,dlM1964counterex,Ornstein1962,Mityagin1958} it follows that  $\int_\Omega\norm{\sym\D u}\,\intd{x}$ does not dominate each of the quantities $\int_\Omega \abs{\partial_i u_j}\,\intd{x}$ for any vector field $u\in W^{1,\,1}_0(\Omega,\R^n)$. Nevertheless, Poincar\'{e}-type inequalities estimating certain integral norms of the deformation $u$ in terms of the total variation of the symmetric strain tensor $\sym\D u$ are still true. For Poincar\'{e}-type inequalities for functions of bounded deformation involving only the deviatoric part of the symmetrized gradient we refer to \cite{FuchsRepin2010Poincaretracefree}.

These Korn inequalities are crucial for a priori estimates in linear elasticity and fluid mechanics, so that they are cornerstones for well-posedness results in linear elasticity ($L^2$-setting) and the Stokes-problem ($L^p$-setting), cf.~\cite{horgan1995} and \cite{Ciarlet2013FAbook} for a modern elaboration, whereas the trace-free equivalents found applications in micropolar Cosserat-type models \cite{agn_jeong2009numerical,agn_neff2009subgrid,agn_jeong2008existence,FuchsSchirra2009tracefree2D} and general relativity \cite{Dain2006tracefree}.

The Korn inequalities generalize to many different settings, including the geometrically nonlinear counterpart \cite{friesecke2002rigidity,LM2016optimalconstants}, mixed growth conditions \cite{CDM2014mixedgrowth}, incompatible fields (also with dislocations) \cite{MSZ2014incompatible,agn_neff2015poincare,agn_bauer2013dev,agn_lewintan2019KornLp,agn_lewintan2019KornLpN,agn_lewintan2020KornLp_tracefree,Garroni10,conti2020sharp,gmeineder2020kornmaxwellsobolev} and trace-free infinitesimal strain measures \cite{Dain2006tracefree, agn_jeong2008existence,Reshetnyak1970,Reshetnyak1994,FuchsSchirra2009tracefree2D,Schirra2012tracefreenD,agn_bauer2013dev,agn_lewintan2020KornLp_tracefree,agn_lewintan2020KornLpN_tracefree}. For trace-free Korn's inequalities in pseudo-Euclidean space see \cite{Wang2008tracefreeinpseudo} and for trace-free Korn inequalities on manifolds see \cite{Dain2006tracefree, holst2007rough}. It is also possible to consider tangential boundary conditions, cf.~\cite{Villani-tangentBCs,Ryzhak-tangentBCs,Bauer-tangentBCs,Bauer-tangentBCs2,domnguez2019korns}. Other generalizations are applicable to Orlicz-spaces \cite{Fuchs2010,BD2012Orlicz,BCD2017Orlicz,Cianchi2014Orlicz,Fuchs2010tracefreeKorn} and SBD functions with small jump sets \cite{Friedrich2017SBD,CCF2016SBD,Friedrich2018SBD}, thin domains \cite{LM2011thindomians,GH2018thindomains,Harutyunyan2017thindomains,miura2020navierstokes}  as well as the case of non-constant coefficients \cite{agn_neff2002korn,agn_lankeit2013uniqueness,agn_neff2014counterexamples,Pompe2003Korn}. Moreover Korn-type inequalities are valid on H\"older and John domains, see \cite{JK2017John,DM2004Jonesdomains,DRS2010Johndomains,ADM2006Johndomains,LopezGarcia2018tracefreeKorn,DingBo2020tracefreeKorn} and also the recent monograph \cite{AcostaDuran2017book} which
 relates those Korn inequalities to the existence of a right inverse of the divergence operator, to the Stokes equations and other inequalities. Piecewise Korn-type inequalities subordinate to a FEM-mesh and involving jumps across element boundaries have also been investigated, see e.g.~\cite{Brenner2004Korn,agn_lew2004optimal}. In the recent paper \cite{spector2020bmo} the authors established a Korn inequality involving the BMO-seminorms which is valid on all bounded domains and with a constant depending only on the dimension. Here we focus on inequalities for general tensor fields.

\subsection{Korn-type inequalities for incompatible tensor fields}
Classical Korn's inequalities require compatibility, i.e.,~a gradient $\D u$ (the Jacobian matrix). Generalizations of such estimates to general fields $P$ then need a control of the distance of $P$ to a gradient by adding the incompatibility measure (the dislocation density tensor) $\Curl P$. The matrix $\Curl$ operation is to be understood as row-wise application of the classical $\curl$ to vectors. Even though
the usual $\Curl$ operation on $\R^3$ has a natural extension to all dimensions, the case $n=3$  deserves our special attention, not only from the viewpoint of modeling but also since the matrix $\Curl$ then returns a square matrix in three dimensions. As direct generalization of Korn's first inequality \eqref{eq:Korn1} we have for all $P\in  W^{1,\,p}_0(\Curl; \Omega,\R^{3\times3})$
\begin{align}\label{eq:Korn_Lp}
 \norm{ P }_{L^p(\Omega,\R^{3\times3})}&\leq c\,\left(\norm{ \sym P }_{L^p(\Omega,\R^{3\times3})} + \norm{ \Curl P }_{L^p(\Omega,\R^{3\times3})}\right),
\intertext{cf. \cite{agn_neff2015poincare} for $p=2$ and \cite{agn_lewintan2019KornLp} for all $p>1$. Furthermore, the version \eqref{eq:KornQuant} generalizes  to
}
 \inf_{\widetilde{A}\in\so(3)}\norm{P-\widetilde{A}}_{L^p(\Omega,\R^{3\times3})}&\leq c\,\left(\norm{ \sym P }_{L^p(\Omega,\R^{3\times3})}+ \norm{ \Curl P }_{L^p(\Omega,\R^{3\times3})}\right)\label{eq:Korn_Lp_w}
 \intertext{
 for all $P\in  W^{1,\,p}(\Curl; \Omega,\R^{3\times3})$, cf. \cite{agn_lewintan2019KornLp}. These estimates also hold true in all dimensions $n\geq2$ with an adequate understanding of the matrix $\Curl$,  \cite{agn_lewintan2019KornLpN}.  However, in two dimensions even stronger estimates hold true, cf. \cite{Garroni10} and its nonlinear counterpart in \cite{MSZ2014incompatible}, so especially, for fields $P\in L^1(\Omega,\R^{2\times2})$ with $\Curl P\in L^1(\Omega,\R^2)$ it follows that  $P\in L^2(\Omega,\R^{2\times2})$ and}
\norm{P}_{L^2(\Omega,\R^{2\times2})} &\leq c\,\left(\norm{\sym P}_{L^2(\Omega,\R^{2\times2})} + \norm{\Curl P}_{L^1(\Omega,\R^2)}\right)\label{eq:garroni}
\intertext{
under the normalization condition $\int_\Omega \skew P\, \intd{x}= 0$, cf. \cite{Garroni10}. However, this is essentially a result for the divergence, since $\Div$ is a rotated $\Curl$ in two dimensions.\footnotemark\ Indeed, the authors of \cite{Garroni10} make use of the fact that a vector field $u\in L^1(\Omega,\R^2)$ satisfying $\div u\in H^{-2}(\Omega)$ belongs to $H^{-1}(\Omega)$ with}
 \norm{u}_{H^{-1}(\Omega,\R^2)}&\le c\,(\norm{u}_{L^1(\Omega,\R^2)}+\norm{\div u}_{H^{-2}(\Omega)})
 \end{align}
which follows from \cite{BrezisSchaftingen2007}. For the geometrically nonlinear counterpart of \eqref{eq:garroni} in a mixed-growth setting in two dimensions we refer the reader to \cite{Ginster2019gradientplast} and higher-dimensional analogues can be found in \cite{lauteri2017geometric,conti2020sharp}.
\footnotetext{The Babu\v{s}ka-Aziz theorem implies that over a planar Lipschitz domain $\Omega\subset\R^2$ it holds
\[
\norm{\D u}_{L^2(\Omega,\R^{2\times2})} \le c\,\norm{\div u}_{L^2(\Omega,\R)}
\]
for all $u\in H^1_0(\Omega,\R^2)$ such that $\int_\Omega\div u\,\intd{x}=0$, cf.~\cite[Section 6]{horgan1995}.
}

 Improvements of the Korn inequalities for incompatible tensor fields \eqref{eq:Korn_Lp} and  \eqref{eq:Korn_Lp_w} towards the trace-free cases are also valid. For all $P\in  W^{1,\,p}_0(\Curl; \Omega,\R^{3\times3})$, where
 \begin{align}
 W^{1,\,p}_0(\Curl; \Omega,\R^{3\times3})
 &\coloneqq \{P\in L^p(\Omega,\R^{3\times3}) \mid \Curl P \in L^p(\Omega,\R^{3\times3}),\ P \times \nu = 0 \text{ on } \partial \Omega\}
\intertext{ one has}
 \norm{ P }_{L^p(\Omega,\R^{3\times3})}&\leq c\,\left(\norm{\dev \sym P }_{L^p(\Omega,\R^{3\times3})} + \norm{ \dev\Curl P }_{L^p(\Omega,\R^{3\times3})}\right),\label{eq:Korn_Lp-dev}
 \intertext{cf. \cite{agn_bauer2013dev}  for $p=2$ and \cite{agn_lewintan2020KornLp_tracefree} for all $p>1$. Moreover,  we have}
 \inf_{T\in K_{dS,dC}}\norm{P-T}_{L^p(\Omega,\R^{3\times3})}&\leq c\,\left(\norm{\dev \sym P }_{L^p(\Omega,\R^{3\times3})}+ \norm{\dev\Curl P }_{L^p(\Omega,\R^{3\times3})}\right)\label{eq:Korn_Lp_dev_quant_all}
 \intertext{for all $P\in  W^{1,\,p}(\Curl; \Omega,\R^{3\times3})$, cf. \cite{agn_lewintan2020KornLp_tracefree}, where the kernel on the right hand side is given by}
  K_{dS,dC} = \{T:\Omega\to\R^{3\times3} \mid   T(x)&=\Anti\big(\widetilde{A}\,x+\beta\, x+b \big)+\big(\skalarProd{\axl\widetilde{A}}{x}+\gamma \big)\,\id, \notag \\
  &\hspace{10em} \widetilde{A}\in\so(3), b\in\R^3, \beta,\gamma\in\R\}\,,\label{eq:kernel_dSdC}
\end{align}
where ~ $\Anti: \R^3\to\so(3)$  is the canonical identification (consistent with the vector product) of $\R^3$ and the vectorspace of skew-symmetric matrices  $\so(3)$ and ~$\axl:\so(3)\to\R^3$ is its inverse.
The appearance of the $\dev\Curl$ operator on the right hand side would suggest to extend the Banach space $W^{1,\,p}(\Curl; \Omega,\R^{3\times3})$ to $p$-integrable tensor fields $P$ with $p$-integrable $\dev\Curl P$, but this would not be a new space. Indeed, in \cite{agn_lewintan2020KornLp_tracefree} the authors showed that for all $P\in\mathscr{D}'(\Omega,\R^{3\times3})$ and all $m\in\Z$ one has
  \begin{equation}\label{eq:reicht_ohne_dev}
     \Curl P\in W^{m,\,p}(\Omega,\R^{3\times3})\quad \Leftrightarrow\quad \dev\Curl P\in W^{m,\,p}(\Omega,\R^{3\times3}).
  \end{equation}
Note, that the estimates \eqref{eq:Korn_Lp-dev} and \eqref{eq:Korn_Lp_dev_quant_all} are strictly restricted to the case of three dimensions since the deviatoric operator acts on square matrices and only in the three-dimensional setting the matrix $\Curl$ operator returns again a square matrix. On the other hand, the corresponding weaker estimates in terms of $\norm{\dev_n\sym  P}_{L^p} +\norm{\Curl P}_{L^p}$ hold true in all dimensions $n\geq 3$, cf. \cite{agn_bauer2013dev} for  $p=2$ and \cite{agn_lewintan2020KornLpN_tracefree} for all $p>1$.

For compatible $P=\D u$ we get back from \eqref{eq:Korn_Lp}, \eqref{eq:Korn_Lp_w}, \eqref{eq:Korn_Lp-dev} and \eqref{eq:Korn_Lp_dev_quant_all} the corresponding classical Korn inequalities. Recently, Gmeineder and Spector \cite{gmeineder2020kornmaxwellsobolev} extended inequality \eqref{eq:Korn_Lp} to the case where $\sym P$ is generalized to any linear operator $A(P)$ such that $A(\D u)$ is a first order elliptic operator, thus including also one result of \cite{agn_lewintan2020KornLp_tracefree} with $\dev\sym P$.

The objective of the present paper is to further improve on estimate \eqref{eq:Korn_Lp} by showing that it already suffices to consider the symmetric or even the trace-free symmetric part of the $\Curl$. More precisely, for all $P\in  W^{1,\,p, \, r}_0(\dev\sym\Curl; \Omega,\R^{3\times3})$ where
\begin{align*}
  W^{1,\,p, \, r}_0&(\dev\sym\Curl; \Omega,\R^{3\times3}) \corrected{\coloneqq} \\
  & \{ P \in L^p(\Omega; \R^{3 \times 3}) \mid \dev \sym \Curl P \in L^r(\Omega; \R^{3 \times 3}),\   \dev \sym (P \times \nu) = 0 \text{ on $\partial \Omega$}\}
 \end{align*}
and
 $$ r \in [1, \infty), \qquad \frac1r \le \frac1p + \frac13, \qquad r >1 \quad \text{if $p = \frac32$}$$
 there exists a constant $c=c(p,\Omega,r)>0$ such that one has
 \begin{align}\label{eq:Korn_Lp-devsymCurl}
\norm{ P }_{L^p(\Omega,\R^{3\times3})}&\leq c\,\left(\norm{\sym P }_{L^p(\Omega,\R^{3\times3})} + \norm{ \dev\sym \Curl P }_{L^{r}(\Omega,\R^{3\times3})}\right)\,.
\intertext{If no boundary conditions are imposed then we show}
 \inf_{T\in K_{S,dSC}}\norm{P-T}_{L^p(\Omega,\R^{3\times3})}&\leq c\,\left(\norm{\sym P }_{L^p(\Omega,\R^{3\times3})}+ \norm{\dev\sym\Curl P }_{L^r(\Omega,\R^{3\times3})}\right)\,,\label{eq:Korn_Lp_quant-devsymCurl}
 \intertext{where the kernel is given by}
  K_{S,dSC}\corrected{=}\{T:\Omega\to\R^{3\times3} \mid   T(x)&=\Anti\big(\widetilde{A}\,x+\beta\, x+b +\skalarProd{d}{x}\,x -\frac12d\norm{x}^2\big), \notag \\ &\hspace{10em} \widetilde{A}\in\so(3), b,d\in\R^3, \beta\in\R\} \label{eq:kernel_SdSC-intro}\,.
 \end{align}
 \begin{remark}
 The right-hand side of \eqref{eq:Korn_Lp-devsymCurl} provides a norm on smooth, compactly supported functions $P\in C^\infty_0(\Omega,\R^{3\times 3})$. Indeed, $\sym P\equiv 0$ implies $P=A\in C^\infty_0(\Omega,\so(3))$, so that by Nye's formula \eqref{eq:Nye}$_1$ the condition ~ $\dev\sym \Curl A\equiv0$ ~ reads already ~ $\dev\sym\D a \equiv0$ ~ with $a\coloneqq\axl A\in C^\infty_0(\Omega,\R^3)$, where $\axl:\so(3)\to\R^3$ associates to a skew-symmetric  matrix $A\in\so(3)$ the vector ~ $\axl A\coloneqq (-A_{23},A_{13},-A_{12})^T$. The trace-free Korn's inequality \eqref{eq:Korn1-dev} then gives ~ $\D a \equiv 0$.~ Hence, $a=\axl A $ is a constant vector field, $P=A$ is a constant skew-symmetric matrix field, and with the boundary condition we obtain $P\equiv0$.
\end{remark}
\begin{remark}\label{Rem:keineNorm}
On the other hand, there are no such estimates in terms of
$$\norm{\dev P}+\norm{\sym\Curl P}, \quad \norm{\dev\sym P}+\norm{\sym\Curl P}\quad \text{or}\quad \norm{\dev\sym P}+\norm{\dev\sym\Curl P}$$
due to the example ~ $P=\zeta\cdot\id$ ~ for which ~ $\Curl P =-\Anti(\nabla \zeta)$, ~ so that the corresponding right-hand sides would vanish, since here we have ~ $\dev P=\dev\sym P = 0$ ~ and also ~ $\sym \Curl P = \dev\sym\Curl P =0$.
\end{remark}

\subsection{Proof ideas for Korn inequalities}
There exist many different proofs of the classical Korn's inequalities, cf.~the discussions in \cite{Ciarlet2010,agn_neff2015poincare,ACM2015,KO88,Nitsche81,Gobert1962,HlavacekNecas-I,HlavacekNecas-II,Friedrichs1947,Fichera1950,Fichera1972,Eidus1951,BernsteinToupin1960,PayneWeinberger1961,Kato1979,Tiero1999,Campanato1971,Tiero2001equivanlenc_of_ineq,Ting1972Korn,Romano2000} as well as \cite[Sect. 6.15]{Ciarlet2013FAbook} and the references contained therein. A rather concise and elegant argument, see \cite{Geymonat86, DuvautLions72,,Villani-tangentBCs,Girault1986FEM} and also advocated by P. G. Ciarlet and his coworkers \cite{Ciarlet2010, Ciarlet2013FAbook, Ciarlet2005korn,Ciarlet2005intro, CMM2018} uses the Lions lemma resp.~Ne\v{c}as estimate, the compact embedding $\corrected{W^{1,\,p}\subset\!\subset L^p}$ and the well-known representation of the second distributional derivatives of the displacement $u$ by a linear combination of the first derivatives of the symmetrized gradient $\D u$, namely
\begin{subequations}\label{eq:sec_der_id}
\begin{align}
 \partial_i\partial_j u_k = \partial_j(\sym \D u)_{ik}+ \partial_i(\sym \D  u)_{jk}-\partial_k(\sym \D u)_{ij},
\shortintertext{i.e.}
\D^2 u = L(\D\,  \sym \D u) \text{ with a constant coefficient linear operator } L.
\end{align}
\end{subequations}
Also the trace-free Korn's inequalities can be deduced in such a way, relying on the ``higher order'' analogues of the differential relation \eqref{eq:sec_der_id}:
\begin{equation}
 \D\Delta u = L(\D^2 \dev_n\sym \D u),
\end{equation}
cf.~\cite{Dain2006tracefree} for the case $p=2$ and \cite{Schirra2012tracefreenD} for all $p>1$.

The first and the last author used a similar reasoning in their series of papers
\cite{agn_lewintan2019KornLp,agn_lewintan2019KornLpN,agn_lewintan2020KornLp_tracefree,agn_lewintan2020KornLpN_tracefree} to obtain the Korn inequalities for incompatible tensor fields mentioned above. In particular, the gradient of a skew-symmetric matrix field $A$ 
can be expressed as linear combination of the entries of the matrix $\Curl$:
\begin{equation}
 \D A = L(\Curl A),
\end{equation}
which in three dimensions reads exactly as \emph{Nye's formula}  \cite[eq.\!\! (7)]{Nye53}:
\begin{align}\label{eq:Nye}
 \Curl A = \tr(\D \axl A)\,\id- (\D \axl A)^T,\quad
\text{resp.} \quad
\D\axl A = \frac12 (\tr[\Curl A])\id - (\Curl A)^T.
\end{align}
Furthermore, the second derivatives of a skew-symmetric matrix field $A$ are given by linear combinations of the entries of the derivative of the deviatoric matrix $\Curl$:
\begin{equation}
 \D^2 A = L(\D\, \dev\Curl A)
\end{equation}
which was used in the proof of the trace-free case \cite{agn_lewintan2020KornLp_tracefree}.
The expression \eqref{eq:Nye}$_1$ admits  a counterpart on the group of orthogonal matrices $\operatorname{O}(3)$ and even in higher spatial dimensions, see e.g.~\cite{agn_munch2008curl}. Note in passing, that the representation of the kernel of $\sym \D u\equiv0$ can either be deduced from \eqref{eq:sec_der_id} or \eqref{eq:Nye} and yields the class $\mathrm{RM}$ of infinitesimal rigid motions
\begin{equation}
 \mathrm{RM}\coloneqq\{\widetilde{A}\,x+b \mid  \widetilde{A}\in\so(3), b\in\R^3\}.
\end{equation}
Indeed, assuming $\sym \D u\equiv0$
\begin{itemize}
 \item \eqref{eq:sec_der_id} implies that $\D^2 u \equiv 0$, so that $u$ has to be affine with $u\in\mathrm{RM}$, equivalently,
 \item since $\D u = A(x)$ with a skew-symmetric matrix field $A$, we obtain $\Curl A = \Curl \D u \equiv0 $, so that by \eqref{eq:Nye} we deduce $\D \axl A \equiv 0$ and hence $A\equiv\operatorname{const}$, i.e.,~again $u\in\mathrm{RM}$.
\end{itemize}

Summarizing, the following differential relations connecting higher order derivatives have been used in the distributional sense for
\medskip

\begin{tabular}{@{}l>{$}r <{$}@{ \ $ = $\ }>{$}l<{$}c@{}}
\textbullet\ \ classical Korn: &\D^2 u& L(\D\.\sym \D u) & cf.~\cite{Ciarlet2005korn,Geymonat86, DuvautLions72}\\[1.5ex]
\textbullet\ \ trace-free classical Korn: &\D \Delta u & L(\D^2\dev_n\sym\D u) & cf.~\cite{Dain2006tracefree,Schirra2012tracefreenD}\\[1.5ex]
\textbullet\ \ incompatible Korn: &\D A & L(\Curl A) & cf.~\cite{agn_lewintan2019KornLp,agn_lewintan2019KornLpN}\\[1.5ex]
\textbullet\ \ trace-free incompatible Korn: &\D^2 A &  L(\D\.\dev\Curl A) & cf.~\cite{agn_lewintan2020KornLp_tracefree}\\[0.5ex]
&\hspace{-4em}\D^2 (A+\zeta\cdot\id) & L(\D\Curl(A+\zeta\cdot\id))& cf.~\cite{agn_lewintan2020KornLp_tracefree,agn_lewintan2020KornLpN_tracefree}\\[0.5ex]
&\hspace{-4em}\D^3 (A+\zeta\cdot\id) & L(\D^2\dev\Curl(A+\zeta\cdot\id)) &cf.~\cite{agn_lewintan2020KornLp_tracefree}\\[1.5ex]
\textbullet\ \ symmetrized incompatible Korn: & \D^2 A & L(\D\.\sym\Curl A) & \multirow{2}{*}{present paper}\\[0.5ex]
\textbullet\ \ conformally invariant incompatible Korn:&\D^3 A & L(\D^2\dev\sym\Curl A) &
\end{tabular}\medskip

\noindent
denoting by $u$ a displacement vector field, by $A$ a skew-symmetric tensor field, by $\zeta$ a scalar field and by $L$ a corresponding linear operator with constant coefficients. Moreover, we have by  \cite{agn_lewintan2020KornLp_tracefree} for a general field $P$
\begin{align}
 \D\Curl P = L(\D \dev\Curl P)\,.
\end{align}

\subsection{Motivation for Korn type estimates for incompatible tensor fields}
Korn type inequalities for incompatible tensor fields originally motivated from infinitesimal gradient plasticity with plastic spin as well as in the linear relaxed micromorphic elasticity, see e.g. \cite{agn_ebobisse2016canonical, agn_ebobisse2018well,agn_neff2019static, agn_munch2018rotational,agn_neff2014unifying,agn_neff2010stable,agn_neff2009new, agn_ebobisse2010existence,agn_ebobisse2017fourth,RS2017viscoplast,agn_neff2009notes,agn_neff2015relaxed,agn_ghiba2017variant,Bardella_et_al2019gradientplast,Menzel2000gradientplast,Gurtin2000gradientplast,Gurtin2005gradientplast,Gurtin2005gradientplastBurger,Fleck1994gradientplast,Wulfinghoff2015gradientplast, agn_bauer2013dev} and the references contained therein.

\subsubsection{Application to the relaxed micromorphic model}
The relaxed micromorphic model is a novel micromorphic framework \cite{agn_neff2014unifying,agn_neff2019static} that allows e.g.~the description of microstructure-related frequency band-gaps \cite{agn_agostino2019dynamic} through a homogenized linear model.
The goal is to find the displacement $u\colon\Omega\subseteq\R^3\to\R^3$ and the non-symmetric micro-distortion field $P\colon\Omega\subseteq\R^3\to\R^{3\times3}$ minimizing  
\begin{equation*}\label{ener1}
\int_\Omega W\left(\D u,P,\Curl P \right) +\skalarProd{f}{u} \,dx\ ,\quad\text{such that }(u,P)\in H^1(\Omega)\times H(\Curl),
\end{equation*}\vspace{-1.2ex}
where the energy $W$ is defined as
\begin{equation}\label{eq:energy1}
\begin{split}
W =& \ \frac{1}{2}\,\skalarProd{\Ce\,\sym\left(\D u-P\right)}{\sym\left(\D u-P\right)}_{\R^{3\times3}}
 \ + \ \frac{1}{2}\,\skalarProd{\Cmicro\,\sym\,P}{\sym\,P}_{\R^{3\times3}}\\
 & + \ \frac{1}{2}\, \skalarProd{\Cc\,\skew\left(\D u-P\right)}{\skew\left(\D u-P\right)}_{\R^{3\times3}}
 \  + \ \frac{\mu\Lc^2}{2}\, \skalarProd{\mathbb{L} \Curl P}{\Curl P}_{\R^{3\times3}}\,.
\end{split}
 \end{equation}
Here, $\Ce,\,\Cmicro\colon\Sym(3)\to\Sym(3)$ are classical \nth{4} order elasticity tensors, $\Cc\colon\so(3)\to\so(3)$ is a \nth{4} order rotational coupling tensor, $\Lc\geq0$ is a characteristic length scale, $\mu$ is a typical effective shear modulus and $\mathbb{L}\colon\R^{3\times3}\to\R^{3\times3}$. The associated Euler-Lagrange equations read
\begin{equation}\label{eq:EL-first}
\begin{split}
 \Div\left[ \Ce\,\sym\left(\D u- P\right)+\Cc\,\skew\left(\D u- P\right)\right]&=f,\\
 \Ce\,\sym\left(\D u- P\right)+\Cc\,\skew\left(\D u- P\right)-\Cmicro\,\sym\,P-\mu\,L_{c}^{2}\,\Curl[\mathbb{L}\Curl\,P]&=0.
\end{split}
\end{equation}
The generalized moment balance \eqref{eq:EL-first}$_2$ can be seen as a tensorial Maxwell problem due to the $\Curl[\mathbb{L}\Curl P]$ operation, cf.~\cite{BSL}. The most general quadratic representation of the curvature energy is given by
\begin{equation}\label{eq:curvenergy}
 \skalarProd{\mathbb L \Curl P}{\Curl P}
\end{equation}
where $\mathbb L:\R^{3\times 3} \to\R^{3\times 3}$ is a non-standard fourth order tensor with $45$ independent entries acting on the non-symmetric second order tensor $\Curl P\in\R^{3\times 3}$. Since $\Curl P$ transforms as a second order tensor under rotations of the coordinate system, cf.~\cite{agn_munch2018rotational,agn_neff2009notes}, assuming a certain degree of anisotropy allows  one to reduce the complexity of $\mathbb L$. Notably, the most general isotropic quadratic expression of the curvature energy is given by
\begin{equation}\label{eq:curvenergy_alphas}
 \alpha_1\norm{\dev\sym\Curl P}^2+\alpha_2\norm{\skew \Curl P}^2+\frac{\alpha_3}{3}\tr^2(\Curl P),
\end{equation}
with three free parameters $\alpha_1,\alpha_2,\alpha_3\in\R^+$.
Here we have used the orthogonal decomposition of \corrected{$\R^{3\times 3}$} into orthogonal pieces, namely
\begin{subequations}\label{eq:decomposition_matrix}
 \begin{align}
  \corrected{\R^{3\times3}} &= [\sl(3)\cap\Sym(3)]\oplus \so(3)\oplus\R\cdot\id
\intertext{so that for any square matrix $X\in\R^{3\times3}$ we have}
 X &=\dev\sym X + \skew X + \frac13\tr(X)\cdot \id,
\end{align}
\end{subequations}
where $\dev\sym$, $\skew$, $\tr$ are orthogonal projections on the vector space $\sl(3)\cap\Sym(3)$ of symmetric trace free matrices, the space $\so(3)$ of skew-symmetric matrices, 
and the space $\R\cdot\id$ of spherical tensors, respectively.

In order to reduce complexity in the model one might be tempted to replace \eqref{eq:curvenergy} with
\begin{equation}\label{eq:curvenergy-sym}
 \skalarProd{\widehat{\mathbb L}\sym\Curl P}{\sym\Curl P}
\end{equation}
where $\widehat{\mathbb L}:\Sym(3)\to\Sym(3)$ is now a classical positive definite fourth order elasticity tensor, whose representation for all anisotropy classes is completely known. A weak formulation of the static problem
\begin{equation}\label{eq:EL-second}
\begin{split}
 \Div\left[ \Ce\,\sym\left(\D u- P\right)+\Cc\,\skew\left(\D u- P\right)\right]&=f,\\
 \Ce\,\sym\left(\D u- P\right)+\Cc\,\skew\left(\D u- P\right)-\Cmicro\,\sym\,P-\mu\,L_{c}^{2}\,\Curl[\widehat{\mathbb{L}}\sym\Curl\,P]&=0\,,
\end{split}
\end{equation}
is naturally formulated in the space $H(\sym\Curl;\Omega,\R^{3\times3})\coloneqq\{P\in L^2(\Omega,\R^{3\times3}) \mid  \sym\Curl P\in L^2(\Omega,\R^{3\times3})\}$ and our new result shows that this problem  is well-posed for a suitable prescription of tangential boundary data. Returning to \eqref{eq:curvenergy_alphas}, the problem may be even further ``relaxed'' by requiring only to control
\begin{equation}
 \norm{\dev\sym\Curl P}^2.
\end{equation}
In this case, the natural space to consider is the Hilbert space $H(\dev\sym\Curl;\Omega,\R^{3\times3})\coloneqq\{ P\in L^2(\Omega,\R^{3\times3}) \mid  \dev\sym\Curl P\in L^2(\Omega,\R^{3\times3})\}$ and our result \eqref{eq:Korn_Lp-devsymCurl} implies that the weak formulation is still well-posed.
\begin{remark}[Nothing new in plane strain]
Note, that due to the structure of the three-dimensional matrix $\Curl$ operator in  plain strain, i.e.,~assuming that
\begin{equation}
 \widehat{P}(x,y,z)=\begin{pmatrix}
                     \widehat{P}_{11}(x,y) & \widehat{P}_{12}(x,y) & 0\\ \widehat{P}_{21}(x,y) & \widehat{P}_{22}(x,y) & 0 \\ 0 & 0 & 0
                    \end{pmatrix},
                    \qquad
                    \Curl \widehat{P} = \begin{pmatrix}
                                         0 & 0& \ast \\ 0 & 0& \ast \\ 0 & 0 & 0
                                        \end{pmatrix}
\end{equation}
the operation $\sym$ or $\dev\sym$ is not leaving the classical $H(\Curl;\Omega,\R^{3\times3})$ space, since
\begin{equation}
 \norm{\dev_3\sym\Curl \widehat{P}}^2 = \norm{\sym\Curl \widehat P}^2 =\frac12 \norm{\Curl \widehat P}^2.
\end{equation}
Hence, new properties to be discovered are strictly three-dimensional in nature.
 \end{remark}

 \subsubsection{Cosserat model with weakest curvature energy -- conformally invariant curvature}
The use of the dislocation density tensor $\Curl P$ in the relaxed micromorphic model allows a smooth transition in the modeling to the classical linear Cosserat model. Indeed, letting formally $\Cmicro\to\infty$ in the relaxed micromorphic model \eqref{eq:energy1}, i.e.,~assuming  $P=A\in\so(3)$ is skew-symmetric, the (isotropic) elastic Cosserat free energy can be written as
\begin{equation}\label{eq:Cosserat}
\begin{split}
 \int_\Omega\mu\norm{\sym \D u}^2+&\mu_c\norm{\skew(\D u -A)}^2+\frac\lambda2\tr^2(\D u)\\
 &+ \alpha_1\norm{\dev\sym\Curl A}^2 +\alpha_2\norm{\skew \Curl A}^2+\frac{\alpha_3}{3}\tr^2(\Curl A)\,\intd{x}
 \quad\to\quad \min\,.
\end{split}
\end{equation}
In \cite{agn_neff2010stable,agn_neff2009new} it has been shown that choosing $\alpha_1>0$, $\alpha_2=\alpha_3=0$ is mandatory for offering \textit{bounded stiffness in bending and torsion} for arbitrary small specimen. This corresponds to the \textit{conformally invariant} curvature case
\begin{equation}\label{eq:conform}
 \norm{\dev\sym\Curl A}^2=\norm{\dev\sym\D \axl(A)}^2\,.
\end{equation}
Well-posedness results are then based on the trace-free Korn's inequality \cite{agn_neff2009new}.

Finally, letting the Cosserat couple modulus $\mu_c\to\infty$ in \eqref{eq:Cosserat}, one obtains the so-called \textit{modified} indeterminate couple stress model \cite{agn_neff2009subgrid,agn_ghiba2017variant,}
\begin{equation}
 \int_\Omega\mu\norm{\sym \D u}^2+\frac\lambda2\tr^2(\D u)+\alpha_1\underset{\text{conformally invariant curvature\footnotemark}}{\underset{=\norm{\sym\D\curl u}^2}{\underbrace{\norm{\dev\sym\D \curl u}^2}}} \intd{x}\quad \to  \quad \min\,.
\end{equation}
\footnotetext{$\tr(\sym \D \curl u)=\tr(\D\curl u)=\div\curl u\equiv0.$}
In \cite{agn_neff2009subgrid} this curvature energy has been obtained by a passage
from a discrete model to a continuum modeling, invoking  a ``micro-randomness'' assumption, which introduces an additional invariance property beyond isotropy.

\subsubsection{Application to gradient plasticity with plastic spin}
Experiments with differently sized specimens have revealed a pronounced size-effect in elasto-plastic transformations \cite{Fleck1994gradientplast,Gurtin2000gradientplast,Gurtin2005gradientplast,Gurtin2005gradientplastBurger} which cannot be described with classical phenomenological elasto-plasticity models. For the sake of simplicity we assume in the following the additive decomposition of the displacement gradient $\D u$ into non-symmetric elastic (recoverable) and non-symmetric plastic (permanent) distortions $e$ and $P$, respectively:
\begin{equation}\label{eq:decomp_grad}
 \D u = e+P, \qquad \varepsilon\coloneqq\sym \D u = \sym e + \sym P = \varepsilon_e+\varepsilon_p\,,
\end{equation}
under the side condition of plastic incompressibility ~ $\tr(P)=\tr(\varepsilon_p)\equiv 0$.~ A simplified framework for size-independent plasticity can be sketched, based on the introduction of the total free energy, which consist of elastic contributions and local hardening
\begin{align}
W(\D u,\,P) =& \int_\Omega\underset{\text{elastically stored energy}}{\underbrace{\skalarProd{\Ce\,\sym\left(\D u-P\right)}{\sym\left(\D u-P\right)}_{\R^{3\times3}}}}
 \ + \underset{\text{local hardening}}{\underbrace{\skalarProd{\mathbb{C}_{\text{hard}}\,\sym\,P}{\sym\,P}_{\R^{3\times3}}}}+\skalarProd{f}{u}\,\intd{x}\notag \\
 &= \int_\Omega \skalarProd{\Ce\,\left(\varepsilon-\varepsilon_p\right)}{\left(\varepsilon-\varepsilon_p\right)}_{\R^{3\times3}}+\skalarProd{\mathbb{C}_{\text{hard}}\,\varepsilon_p}{\varepsilon_p}+\skalarProd{f}{u}\,\intd{x},\label{eq:energy_plast}
\end{align}
where $\Ce$, $\mathbb{C}_{\text{hard}}$ are classical positive definite fourth order tensors acting on symmetric arguments. \corrected{We are working here in a phenomenological modeling context. A variational approach to single crystals with dislocations, different from out presented phenomenological viewpoint, has been explored in \cite{Scala2020properties,Scala2019variationalapproach} based on \cite{Mueller2008crystals}. Our term $\skalarProd{\mathbb{C}_{\text{hard}}\,\sym\,P}{\sym\,P}$ gives rise to the usual Prager-type  backstress term (linear kinematic hardening) which appears ubiquitous in the literature.} The appearance of $\sym(\D u -P)$ and $\sym P$ instead of $\D u -P$ and $P$ alone is dictated by linearized frame indifference of the model. Equilibrium of forces
\begin{equation}
 \Div \Ce \sym(\D u -P) = \Div \Ce (\varepsilon-\varepsilon_p)= f
\end{equation}
appears from variation of \eqref{eq:energy_plast} with respect to the displacement $u$. It remains to postulate a ``flow rule'', i.e.,~an evolution for the plastic variable $P$. This equation appears as gradient flow with respect to $P$ in the form
\begin{equation}\label{eq:gradientflowP}
 \dot{P}=\mathscr{F}(-\D_{P}W(\D u,\, P))= \mathscr{F}(\Ce\,\sym(\D u -P)-\mathbb{C}_{\text{hard}}\,\sym\,P)
\end{equation}
together with suitable initial conditions for $P$ and boundary conditions for $u$, where $\mathscr{F}:\R^{3\times3}\to\R^{3\times 3}$ is monotone, i.e.,~$\skalarProd{\mathscr{F}(X)-\mathscr{F}(Y)}{X-Y}_{\R^{3\times3}}\ge0$ and maps symmetric arguments to trace-free symmetric arguments, the increment $\dot{P}$ is determined to be trace-free symmetric and \eqref{eq:gradientflowP} can be therefore recast as
\begin{equation}
 \dot{\varepsilon}_p=\mathscr{F}(\Ce\,\left(\varepsilon-\varepsilon_p\right)-\mathbb{C}_{\text{hard}}\,\varepsilon_p)\,,\quad \tr(\varepsilon_p)=0\,.
\end{equation}
In order to extend the modeling framework to incorporate size-dependence, let us focus on the introduction of energetic length scales. In this case, one augments the total energy \eqref{eq:energy_plast}  by some terms involving space derivatives of the plastic distortion $P$ or the plastic strain $\varepsilon_p\coloneqq\sym P$, for simplicity
\begin{equation}\label{eq:normen_der_hilfsterme}
 \norm{\D P}^2 \quad \text{or}\quad \norm{\D \varepsilon_p}^2\,.
\end{equation}
Accordingly, based on \eqref{eq:normen_der_hilfsterme}$_1$, the evolution law \eqref{eq:gradientflowP} needs to be adapted to
\begin{equation}\label{eq:gradientflowP-adapted}
 \dot{P}=\mathscr{F}(-\D_{P}W(\D u,\, P,\,\D P))=\mathscr{F}(\Ce\,\sym(\D u -P)-\mathbb{C}_{\text{hard}}\,\sym\,P+\Delta P)
\end{equation}
and suitable boundary conditions for the plastic distortion $P$, here Dirichlet clamping $P_{|\partial\Omega}\equiv 0$, cf.~\cite{HanReddy2013Plasticity}. For initial condition $P(0)\in\Sym(3)$ \eqref{eq:gradientflowP-adapted} can again be recast into\footnote{Laplace component-wise and observe that $\Delta \varepsilon_p\in\Sym(3)$ for $\varepsilon_p\in\Sym(3)$.}
\begin{equation}
 \dot{\varepsilon}_p=\mathscr{F}(\Ce\,\left(\varepsilon-\varepsilon_p\right)-\mathbb{C}_{\text{hard}}\,\varepsilon_p+\Delta\varepsilon_p)\,, \quad \varepsilon_p{}_{|\raisebox{-3pt}{$\scriptstyle\partial\Omega$}}=0.
\end{equation}
Such a model is already able to predict that smaller samples are relatively stiffer. However, the simple gradient terms in \eqref{eq:normen_der_hilfsterme} lack a microscopical justification. \corrected{However, $\Delta P$ can be seen as regularization term as in \cite{Francfort1994effects}.} Since plasticity is mediated by dislocation movements it inspires that a physically more suitable description is given by considering the \textit{dislocation density tensor} $\Curl P$ and, in first approximation, a simple quadratic function thereof to replace \eqref{eq:normen_der_hilfsterme}. Hence, the total stored energy can be written
\begin{equation}\label{eq:energy-plast-curl}
 \int_{\Omega}\skalarProd{\Ce\,\sym\left(\D u-P\right)}{\sym\left(\D u-P\right)}+\skalarProd{\mathbb{C}_{\text{hard}}\,\sym\,P}{\sym\,P}+\skalarProd{\!\Curl P}{\Curl P}+\skalarProd{f}{u}\,\intd{x}\,.
\end{equation}
Since $\Curl$ is self-adjoint with suitable  \textit{tangential} boundary conditions  $P\times\nu_{|\partial\Omega}=0$, the evolution law turns into
\begin{equation}\label{eq:gradientflowP-mitCurl}
 \dot{P}=\mathscr{F}(\Ce\,\sym(\D u -P)-\mathbb{C}_{\text{hard}}\,\sym\,P-\Curl\Curl P)\,.
\end{equation}
Note that \eqref{eq:gradientflowP-mitCurl} is necessarily an evolution for a non-symmetric plastic distortion $P$ since the contribution $\Curl\Curl P$ does not have any symmetry properties. Such models are called gradient plasticity models with plastic spin or distortion gradient plasticity, cf.~\cite{agn_ebobisse2010existence,agn_ebobisse2016canonical,agn_ebobisse2017fourth,agn_ebobisse2018well,Menzel2000gradientplast,MSZ2014incompatible,agn_neff2009notes,RS2017viscoplast,Wulfinghoff2015gradientplast}.

A closer look at \eqref{eq:energy-plast-curl} reveals that the energy provides a uniform control over
\begin{equation}\label{eq:control_estimate}
 \skalarProd{\mathbb{C}_{\text{hard}}\,\sym\,P}{\sym P}+\norm{\Curl P}^2\ge c\,\big(\norm{\sym P}^2+\norm{\Curl P}^2\big)
\end{equation}
where $\mathbb{C}_{\text{hard}}$ is assumed to be positive definite. That the right-hand side of \eqref{eq:control_estimate} provides a norm on smooth, compactly supported functions $P\in C_0^\infty(\Omega,\R^{3\times 3})$ was first noted in \cite{agn_neff2009notes}. Indeed, $\sym P\equiv 0$ implies $P=A\in C_0^\infty(\Omega,\so(3))$, so that using Nye's formula \eqref{eq:Nye}$_2$ we deduce from $\Curl A\equiv 0$ that $\D\axl A\equiv0$. Hence, $\axl A$ is a constant vector field, $P=A$ is a constant skew-symmetric matrix field, and with the boundary condition we obtain $P\equiv0$.
Thus, controlling the plastic strain $\sym P$ in $L^2$ and the dislocation density tensor $\Curl P\in L^2$, together with suitable tangential boundary conditions\footnote{In the context of gradient plasticity, the boundary conditions $P\times \nu=0$ postulates ``no flux of the Burgers vector across the boundary surface" and is referred to as ``micro-hard", cf.~\cite{Gurtin2005gradientplastBurger}.} $P\times \nu=0$ on $\partial\Omega$ one controls the full plastic distortion $P\in L^2$. The result led to a sequel of well-posedness results in gradient plasticity with plastic spin in $H(\Curl)$, cf.~\cite{agn_ebobisse2010existence,agn_ebobisse2016canonical,agn_ebobisse2017fourth,agn_ebobisse2018well,agn_neff2009notes}.

However, choosing a simple quadratic energy in $\Curl P$ in \eqref{eq:energy-plast-curl} is, for many situations, not suitable. A major scientific question is therefore, how to replace $\norm{\Curl P}^2 $ in order to capture important physical features. Let us write $\mathcal{G}(\Curl P)$ for this extension. Experimental evidence suggests to use $\mathcal{G}$ with sub-quadratic  growth and where the behavior at zero can be differentiable or not. Indeed, in \cite{Wulfinghoff2015gradientplast} it is argued to consider the one-homogeneous expression $\mathcal{G}(\Curl P)=\norm{\Curl P}$ or $\mathcal{G}(\Curl P)=\norm{\Curl P}\cdot\log\norm{\Curl P}$, see also \cite{Bardella_et_al2019gradientplast,CO2005dislocation,Ohno2007higher}.

It is furthermore possible to extract some geometrical information from the dislocation density tensor (on the mesoscale). The indices $i$ and $j$ of $(\Curl P)_{ij}$ determine the orientation of the Burger's vector and the dislocation line, respectively. The diagonal components of $\Curl P$ describe \textit{screw dislocations} and the off-diagonal components describe \textit{edge dislocations}. \corrected{For an overview on dislocations in the framework of different types of generalized continua we refer the reader to \cite{Lazar2010disclocations,Acharya2003forces,Laza2006revisited} and the references therein.} Lazar \cite{Lazar2002elastoplast,Lazar:hal}, see also \cite{agn_neff2015relaxed}, has used the decomposition of the dislocation density tensor into $\SO(3)$-irreducible pieces
\begin{equation}\label{eq:density_tensor-decomp}
 \Curl P = \underset{\text{``tentor''}}{\underbrace{\dev\sym\Curl P}}+\underset{\text{``trator''}}{\underbrace{\skew \Curl P}} + \underset{\text{``axitor''}}{\underbrace{\frac13\tr(\Curl P)\cdot\id}}
\end{equation}
i.e.,~ the orthogonal decomposition \eqref{eq:decomposition_matrix}. So, the axitor corresponds to the sum of all possible screw dislocations, the trator to ``skew-symmetric" edge dislocations and the tentor describes a combination of ``symmetric" edge-dislocations and single screw-dislocations, cf.~\cite{agn_neff2015relaxed}. In addition, for compatible $X=\D u$ the decomposition \eqref{eq:decomposition_matrix} reads
\begin{equation}\label{eq:decompoDu}
 \D u =\underset{\text{shape-change}}{\underset{\text{anti-conformal part,}}{\underset{\text{shear}}{\underbrace{\dev\sym\D u }}}}+\underset{\text{no shape change}}{\underset{\text{conformal part,}}{\underbrace{\underset{\text{rotation}}{\underset{=\, \frac12\Anti(\curl u) }{\underbrace{\skew \D u}}} + \underset{\text{volumetric part}}{\frac13\,\underset{=\,\div u}{\underbrace{\tr(\D u)}}\cdot\id}}}}\,.
\end{equation}
The introduced nomenclature coming from the fact that $\dev\sym\D u=0$ implies that $u=\varphi_C$ is an infinitesimal conformal mapping, see \eqref{eq:infinitesimalconfis}.\footnote{Note, that besides the divergence and the curl of a vector field also the term $\dev\sym\D u$ has a physical interpretation, namely as the \textit{shear}, since for ``a cube of moving fluid, the shear [of the velocity of that fluid] represents the rate at which each side is deviating from a square, and the nature of that deviation", cf. \cite{Romano2012shear} where  the authors also make use of the ``natural" decomposition \eqref{eq:decompoDu}, cf. \cite[eq. (6)]{Romano2012shear}.}
\begin{figure}[h]\centering
\includegraphics{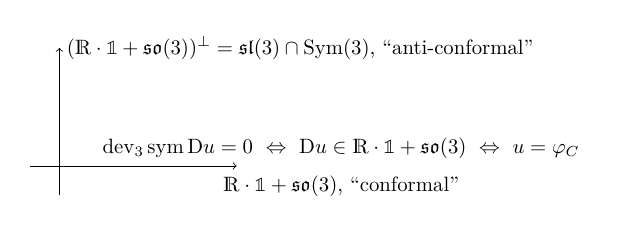}\vspace{-2ex}
\caption{Orthogonal decomposition and infinitesimal conformal mappings, see \eqref{eq:infinitesimalconfis}.}
 \end{figure}

For improved transparency in the physical modeling, we may now assume an additively decomposed ansatz for $\mathcal{G}$:
\begin{equation}
 \begin{split}
  \mathcal{G}(\Curl P)&= \mathcal{G}(\dev\sym\Curl P+\skew\Curl P + \frac13\tr(\Curl P)\cdot\id)\\
  &=\mathcal{G}_1(\dev\sym\Curl P)+\mathcal{G}_2(\skew\Curl P) + \mathcal{G}_3(\tr(\Curl P)\cdot\id)\\
  & = \alpha_1\norm{\dev\sym\Curl P}^{q_1}+\alpha_2\norm{\skew\Curl P}^{q_2}+\frac{\alpha_3}{3}\abs{\tr(\Curl P)}^{q_3}
 \end{split}
\end{equation}
where in the last step we considered a simple isotropic example. Our novel result \eqref{eq:Korn_Lp-devsymCurl} shows that under the conditions $\alpha_1>0$, $\alpha_2=\alpha_3\ge 0$, $q_1>1$  one can control the plastic distortion $P$ in $L^{q_1}$.

\subsubsection{Incompatible linear elasticity}
Instead of the classical linear elasticity problem
\begin{equation}\label{eq:classlinelast}
 \int_{\Omega} \skalarProd{\Ce\sym\D u}{\sym \D u}+\skalarProd{\underset{f}{\underbrace{\Div G}}}{u}\,\intd{x} \quad\to\quad \min \qquad u\in H^1(\Omega,\R^3),\footnotemark
\end{equation}
\footnotetext{$\div(G^T u)=\skalarProd{\Div G}{u}_{\R^3}+\skalarProd{G}{\D u}_{\R^{3\times3}}$, to fix notation.}
that is
\begin{equation}\label{eq:el-to-recover}
 \Div \Ce\sym\D u= \Div G, \qquad u_{|\Gamma}=\widetilde{u}\in H^1(\Omega,\R^3),
\end{equation}
we may consider the corresponding incompatible nonlinear elasticity problem
\begin{subequations}\label{eq:incomplinelast}
\begin{equation}
 \int_{\Omega} \skalarProd{\Ce\sym e}{\sym e}-\skalarProd{G}{e} + \mu\,L_c^r\,\norm{\dev\sym\Curl e}^r\,\intd{x} \ \to\ \min \quad e\in W^{1,\,2,\,r}(\dev\sym\Curl;\Omega,\R^{3\times3}),
\end{equation}
in other words the strong form of the second order Euler-Lagrange equations formally reads
\begin{equation}
\Ce\sym e +r\,\mu\, L_c^r\Curl\left(\frac{\dev\sym\Curl e}{\norm{\dev\sym\Curl e}^{r-2}}\right)=G, \qquad \tr(\Ce\sym e) = \tr(G),\footnotemark
\end{equation}
under the (consistent) symmetrized tangential boundary condition
\begin{equation}
 \dev\sym (e\times \nu) _{|\Gamma} = \dev\sym(\D \widetilde{u}\times \nu)
\end{equation}
\end{subequations}
\footnotetext{Since $\tr(\Curl S)= 0$ for $S\in\Sym(3)$.}
where $1<r\le 2$, $\widetilde{u}\in H^1(\Omega,\R^3)$ and $G\in H^1_0(\Omega,\R^{3\times 3})$ are prescribed and $\nu$ is the outward unit vector field to $\partial\Omega$. According to our Theorem \ref{thm:main2sym_new}, the solution to \eqref{eq:incomplinelast} is unique with respect to the non-symmetric elastic distortion $e\in W^{1,\,2,\,r}(\dev\sym\Curl;\Omega,\R^{3\times3})$. Note that for $\frac65\le r \le 2$ it holds  $W^{1,\,2,\,r}(\dev\sym\Curl;\Omega,\R^{3\times3})\supsetneqq W^{1,\,2}(\Omega,\R^{3\times3})$. The formulation \eqref{eq:incomplinelast}   might therefore be useful in problems with fracture. Furthermore, the stored energy in \eqref{eq:incomplinelast} is always bounded above by the corresponding energy in \eqref{eq:classlinelast} for the compatible case. Replacing $e=\D u$ and taking the divergence, recovers \eqref{eq:el-to-recover}$_1$.

In the same spirit, in \cite{Amstutz2020existence} the authors have considered the \emph{non-variational} second-order problem
\begin{equation}
 \Ce\sym e + 2 \mu L_c^2\, \inc(\sym e) = G,\qquad   \inc(\sym e)\,\nu=0 \quad \text{on $\partial\Omega$},
\end{equation}
which also looks for a ``relaxation" of linear elasticity and determines a unique solution $\sym e \in L^2(\Omega)$, $\inc(\sym e)\in L^2(\Omega)$, where $\inc P \coloneqq\Curl[(\Curl P)^T]$ and further properties of the $\inc$ operator will be discussed below. Replacing $\sym e=\sym \D u$ and taking the divergence would also recover \eqref{eq:el-to-recover}$_1$.

\section{Notations and technical preliminaries}
Let $n\geq2$. For vectors $a,b\in\R^n$ we consider the scalar product  $\skalarProd{a}{b}\coloneqq\sum_{i=1}^n a_i\,b_i \in \R$, the (squared) norm  $\norm{a}^2\coloneqq\skalarProd{a}{a}$ and the dyadic product  $a\otimes b \coloneqq \left(a_i\,b_j\right)_{i,j=1,\ldots,n}\in \R^{n\times n}$. Similarly, the scalar product for matrices $P,Q\in\R^{n\times n}$ is given by $\skalarProd{P}{Q} \coloneqq\sum_{i,j=1}^n P_{ij}\,Q_{ij} \in \R$ and the (squared) Frobenius-norm by $\norm{P}^2\coloneqq\skalarProd{P}{P}$.
Moreover, $P^T\coloneqq (P_{ji})_{i,j=1,\ldots,n}$ stands for the transposition of the matrix $P=(P_{ij})_{i,j=1,\ldots,n}$. We make use of the orthogonal decomposition of the latter into the symmetric part $\sym P \coloneqq \frac12\left(P+P^T\right)$ and the skew-symmetric part $\skew P \coloneqq \frac12\left(P-P^T\right)$. We denote by $\sl(n)\coloneqq\{X\in\R^{n\times n}\mid \tr(X)=0\}$ the Lie-algebra of trace-free matrices, with $\Sym(n)\coloneqq\{X\in\R^{n\times n}\mid  X^T=X\}$ and by $\so(n)\coloneqq \{A\in\R^{n\times n}\mid A^T = -A\}$ the Lie-Algebra of skew-symmetric matrices. For the identity matrix we write $\id$, so that the trace of a squared matrix $P$ is \ $\tr P \coloneqq \skalarProd{P}{\id}$. The deviatoric (trace-free) part of $P$ is given by $\dev_n P\coloneqq P -\frac1n\tr(P)\,\id$ and in three dimensions we will suppress its index, i.e.,~we write $\dev$ instead of $\dev_3$.

By $\mathscr{D}'(\Omega)$ we denote the space of distributions on a bounded Lipschitz domain $\Omega\subset\R^n$ and by $W^{-k,\,p}(\Omega)$ the dual space of $W^{k,\,p'}_0(\Omega)$, where $p'=\frac{p}{p-1}$ is the dual H\"older exponent to $p$.

\subsection{The three-dimensional case}\label{sec:3D}
In $\R^3$ we further make use of the vector product $\times:\R^3\times\R^3\to\R^3$.
For a fixed vector $a\in\R^3$ the cross product $a\times.$ is linear in
the second component, so that there exists a unique matrix $\Anti(a)$ such that

\begin{equation}\label{eq:def_Anti}
 a \times b \eqqcolon \Anti(a)\, b \qquad \forall\, b\in\R^3.
\end{equation}
Direct computations show that for  $a=(a_1,a_2,a_3)^T$ the matrix $\Anti(a)$ is of the form
\begin{equation}\label{eq:Anti_Komp}
\Anti(a)=
 \begin{pmatrix}
   0 & -a_3 & a_2\\ a_3 & 0 & -a_1 \\ -a_2 & a_1 &0
  \end{pmatrix},
\end{equation}
so that with $\Anti:\R^3\to\so(3)$ we have a canonical identification of $\R^3$ with the vector space of skew-symmetric matrices $\so(3)$. This algebraic approach to the cross product facilitates some of the traditional proofs of vector algebra, cf.~\cite{Trenkler2001vectprod,Room1952vectprod,LiuTrenkler2008,GTT1999vectprodinC}. Indeed,  also the notations $T_a$, $W(a)$ or even $[a]_\times$ are used for $\Anti(a)$, but the latter emphasizes that we deal with a skew-symmetric matrix. Furthermore, the vector product can be written as
\begin{equation}\label{eq:vectCross}
 a \times b = \Anti(a)\, b = -b\times a = -\Anti(b)\,a=(a^T \Anti(b))^T \qquad \forall\, a,b\in\R^3.
\end{equation}
The inverse of $\Anti$ will be called $\axl:\so(3)\to\R^3$ and it associates to a skew-symmetric  matrix $A\in\so(3)$ the (axial) vector ~ $\axl A\coloneqq (-A_{23},A_{13},-A_{12})^T$, so that
\begin{equation}\label{eq:axlDef}
 A\, b = \axl(A)\times b \qquad \forall\, b\in\R^3.
\end{equation}
The identification of the vector product with a suitable matrix product allows us to generalize the vector product in $\R^3$ to a vector product of a vector $b\in\R^3$ and a matrix $P\in\R^{3\times 3}$ from the left and from the right:
\begin{equation}\label{eq:matrixCross}
 b \times P \coloneqq \Anti(b)\, P \qquad \text{and}\qquad P\times b \coloneqq P\,\Anti(b),
\end{equation}
Thus, $b \times P $   is given by a column-wise vector multiplication
\begin{align}
\begin{split}
 b\times P &= \Anti(b)\,\big(P\,e_1 \mid  P\,e_2 \mid  P\,e_3\big) = \big(\Anti(b)\,P\,e_1 \mid \Anti(b)\,P\,e_2 \mid  \Anti(b)\,P\,e_3\big)\\
 & = \big(b\times(P\,e_1) \mid  b\times (P\,e_2) \mid  b\times (P\,e_3)\big),
 \end{split}
\end{align}
whereas in $P\times b$ we have a row-wise vector multiplication
\begin{align}\label{eq:matrixCurl}
\begin{split}
 P\times b &=
 \begin{pmatrix}
  (P^T\,e_1)^T \\ (P^T\,e_2)^T \\(P^T\,e_3)^T
 \end{pmatrix}\Anti(b) =
 \begin{pmatrix}
  (P^T\,e_1)^T\Anti(b) \\ (P^T\,e_2)^T\Anti(b) \\(P^T\,e_3)^T\Anti(b)
 \end{pmatrix}
\overset{\eqref{eq:vectCross}}{=}
\begin{pmatrix}
 ((P^Te_1)\times b)^T \\((P^Te_3)\times b)^T \\ ((P^Te_3)\times b)^T
\end{pmatrix}\,.
 \end{split}
\end{align}
For the identity matrix we obtain
\begin{equation}\label{eq:idcrossvec}
 \id\times b = \id\,\Anti(b) =\Anti(b) \qquad \forall\ b\in\R^3.
\end{equation}
Formally, Nye's formula \eqref{eq:Nye} is a consequence of the following algebraic identity\footnote{This algebraic relation is already contained in \cite[p. 691 (ii)]{Room1952vectprod}.}:
\begin{equation}\label{eq:prod_id}
\Anti(a)\Anti(b)=(\Anti(a))\times b = b \otimes a -\skalarProd{b}{a}\, \id = b\otimes a - \tr(b\otimes a)\,\id  \qquad \forall\,  a,b\in\R^3
\end{equation}
and the second identity \eqref{eq:Nye}$_1$ comes from the converse expression
\begin{equation}
 b\otimes a  = (\Anti(a))\times b + \skalarProd{b}{a}\,\id = (\Anti(a))\times b - \frac12\tr((\Anti(a))\times b)\,\id
\end{equation}
where we have used
\begin{equation}\label{eq:trAntiatimesb}
 \tr((\Anti(a))\times b) ) = \tr(\Anti(a) \Anti(b)) = \skalarProd{\!\Anti(a) \Anti(b)}{\id}=-\skalarProd{\!\Anti(a)}{\Anti(b)} \overset{\eqref{eq:Anti_Komp}}{=}-2 \skalarProd{a}{b}.
\end{equation}
In addition, for all $b\in\R^3$ we obtain
\begin{align}
 \Anti(b)\Anti(b)&\overset{\eqref{eq:matrixCross}}{=} \Anti(b)\times b \overset{\eqref{eq:prod_id}}{=} b\otimes b -\norm{b}^2\id\,, \label{eq:anti^2}
 \shortintertext{so that}
 \Anti(b)\Anti(b)\Anti(b)&\underset{\eqref{eq:matrixCross}}{\overset{\eqref{eq:anti^2}}{=}}\big(b\otimes b -\norm{b}^2\id \big)\times b = -\norm{b}^2\Anti(b)\,.\label{eq:anti^3}
\end{align}
Consequently, for a symmetric matrix $S$ we have \ $\tr(S\times b)= 0$ for any $b\in\R^3$, since
\begin{align}\label{eq:trStimesb}
\tr(S\times b)=\skalarProd{S\times b}{\id}\overset{\eqref{eq:matrixCross}}{=}\skalarProd{S\,\Anti(b)}{\id}=\skalarProd{\Anti(b)}{S^T}\overset{S\in\Sym(3)}{=}0\,,
\end{align}
and similarly
\begin{align}
 \tr((S\times b)\times b)^T\times b) &= \skalarProd{(S\times b)\times b)^T\times b}{\id}
 \overset{\eqref{eq:matrixCross}}{=}\skalarProd{\Anti(b)\Anti(b)\, S \Anti(b)}{\id}\notag \\ & = -\skalarProd{S}{\Anti(b)\Anti(b)\Anti(b)}\overset{\eqref{eq:anti^3}}{=} \norm{b}^2\skalarProd{S}{\Anti(b)} \overset{S\in\Sym(3)}{=}0\,.
 \label{eq:trdreimalcross}
\end{align}
Furthermore, we can consider the vector multiplication on both sides:
\begin{equation}
 b\times P \times b = \Anti(b)\, P \, \Anti(b)\,.
\end{equation}
However, from the viewpoint of application it is more convenient to look at
\begin{equation}\label{eq:alg_inc}
 (P\times b)^T\times b = (P\,\Anti (b))^T\Anti(b) = - \Anti(b)\,P^T\Anti(b) = - b\times P^T\times b.
\end{equation}
In particular, for a skew-symmetric matrix $A\in\so(3)$ and a symmetric matrix $S\in\Sym(3)$ we have
\begin{equation}
 (A\times b)^T\times b = b\times A \times b \qquad \text{and}\qquad  (S\times b)^T\times b = - b\times S \times b\,.
\end{equation}

\begin{observation}\label{obs:2}
  For $a,b\in\R^3$ we have
\begin{equation}\label{eq:normUNGLdevsym}
\frac12\norm{a}^2\norm{b}^2\le \norm{\dev\sym(\Anti(a)\times b)}^2\le \frac23\norm{a}^2\norm{b}^2.
\end{equation}
\end{observation}
\begin{proof}
  Considering the ~ $\dev\sym$ ~ parts on both sides of \eqref{eq:prod_id} we obtain
 \begin{equation}\label{eq:iii}
  \dev\sym(\Anti(a)\times b) \overset{\eqref{eq:prod_id}}{=} \dev\sym(b\otimes a) = \sym(a\otimes b)-\frac13\tr(b\otimes a)\cdot\id=\sym(a\otimes b)-\frac13\skalarProd{a}{b}\cdot\id.
 \end{equation}
 Since,
 \begin{equation}\label{eq:norm_sym}
  \norm{\sym(a\otimes b)}^2=\frac14\norm{a\otimes b+ b\otimes a}^2=\frac12\norm{a\otimes b}^2+\frac12\skalarProd{a\otimes b}{b\otimes a} = \frac12\norm{a}^2\norm{b}^2+\frac12\skalarProd{a}{b}^2,
 \end{equation}
taking the squared norm on both sides of \eqref{eq:iii} we obtain
\begin{align}
 \norm{\dev\sym(\Anti(a)\times b)}^2\ \ &\overset{\mathclap{\eqref{eq:iii}}}{=}\ \ \norm{\sym(a\otimes b)}^2 +\frac19\skalarProd{a}{b}^2\norm{\id}^2-\frac23\skalarProd{a}{b}\skalarProd{\sym(a\otimes b)}{\id} \\
 &\overset{\mathclap{\eqref{eq:norm_sym}}}{=}\  \frac12\norm{a}^2\norm{b}^2+\frac12\skalarProd{a}{b}^2+\frac13\skalarProd{a}{b}^2-\frac23\skalarProd{a}{b}^2
  = \ \frac12\norm{a}^2\norm{b}^2+\frac16\skalarProd{a}{b}^2. \notag
\end{align}
The right hand side is bounded from above by $\frac23\norm{a}^2\norm{b}^2$ and from below by $\frac12\norm{a}^2\norm{b}^2$. These bounds are sharp if $a$ is parallel to $b$ and if $a$ is perpendicular to $b$, respectively.
\end{proof}
\begin{remark}
 Due to the identification of skew-symmetric matrices with vectors in $\R^3$ the relation \eqref{eq:normUNGLdevsym} reads also
 \begin{equation}\label{eq:normUNGLdevsymA}
  \frac18\norm{A}^2\norm{\widetilde{A}}^2\le\norm{\dev\sym(A\,\widetilde{A})}^2\le\frac16\norm{A}^2\norm{\widetilde{A}}^2 \qquad \forall A, \widetilde{A}\in\so(3).
 \end{equation}
Indeed, setting $a\coloneqq \axl A$ and $\widetilde{a}\coloneqq \axl \widetilde{A}$, we have $A\,\widetilde{A}= \Anti(a)\times\widetilde{a}$, so that the estimate follows from \eqref{eq:normUNGLdevsym} in combination with the identities
 $\norm{A}^2=2\norm{a}^2$ and $\norm{\widetilde{A}}^2=2\norm{\widetilde{a}}^2$. The bounds in \eqref{eq:normUNGLdevsymA} are sharp. The upper bound is achieved for $A=\widetilde{A}$ and the lower bound is achieved, e.g., for $A=\Anti(e_1)=e_3\otimes e_2 - e_2 \otimes e_3$ and $\widetilde{A}=\Anti(e_2)=e_1\otimes e_3 -e_3 \otimes e_1$.
\end{remark}
In \cite{agn_lewintan2020KornLp_tracefree} we used moreover for $P\in\R^{3\times 3}$ and $b\in\R^3$ the relation
\begin{equation}
 \dev(P\times b) = 0 \quad \Leftrightarrow \quad P\times b = 0.
\end{equation}
Here,  we use a similar equivalence.
\begin{observation}\label{obs:3}
 For $P\in\R^{3\times 3}$ and $b\in\R^3$ we have
 \begin{equation}\label{eq:aus_obs3}
  \dev\sym(P\times b)= 0 \ \Leftrightarrow \ \sym(P\times b) = 0 \,.
 \end{equation}
\end{observation}
\begin{remark}
 Surely, \eqref{eq:aus_obs3} is not equivalent to the condition $P\times b=0$, cf.~the example in \eqref{eq:idcrossvec}.
\end{remark}
\begin{proof}[Proof of Observation \ref{obs:3}]
 We decompose $P$ into its symmetric and skew-symmetric part, i.e.,
 \[
  P= S +\Anti(a), \qquad \text{for some }S\in\Sym(3), \ a\in\R^3.
 \]
and obtain
\begin{align}
 \dev\sym(P\times b) &= \ \sym(P\times b) -\frac13\tr(P\times b)\,\id \overset{\eqref{eq:trStimesb}}{=} \sym(P\times b) -\frac13\tr(\Anti(a)\times b)\,\id \notag\\ &
 \overset{\mathclap{\eqref{eq:trAntiatimesb}}}{=}\ \sym(P\times b) +\frac23\skalarProd{a}{b}\,\id.\label{eq:firststepdevsymPtimesb}
\end{align}
Moreover, for any matrix $P\in\R^{3\times3}$ it holds
\begin{equation}\label{eq:crossprod_scalarprod}
 (P\times b)\,b \overset{\eqref{eq:matrixCross}}{=}P\,\Anti(b)\, b \overset{\eqref{eq:vectCross}}{=}P\, (b\times b) =0\,.
\end{equation}
Thus, we obtain
\begin{equation}\label{eq:hilfs_schritt2a}
\skalarProd{b}{\dev\sym(P\times b)\,b} \overset{\eqref{eq:firststepdevsymPtimesb}}{=}\skalarProd{b}{\big(\sym(P\times b) +\frac23\skalarProd{a}{b}\id\big)\, b} \overset{\eqref{eq:crossprod_scalarprod}}{=} \frac23\skalarProd{a}{b}\,\norm{b}^2,
\end{equation}
and the statement follows from the identity
\begin{align}
 \norm{b}^2\, \sym(P\times b) \ \ &\overset{\mathclap{\eqref{eq:firststepdevsymPtimesb}}}{=}\ \ \norm{b}^2 \, \dev\sym(P\times b)-\frac23\norm{b}^2\skalarProd{a}{b}\,\id \notag \\
 &\overset{\mathclap{\eqref{eq:hilfs_schritt2a}}}{=}\ \ \norm{b}^2 \, \dev\sym(P\times b) - \skalarProd{b}{\dev\sym(P\times b)\,b}\,\id\,.\label{eq:PtbedevsymPtb}
 \end{align}
Applying the Cauchy-Bunyakovsky-Schwarz inequality on the right hand side of \eqref{eq:PtbedevsymPtb} we obtain
\begin{equation}\label{eq:normequiv-top}
 \norm{\dev\sym(P\times b)}\le\norm{\sym(P\times b)} \overset{\eqref{eq:PtbedevsymPtb}}{\le}\left(1+\sqrt{3}\right)\norm{\dev\sym(P\times b)}\,.
\end{equation}
 \end{proof}

\subsection{Considerations from vector calculus}\label{sec:Formalism}
The vector differential operator $\nabla$ behaves algebraically like a vector, so that, formally, the derivative, the divergence and the curl of a vector field $a\in \mathscr{D}'(\Omega,\R^3)$ can be expressed as
\begin{equation}
 \D a = a\otimes \nabla=(\nabla\otimes a)^T, \quad \div a= \skalarProd{a}{\nabla}=\skalarProd{\D a}{\id}=\tr(\D a) \quad \text{and} \quad  \curl a = a\times (-\nabla)=\nabla\times a.
\end{equation}
Formally, the Laplace operator behaves like a scalar with $\Delta=\norm{\nabla}^2$.

More generally, we can use multilinear expressions to define differential operators as follows. Let $V$ and $W$ be finite-dimensional vectorspaces and let $\mathrm{Lin}(V,W)$  denote the
space of linear maps from $V$ to $W$. Let $M: \R^d \times \ldots \times \R^d \to \mathrm{Lin}(V,W)$ be a multilinear map
and denote by  $M_{i_1 \ldots i_r}\coloneqq M(e_{i_1}, \ldots, e_{i_r})$  the coeffiencts of $M$ 
with respect to the standard basis $e_1, \ldots, e_d$ of $\R^d$. 
We define a differential operator $D_M$ by
$$ D_M = \sum_{i_1, \ldots, i_r}  M_{i_1 \ldots i_r} \partial_{i_1}  \ldots \partial_{i_r}$$
where each index $i_j$ runs from $1$ to $d$.
Let $\Omega \in \R^d$ be open. Then $D_M$ maps a distribution $f  \in\mathscr{D}'(\Omega,V)$
to a distribution $D_M f \in\mathscr{D}'(\Omega,W)$. 
The following simple observation allows us to transfer algebraic identities into identities of vector calculus. 
We have
\begin{equation} \label{eq:calculus_identity}
\forall\ b \in \R^d \quad M(b, \ldots, b) = 0 \quad \Longleftrightarrow
\quad
\forall\ f \in \mathscr{D}'(\Omega,V)  \quad D_M f = 0.
\end{equation}
Indeed, since $\partial_i \partial_j = \partial_j \partial_i$ in the sense of distributions, both assertions
are equivalent to the statement that the symmetrized coefficients of $M$ vanish.
For example, the algebraic identity $\tr ((a \times b) \otimes b) = 0$ for all $a, b \in \R^3$ translates into the identity
$\div \curl f = 0$ for all $f \in \mathscr{D}'(\Omega,\R^3)$. Since $M$ is multilinear we often use the notation
\begin{equation}
M(\nabla, \ldots, \nabla) \coloneqq D_M
\end{equation}
With this notation,  \eqref{eq:calculus_identity} asserts that we  can formally compute as if $\nabla$ was 
a vector in $\R^d$. 

Of special interest is the operator curl  and its row-wise extension to a matrix-valued operator
$\Curl$.  Thus, formally,
\begin{equation}\label{eq:Curl_def}
 \Curl P \coloneqq P\times (-\nabla) = -P\,\Anti(\nabla)
\end{equation}
for $P\in\mathscr{D}'(\Omega,\R^{3\times3})$ where the vector product acts row-wise, cf. \eqref{eq:matrixCurl}. Surely, $\nabla\times P$ or $-\nabla\times P$ would also be interesting candidates to consider, but, among them, only the matrix $\Curl$ from \eqref{eq:Curl_def} kills the derivative of a general vector field $a\in \mathscr{D}'(\Omega,\R^3)$, i.e.,~ $\Curl \D a \equiv0$. For symmetric tensor fields $S\in\mathscr{D}'(\Omega,\Sym(3))$ we obtain by \eqref{eq:trStimesb}
\begin{equation}\label{eq:trCurlS}
 \tr (\Curl S) \equiv 0\,.
\end{equation}
Moreover, it holds for $\zeta\in\mathscr{D}'(\Omega,\R)$ and $a\in\mathscr{D}'(\Omega,\R^3)$
\begin{equation}
 \Curl(\zeta\cdot\id) = -\Anti(\nabla \zeta) \quad\text{and}\quad \Curl \Anti(a) = \div a \cdot\id-(\D a)^T.
\end{equation}
Note in passing that in three dimensions the matrix $\Curl$ returns again a square matrix. Furthermore, we make use of the incompatibility operator
\begin{equation}
 \inc P \coloneqq   \Curl ([\Curl P]^T)=(P\times \nabla)^T\times \nabla\overset{\eqref{eq:alg_inc}}{=}-\nabla\times P^T\times \nabla = -\Anti(\nabla)\,P^T\Anti(\nabla).
\end{equation}
The last expression shows, in particular, that the incompatibility operator preserves symmetry:
\begin{align}
 (\inc P)^T=\inc(P^T), \quad \inc \sym P = \sym \inc P \quad \text{and}\quad \inc\skew P = \skew\inc P\,.
 \label{eq:incsymmetry}
\end{align}
Moreover, the incompatibility operator annihilates the symmetric displacement gradient since
\begin{align}
 \inc(\sym\D a) &= -\nabla\times\sym\D a\times \nabla = -\frac12\nabla\times(\nabla\otimes a + a \otimes \nabla)\times\nabla \notag \\
 & =-\frac12[(\underset{= 0}{\underbrace{\nabla\times\nabla}})\otimes a\times \nabla+\nabla\times a\otimes(\underset{= 0}{\underbrace{\nabla\times\nabla}})] \equiv 0\,.\label{eq:propertyinckill}
 \end{align}
 Note that this formal calculation was already carried out in Lagally's monograph on vector calculus from 1928 \cite[Ziff. 191]{Lagally1928}. The action of the incompatibility operator on spherical tensors and antisymmetric tensors  is given by
\begin{equation}\label{eq:inc_ids}
 \inc(\zeta\cdot\id) = \Delta \zeta\cdot\id -\D^2\zeta\in \Sym(3) \quad \text{and}\quad \inc(\Anti(a))= -\Anti(\nabla\div a)\in\so(3)\,,
\end{equation}
respectively. For symmetric tensor fields $S$ we obtain, formally by \eqref{eq:trdreimalcross}, again
\begin{equation}\label{eq:trincCurlS}
 \tr(\inc\Curl S) \equiv0.
\end{equation}

\begin{remark}
 The incompatibility operator $\inc$ occurs in infinitesimal strain dislocation models, e.g., in the modeling of dislocated crystals  or in the modeling of  elastic materials with dislocations, since the strain cannot be a symmetrized gradient of a vector field as soon as dislocations are present and the notion of incompatibility is at the basis of a new paradigm to describe the inelastic effects, cf. \cite{agn_ebobisse2017fourth,Amstutz2017incompatibility,Lazar2010disclocations,Amstutz2019incompatibility,Amstutz2016analysis,Maggiani2015incompatible}. Furthermore, the equation ~ $\inc\sym e\equiv 0$ ~ is equivalent to the \textit{Saint-Venant compatibility condition(s)}\footnote{Those compatibility conditions can be found in the third appendix \S 32 p. 597 et seq. of the third edition of the lecture notes  \textit{R\'{e}sistance des corps solides} given by Navier and extended with several notes and appendices by Barr\'{e} de Saint-Venant and published as \textit{R\'{e}sum\'{e} des Le\c{c}ons donn\'{e}es \`{a} l'\'{E}cole des Ponts et Chauss\'{e}es sur l'Application de la M\'{e}canique}, vol. I, Paris, 1864. Their coordinate-free version are contained in Lagally's monograph on vector calculus from 1928 \cite[Ziff. 191]{Lagally1928}.} defining the relation between the  displacement vector field $u$  and the symmetric strain $\sym e$, more precisely:
\begin{equation}
\inc \sym e \equiv 0 \quad \Leftrightarrow \quad \sym e = \sym \D u
\end{equation}
in simply connected domains, cf. \cite{Amrouche2006inc,Maggiani2015incompatible}. For investigations over multiply connected domains see e.g.~\cite{Ting1977StVenant,Geymonat2005Venant}.
\end{remark}

\subsection{Linear combinations of higher derivatives}
Our analysis relies on a number of apparently hitherto unnoticed identities which arise from the interaction of the matrix $\Curl$ operator with the algebraic splitting  \eqref{eq:decomposition_matrix}:
$$ X =\dev\sym X + \skew X + \frac13\tr(X)\cdot \id.$$
In particular, we have the following identities.

\begin{lemma} \label{lem:lin_combi_general} Let $A\in\mathscr{D}'(\Omega,\so(3))$. Then
\begin{thmenum}
 \item the entries of $\D^2 A$ are linear combinations of the entries of $\D\, \sym\Curl A$. \label{lin_combi_a}
 \item the entries of $\D^3 A$ are linear combinations of the entries of $\D^2 \dev\sym\Curl A$. \label{lin_combi_b}
 \end{thmenum}
\end{lemma}

\begin{proof}
 By Nye's formula \eqref{eq:Nye}$_1$ we have
 \begin{equation}\label{eq:Nye_sym}
  \sym \Curl A = \tr(\D\axl A )\, \id - \sym(\D\axl A).
 \end{equation}
Taking the trace on both sides we obtain \ $\tr(\sym\Curl A) = 2\, \tr(\D\axl A)$ and inserting this identity into \eqref{eq:Nye_sym} we get
\begin{equation}\label{eq:Nye_sym_A}
 \sym (\D\axl A) = \frac12\tr(\sym\Curl A)\corrected{\id}-  \sym \Curl A.
\end{equation}
Moreover, by the relation \eqref{eq:sec_der_id} used for the proof of the classical Korn's inequality, we obtain
\begin{equation}
 \D^2 \axl A \overset{\eqref{eq:sec_der_id}}{=} L(\D\.\sym \D\axl A) \overset{\eqref{eq:Nye_sym_A}}{=} L_1(\D\.\sym \Curl A).
\end{equation}
In other words, the entries of $\D^2 A $ are linear combinations of the entries of $\D \sym \Curl A$ which establishes part \ref{lin_combi_a}.

To prove  \ref{lin_combi_b} we make use of the incompatibility operator $\inc$, since it
 kills the symmetric displacement gradient, cf. \eqref{eq:propertyinckill}. Consider now the deviatoric part on both sides of \eqref{eq:Nye_sym}:
\begin{equation}\label{eq:Nye_devsym}
 \dev \sym \Curl A = \frac13\tr(\D\axl A) \,\id - \sym(\D\axl A)\,.
\end{equation}
Applying $\inc$ on both sides, we obtain in view of \eqref{eq:propertyinckill} and \eqref{eq:inc_ids}:
\begin{equation}
 3\,\inc\dev\sym\Curl A = \Delta \tr(\D\axl A)\cdot \id - \D^2\tr(\D\axl A)
\end{equation}
or, equivalently,
\begin{equation}\label{eq:Nye_dd_trA}
 \D^2 \tr(\D\axl A)=\frac32\tr(\inc\dev\sym\Curl A)\cdot\id - 3\,\inc\dev\sym\Curl A
 = L_2(\D^2 \dev\sym\Curl A),
\end{equation}
where we have used that the entries of $\inc B$ are, of course, linear combinations of the entries of $\D^2B$, so that by \eqref{eq:Nye_devsym} we have
\begin{equation}\label{eq:Nye_dd_sym_A}
 \D^2 \sym (\D\axl A) = L_3(\D^2 \dev\sym\Curl A).
\end{equation}
The conclusion of part \ref{lin_combi_b} then follows using the relation \eqref{eq:sec_der_id}:
\begin{equation*}
 \D^3 \axl A \overset{\eqref{eq:sec_der_id}}{=} L(\D^2 \sym \D\axl A) \overset{\eqref{eq:Nye_dd_sym_A}}{=} L_4(\D^2 \dev\sym\Curl A).\qedhere
\end{equation*}
\end{proof}
The algebraic considerations above provide information on higher derivatives of $P$ in negative Sobolev spaces. To obtain $L^p$-estimates for $P$ we use the following deep result.

\begin{theorem}[Lions lemma and Ne\v{c}as estimate]  \label{th:necas_estimate}
Let $\Omega\subset\R^n$ be a bounded  Lipschitz  domain.  Let $m \in \Z$ and $p \in (1, \infty)$.
Then $f \in  \mathscr{D}'(\Omega,\R^d)$ and $\D f \in W^{m-1,\,p}(\Omega,\R^{d\times n})$ imply
$f \in W^{m,\,p}(\Omega,\R^d)$.
Moreover,
\begin{equation}  \label{eq:necas_m_p}
 \norm{ f}_{W^{m,\,p}(\Omega,\R^d)} \le c\,\left(\norm{ f}_{W^{m-1,\,p}(\Omega,\R^d)} + \norm{ \D f }_{W^{m-1,\,p}(\Omega,\R^{d\times n})}\right),
\end{equation}
with a constant $c=c(m,p,n,d,\Omega)>0$.
\end{theorem}
For a proof we refer to \cite[Proposition 2.10 and Theorem 2.3]{AG90} and \cite{BS90}. However, for our discussions the heart of the matter is the estimate  \eqref{eq:necas_m_p}, see  Ne\v{c}as \cite[Th\'{e}or\`{e}me 1]{Ne66}. The case $m=0$ is already contained in \cite{Cattabriga61}; for an alternative proof, see \cite[Lemma 11.4.1]{MW2012bvpStokes} and \cite[Chapter IV]{BF2013} as well as \cite{Bramble2003} and \cite{Villani-tangentBCs}.   For further historical remarks, see the discussions in \cite{Ciarlet2010,ACM2015} and the references contained therein. 

Since we only have information on higher order derivatives of $P$ we will use the following 
consequence of  Theorem~\ref{th:necas_estimate}.

\begin{corollary}\label{cor:LionsNecas_k}
 Let $\Omega\subset\R^n$ be a bounded  Lipschitz  domain, $m \in \Z$ and $p \in (1, \infty)$. Denote by  $\D^k f$ the collection of all distributional derivatives of order $k$. Then $f \in  \mathscr{D}'(\Omega,\R^d)$ and $\D^k f \in W^{m-k,\,p}(\Omega,\R^{d\times n^k})$ imply
$f \in W^{m,\,p}(\Omega,\R^d)$.
Moreover, 
\begin{equation}  \label{eq:necask_m_p}
 \norm{ f}_{W^{m,\,p}(\Omega,\R^d)} \le c\,\left(\norm{ f}_{W^{m-1,\,p}(\Omega,\R^d)} +
 \norm{ \D^k f }_{W^{m-k,\,p}(\Omega,\R^{d\times n^k})}\right),
\end{equation}
with a constant $c=c(m,p,n,d,\Omega)>0$.
\end{corollary}
\begin{proof}
The assertion $f \in W^{m,\,p}(\Omega,\R^d)$ and the estimate \eqref{eq:necask_m_p} follow by inductive application of Theorem \ref{th:necas_estimate} to $\D^l f$ with $l=k-1,k-2,\ldots,0$.
\end{proof}

\begin{lemma}\label{lem:kern}
 \label{obs:kernels_k} Let $A\in L^p(\Omega,\so(3))$. Then
\begin{thmenum}
 \item\label{kernel_a} $\sym\Curl A \equiv 0$ in the distributional sense if and only if $A=\Anti(\widetilde{A}\,x+b)$ almost everywhere in $\Omega$,
 \item\label{kernel_b} $\dev\sym\Curl A \equiv 0$ in the distributional sense if and only if $A= \Anti\big(\widetilde{A}\,x+b+\beta\, x +\skalarProd{d}{x}\,x -\frac12d\norm{x}^2\big)$ almost everywhere in $\Omega$
\end{thmenum}
with constant $\widetilde{A}\in\so(3)$, $b,d\in\R^3$ and $\beta,\gamma\in\R$.
\end{lemma}
\begin{remark}
 It is seen already from the calculations of the kernels that there can not be corresponding Korn type inequalities in terms of \ $\norm{\dev\sym P} +\norm{\sym\Curl P}$ \ or $\norm{\dev\sym P} +\norm{\dev\sym\Curl P}$. \ The kernels would be infinite-dimensional, since all restricting information on $\zeta$ would get lost. Indeed, we have
 \[
  \sym\Curl(A+\zeta\cdot\id)\overset{\eqref{eq:Nye}}{=} \tr(\D\axl A)\id - \sym(\D\axl A)
 \]
so that \ $\sym\Curl(A+\zeta\cdot\id)\equiv0 $ \ or \ $\dev\sym\Curl(A+\zeta\cdot\id)\equiv0 $ \ allow $\zeta$ to be arbitrary.
\end{remark}

\begin{remark}
 Solutions of Lemma \ref{lem:kern} have already been partially indicated in the literature, cf. \cite{Reshetnyak1970, agn_bauer2013dev}. We include their full deduction here for the convenience of the reader.
 \end{remark}
\begin{proof}[Proof of Lemma \ref{lem:kern}]
The ``if"-parts follow from a direct calculation using  Nye's formula \eqref{eq:Nye}:
\begin{thmenum}
 \item $\Curl(\Anti(\widetilde{A}\,x+b)) = \widetilde{A} $,
 \item $\D(\Anti\big(\widetilde{A}\,x+b+\beta\, x +\skalarProd{d}{x}\,x -\frac12d\norm{x}^2\big)) = \widetilde{A}+\beta\,\id +\skalarProd{d}{x}\id+x\otimes d-d\otimes x = (\beta+\skalarProd{d}{x})\id+\widetilde{A}+\Anti(d\times x)$, hence,  $\Curl(\Anti\big(\widetilde{A}\,x+b+\beta\, x +\skalarProd{d}{x}\,x -\frac12d\norm{x}^2\big)) = 2(\beta+\skalarProd{d}{x})\id +\widetilde{A}+\Anti(d\times x)$,
 \end{thmenum}
Now, we will focus on the ``only if''-directions.

By \eqref{eq:Nye_sym_A} the condition \ $\sym\Curl A \equiv 0$ \ implies \ $\sym (\D\axl A) \equiv0$, \ so that  the usual calculation for Korn's inequality, cf. \eqref{eq:sec_der_id}, gives that \ $\D\axl A$ \ must be a constant skew-symmetric matrix. Thus,  $$A=\Anti(\widetilde{A}\,x+b)$$ for some $\widetilde{A}\in\so(3)$ and $b\in\R^3$, which establishes \ref{kernel_a}.\medskip

Considering now \ $\dev\sym\Curl A \equiv 0$ \ we obtain by \eqref{eq:Nye_dd_trA} that $\D^2 \tr(\D\axl A)\equiv 0$. Hence,
\begin{equation}\label{eq:tr_axl_a_1}
 \frac13\tr(\D\axl A) = \beta +\skalarProd{d}{x}
\end{equation}
for some $d\in\R^3$ and $\beta\in\R$. Define $\bar{a}$ by
\begin{equation}\label{eq:_hilfs_a_bar_1}
\bar{a}(x)= \beta\,x +\skalarProd{d}{x}x-\frac12d\,\norm{x}^2.
 \end{equation}
Then
$$
\D\bar{a}= \beta\,\id+\skalarProd{d}{x}\,\id+x\otimes d- d\otimes x = (\beta+\skalarProd{d}{x})\,\id + \Anti(d\times x)
$$
and
\begin{equation}\label{eq:abl_bar_a}
 \sym \D \bar{a}= (\beta+\skalarProd{d}{x})\,\id.
\end{equation}
Thus, by \eqref{eq:Nye_devsym} we have
\[
\sym(\D\,(\axl A -\bar{a})) \overset{\eqref{eq:abl_bar_a}}{=} \sym(\D\axl A)-(\beta+\skalarProd{d}{x})\,\id  \overset{\eqref{eq:Nye_devsym}}{=} \frac13\tr(\D\axl A)\corrected{\id} -(\beta+\skalarProd{d}{x})\,\id \overset{\eqref{eq:tr_axl_a_1}}{=}  0\,.
\]
Again, \eqref{eq:sec_der_id} gives that \ $\D\,(\axl A-\bar{a})$ \ must be a constant skew-symmetric matrix and we have
\[
\axl A = \widetilde{A}x + b + \bar{a},
\]
for some $\widetilde{A}\in\so(3)$ and $b\in\R^3$, and statement \ref{kernel_b} follows from the representation \eqref{eq:_hilfs_a_bar_1}.
\end{proof}

\begin{remark}\label{rem:confi}
 The conclusion of \ref{kernel_b} also follows directly from Nye's formula and is connected to infinitesimal conformal maps. Indeed, we have
 \begin{equation}
 \dev\sym\Curl\Anti(a) \overset{\eqref{eq:Nye}_1}{=}-\dev\sym \D a,
 \end{equation}
 so that
 \begin{equation}
  \dev\sym\Curl\Anti(a) \equiv 0 \quad \Leftrightarrow\quad \dev\sym \D a \equiv 0 \quad \Leftrightarrow \quad a=\varphi_C,
 \end{equation}
denoting by $\varphi_C$ an infinitesimal conformal map, so that the expression in \ref{kernel_b} of Lemma \ref{lem:kern} follows from the expression for infinitesimal conformal maps \eqref{eq:infinitesimalconfis}.
\end{remark}

\section{New incompatible Korn type inequalities}

\begin{lemma}\label{lem:basic3}
  Let $\Omega \subset \R^3$ be a bounded Lipschitz domain, $1<p<\infty$ and $P\in\mathscr{D}'(\Omega,\R^{3\times3})$. Then either of the conditions
  \begin{thmenum}
   \item $\sym P\in L^p(\Omega,\R^{3\times3})$ and $\sym\Curl P \in W^{-1,\,p}(\Omega,\R^{3\times3})$,\label{condition_a}
   \item $\sym P\in L^p(\Omega,\R^{3\times3})$ and $\dev\sym\Curl P \in W^{-1,\,p}(\Omega,\R^{3\times3})$,\label{condition_b}
   \end{thmenum}
implies $P\in L^p(\Omega,\R^{3\times3})$. Moreover, we have the estimates
\begin{subequations}\label{eq:ausmLemma}
\begin{align}
  \norm{P}_{L^p(\Omega,\R^{3\times3})} &\leq c\, \Big(\norm{\skew P}_{W^{-1,\,p}(\Omega,\R^{3\times3})}\notag\\
  &\hspace{5em}+ \norm{\sym P}_{L^p(\Omega,\R^{3\times3})}+ \norm{\sym \Curl P }_{W^{-1,\,p}(\Omega,\R^{3\times3})}\Big),\label{eq:from_a}\\
  \norm{P}_{L^p(\Omega,\R^{3\times3})} &\leq c\, \Big(\norm{\skew P}_{W^{-1,\,p}(\Omega,\R^{3\times3})}\notag\\
  &\hspace{5em}+ \norm{\sym P}_{L^p(\Omega,\R^{3\times3})}+ \norm{\dev \sym \Curl P }_{W^{-1,\,p}(\Omega,\R^{3\times3})}\Big),\label{eq:from_b}
  \end{align}
  \end{subequations}
  always with a constant $c=c(p,\Omega)>0$.
\end{lemma}
\begin{remark}
Clearly,  condition  \ref{condition_b} is weaker than condition  \ref{condition_a} and \eqref{eq:from_b} implies \eqref{eq:from_a}. Furthermore, \eqref{eq:from_b} implies the estimate
  \begin{equation}
 \begin{split}
   \norm{P}_{L^p(\Omega,\R^{3\times3})} &\leq c\, \Big(\norm{\skew P}_{W^{-1,\,p}(\Omega,\R^{3\times3})}\\
  &\hspace{5em}+ \norm{\sym P}_{L^p(\Omega,\R^{3\times3})}+ \norm{\Curl P }_{W^{-1,\,p}(\Omega,\R^{3\times3})}\Big)
  \end{split}
 \end{equation}
 in   \cite[Lemma 3.1]{agn_lewintan2019KornLp} as well as the estimate
 \begin{equation}
 \begin{split}
   \norm{P}_{L^p(\Omega,\R^{3\times3})} &\leq c\, \Big(\norm{\skew P}_{W^{-1,\,p}(\Omega,\R^{3\times3})}\\
  &\hspace{5em}+ \norm{\sym P}_{L^p(\Omega,\R^{3\times3})}+ \norm{\dev\Curl P }_{W^{-1,\,p}(\Omega,\R^{3\times3})}\Big)
  \end{split}
 \end{equation}
 in \cite[Lemma 3.6]{agn_lewintan2020KornLp_tracefree},
 but not the estimate
 \begin{equation}
 \begin{split}
   \norm{P}_{L^p(\Omega,\R^{3\times3})} &\leq c\, \Big(\norm{\skew P+\textstyle\frac13\tr P\cdot \id}_{W^{-1,\,p}(\Omega,\R^{3\times3})}\\
  &\hspace{5em}+ \norm{\dev\sym P}_{L^p(\Omega,\R^{3\times3})}+ \norm{\dev\Curl P }_{W^{-1,\,p}(\Omega,\R^{3\times3})}\Big)
  \end{split}
 \end{equation}
in  \cite[Lemma 3.6]{agn_lewintan2020KornLp_tracefree} which uses $\dev \sym P$ rather than $\sym P$ on the right hand side. The point is, that we cannot improve  \eqref{eq:ausmLemma} to an estimate which involves $\dev \sym P$ instead of $\sym P$ on the right hand side, cf.~Remark \ref{Rem:keineNorm}.
 \end{remark}

\begin{proof}[Proof of Lemma \ref{lem:basic3}]
By the previous remark it suffices to establish the assertion $P \in L^p(\Omega; \R^{3 \times 3})$ under condition \ref{condition_b} and to prove the estimate  \eqref{eq:from_b}.

We will follow the same line of reasoning as in the proof of \cite[Lemma 3.1]{agn_lewintan2019KornLp} and start by considering the orthogonal decomposition
  \begin{align*}
  P&=\sym P+\skew P.
\end{align*}
To deduce $\skew P \in L^p(\Omega,\R^{3\times3})$ under assumption \ref{condition_b} we consider
\begin{align}\label{eq:D2curlskew}
\norm{    \D^2\dev\sym\Curl \skew P}_{ W^{-3,\,p}(\Omega,\R^{3\times3^3})    }    
& \leq c\, \norm{\dev\sym\Curl (P-\sym P)}_{ W^{-1,\,p}(\Omega,\R^{3\times3})}\notag\\
&\le c\,(\norm{\dev\sym\Curl P}_{ W^{-1,\,p}(\Omega,\R^{3\times3})} + \norm{\Curl\sym P}_{ W^{-1,\,p}(\Omega,\R^{3\times3})}) \notag \\
&\leq c \,(\norm{\dev\sym\Curl P}_{ W^{-1,\,p}(\Omega,\R^{3\times3})} + \norm{\sym P}_{ L^p(\Omega,\R^{3\times3})}).
\end{align}
Hence, \  $\D^2\dev\sym\Curl \skew P\in W^{-3,\,p}(\Omega,\R^{3\times3^3})$  
and it follows from Lemma \ref{lem:lin_combi_general} \ref{lin_combi_b} that 
\begin{equation}
 \D^3 \skew P\in W^{-3,\,p}(\Omega,\R^{3\times3^4}).
 \end{equation}
Now, we apply Corollary \ref{cor:LionsNecas_k} to $\skew P$ and we deduce that  $\skew P\in L^p(\Omega,\R^{3\times3})$ and
\begin{align}\label{eq:Norm_b}
 \norm{\skew P}_{L^p(\Omega,\R^{3\times 3})}&\le\quad  c\, (\norm{\skew P}_{ W^{-1,\,p}(\Omega,\R^{3\times3})} + 
 \norm{\D^3\skew P}_{ W^{-3,\,p}(\Omega,\R^{3\times3^4})}) \notag \\
 &\overset{\mathclap{\text{Lem. \ref{lem:lin_combi_general} \ref{lin_combi_b}}}}{\leq}\quad c\, ( \norm{\skew P}_{ W^{-1,\,p}(\Omega,\R^{3\times 3})} +  \norm{\D^2\dev\sym\Curl \skew P}_{ W^{-3,\,p}(\Omega,\R^{3\times3^3})} \notag \\
 &\overset{\mathclap{\eqref{eq:D2curlskew}}}{\leq}\quad c\, ( \norm{\skew P}_{ W^{-1,\,p}(\Omega,\R^{3\times3})}  \\
 &\hspace{5em}+\norm{\sym P}_{ L^p(\Omega,\R^{3\times3})}+\norm{\dev\sym\Curl P}_{ W^{-1,\,p}(\Omega,\R^{3\times3})}). \notag \qedhere
\end{align}
\end{proof}

The rigidity results follow by eliminating the corresponding first term on the right-hand side of \eqref{eq:ausmLemma}.
\begin{theorem}\label{thm:main1_k_new}
 Let $\Omega \subset \R^3$ be a bounded Lipschitz domain and $1<p<\infty$. There exists a constant $c=c(p,\Omega)>0$ such that
  for all $P\in  L^{p}(\Omega,\R^{3\times3})$
 \begin{subequations}
\begin{thmenum}
 \item \label{rigid_a}
 \begin{equation}\label{eq:rigid_Korn_Lp_sC}
   \inf_{T\in K_{S,SC}}\norm{P-T}_{L^p(\Omega,\R^{3\times3})}\leq c\,\left(\norm{\sym P }_{L^p(\Omega,\R^{3\times3})}+ \norm{\sym \Curl P }_{W^{-1,\,p}(\Omega,\R^{3\times3})}\right)
 \end{equation}
  \item \label{rigid_b}
 \begin{equation}\label{eq:rigid_Korn_Lp_dsC}
   \inf_{T\in K_{S,dSC}}\norm{P-T}_{L^p(\Omega,\R^{3\times3})}\leq c\,\left(\norm{\sym P }_{L^p(\Omega,\R^{3\times3})}+ \norm{\dev\sym \Curl P }_{W^{-1,\,p}(\Omega,\R^{3\times3})}\right)
   \end{equation}
 \end{thmenum}
 \end{subequations}
 where the kernels are given by
 \begin{subequations}
 \begin{align}
   K_{S,SC} &= \{T:\Omega\to\R^{3\times3} \mid  T(x)=\Anti(\widetilde{A}\,x+b), \ \widetilde{A}\in\so(3), b\in\R^3\}, \label{eq:kernel_SSC} \\
    K_{S,dSC} &= \{T:\Omega\to\R^{3\times3} \mid   T(x)=\Anti\big(\widetilde{A}\,x+\beta\, x+b +\skalarProd{d}{x}\,x -\frac12d\norm{x}^2\big), \notag \\ &\hspace{20em} \widetilde{A}\in\so(3), b,d\in\R^3, \beta\in\R\}, \label{eq:kernel_SdSC}
 \end{align}
 \end{subequations}
\end{theorem}

\begin{remark}
Setting $\widetilde{a}=\axl(\widetilde{A})$ we have for the linear functions in the kernels
\begin{equation}
 \Anti(\widetilde{A}\,x)= \Anti(\Anti(\widetilde{a})\,x)=x\otimes \widetilde{a}-\widetilde{a}\otimes x = 2\,\skew(x\otimes \widetilde{a})
\end{equation}
so that $K_{S,SC}$ can be alternatively written as
\begin{equation}
 K_{S,SC} = \{T:\Omega\to\R^{3\times3} \mid  T(x)=\skew(x\otimes \widetilde{a})+\Anti(b), \ \widetilde{a}, b\in\R^3\} \tag{\ref{eq:kernel_SSC}'}.
\end{equation}
Furthermore,  the elements of $K_{S,dSC}$ are connected to infinitesimal conformal mappings $\varphi_C$ via
\begin{equation}
 K_{S,SC} = \{T:\Omega\to\R^{3\times3} \mid T(x)=\Anti(\varphi_C(x)), \ \text{ with } \dev\sym\D\varphi_C \equiv0\} \tag{\ref{eq:kernel_SdSC}'}
\end{equation}
cf.~Remark \ref{rem:confi}.
\end{remark}

\begin{proof}[Proof of Theorem \ref{thm:main1_k_new}]
We first prove the formulae for the kernels $K_{S,Sc}$ and $K_{S,dSC}$. If 
\begin{equation}
  P \in K_{S,SC}\coloneqq\{P\in  L^p(\Omega,\R^{3\times3}) \mid  \sym P=0 \text{ a.e.~and }\sym\Curl P=0 \text{ in the dist.~sense}\},
 \end{equation}
 then $P = \skew P$ and $\sym \Curl \skew P = 0$. 
 Thus  \eqref{eq:kernel_SSC} follows by virtue of Lemma \ref{lem:kern} \ref{kernel_a}.
 Similarly, the formula  \eqref{eq:kernel_SdSC} follows from  Lemma \ref{lem:kern} \ref{kernel_b}.

The estimates  \eqref{eq:rigid_Korn_Lp_sC} and \eqref{eq:rigid_Korn_Lp_dsC}
now follow from  Lemma~\ref{lem:basic3}, the fact that the kernels are finite-dimensional and the compactness
of the embedding $L^p(\Omega) \hookrightarrow
W^{-1,p}(\Omega)$, see, for example, the proofs \cite{agn_lewintan2019KornLp,agn_lewintan2019KornLpN,agn_lewintan2020KornLp_tracefree} or
\cite[Theorem 6.15-3]{Ciarlet2013FAbook} for similar reasoning. 
For the convenience of the reader we provide the details for the argument for the estimate
 \eqref{eq:rigid_Korn_Lp_sC}. The proof of \eqref{eq:rigid_Korn_Lp_dsC} is analogous.
By $e_1,\ldots,e_M$ we denote a basis of $K_{S,SC}$, and by $\ell_1, \ldots, \ell_M$
we denote the corresponding dual basis of linear functionals on $K_{S,SC}$
which is characterized by the conditions
\begin{equation}\label{eq:basisKern}
 \ell_\alpha(e_j)\coloneqq \delta_{\alpha j}.
\end{equation}
Then, the Hahn-Banach theorem in a normed vector space (see e.g. \cite[Theorem 5.9-1]{Ciarlet2013FAbook}), allows us to extend $\ell_\alpha$ to continuous linear forms - again denoted by $\ell_\alpha$ - on the Banach space $L^p(\Omega,\R^{3\times3})$, $1\le\alpha\le M$. Note that
\begin{equation}\label{eq:verschwindet}
 \forall\ T \in K_{S,SC}  \qquad T = 0 \quad \Leftrightarrow \quad \ell_\alpha(T)= 0 \ \forall\ \alpha\in\{1,\ldots,M\}.
\end{equation}
We claim that
\begin{equation}\label{eq:hilfsungl_sym}
 \norm{P}_{L^p(\Omega,\R^{3\times 3})}\leq c\,\left(\norm{\sym P }_{L^p(\Omega,\R^{3\times 3})}+ \norm{\sym \Curl P }_{W^{-1,\,p}(\Omega,\R^{3\times 3})}+\sum_{\alpha=1}^M\abs{\ell_\alpha(P)} \right).
\end{equation}
Indeed, if this inequality is false, there exists a sequence
$P_k\in L^p(\Omega,\R^{3\times3})$ with the properties
$$
 \norm{P_k}_{L^p(\Omega,\R^{3\times3})}=1 \quad \text{and}\quad \left(\norm{\sym P_k}_{L^p(\Omega,\R^{3\times3})}+\norm{\sym\Curl P_k}_{ W^{-1,p}(\Omega,\R^{3\times3})}+\sum_{\alpha=1}^M\abs{\ell_\alpha(P_k)}\right)< \frac1k.
$$
Hence, (for a subsequence) $P_k\rightharpoonup P^*$ in $L^p(\Omega,\R^{3\times3})$ and we have $\sym P^*\equiv 0$ and $\sym\Curl P^* \equiv 0$ in the distributional sense but also $\ell_\alpha(P^*)=0$ for all $\alpha=1,\ldots,M$, so that $P^*\equiv0$.
 Since the embedding $L^p(\Omega,\R^{3\times 3})\hookrightarrow   W^{-1,\,p}(\Omega,\R^{3\times 3})$ is compact we get $\skew P_k\to \skew P^*\equiv0$ in $W^{-1,\,p}(\Omega,\R^{3\times 3})$.  Thus, $P_k\to 0$ in $W^{-1,\,p}(\Omega,\R^{3\times 3})$ and this yields  to a contradiction with  \eqref{eq:from_a}. Hence \eqref{eq:hilfsungl_sym} holds.
 
Now consider the projection $\pi_{a}:L^p(\Omega,\R^{3\times 3}) \to K_{S,SC}$ given by
\begin{equation}\label{eq:projection}
 \pi_{a}(P)\coloneqq \sum_{j=1}^M \ell_j(P)\, e_j.
\end{equation}
We obtain $\ell_\alpha(P-\pi_{a}(P)) \overset{\eqref{eq:basisKern}}{=} 0$ for all $1\le\alpha\le M$, so that \eqref{eq:rigid_Korn_Lp_sC} follows after inserting $P-\pi_{a}(P)$ into \eqref{eq:hilfsungl_sym}:
\begin{equation*}
\begin{split}
 \inf_{T\in K_{S,SC}}\norm{P-T}_{L^p(\Omega,\R^{3\times3})}&\leq\norm{P-\pi_{a}(P)}_{L^p(\Omega,\R^{3\times 3})}\\
 &\leq c\,\left(\norm{\sym P }_{L^p(\Omega,\R^{3\times 3})}+ \norm{\sym \Curl P }_{W^{-1,\,p}(\Omega,\R^{3\times 3})} \right).\qedhere
 \end{split}
\end{equation*}
\end{proof}

Finally, we show that the   estimates  in Theorem~\ref{thm:main1_k_new} can be improved to  estimates for $P$ itself, and not just for  $P - T$, if we impose a natural boundary condition which annihilates the relevant  kernels.

We focus on the improvement of \eqref{eq:rigid_Korn_Lp_dsC} because this already implies the improved estimate for \eqref{eq:rigid_Korn_Lp_sC}. For a weak definition of boundary values of certain linear combinations of $P$ it is not sufficient to assume only 
$\dev \sym \Curl P \in W^{-1,\,p}(\Omega, \R^{3 \times 3})$. Indeed, this condition is satisfied for  every $P \in L^p(\Omega, \R^{3 \times 3})$.
We thus consider, for $p \in (1, \infty)$ and $r \in [1, \infty)$ the spaces 
\begin{equation}
  W^{1,\,p, \, r}(\dev\sym\Curl; \Omega,\R^{3\times3}) \coloneqq \{P\in L^p(\Omega,\R^{3\times3}) \mid  \dev\sym\Curl P \in L^r(\Omega,\R^{3\times3})\}.
  \label{eq:def_dSC}
\end{equation}
Equipped with the norm
\begin{align}
  \| P \|_{p,r, dSC}  \coloneqq  \|P \|_{L^p(\Omega, \R^{3 \times 3})} +   \| \dev \sym \Curl P \|_{L^r(\Omega, \R^{3 \times 3})}
\end{align}
this space becomes a Banach space.
In terms of scaling the natural relation between $p$ and $r$ is $p= r^*$ where $r^*$ is the Sobolev exponent of $r$. 
To properly treat the borderline case $p = 1^* = \frac32$  we make the following assumptions
\begin{equation} \label{eq:assumption_r}
r \in [1, \infty), \qquad  \frac1r  \le   \frac1p + \frac13, \qquad r >1  \quad \text{if $p = \frac{3}{2}$}.
\end{equation}
We assume that $\Omega \subset \R^3$ is a bounded domain with Lipschitz boundary. 
To define boundary conditions for certain linear combinations of the components of $P$
 in the distributional sense, we first recall that for $q \in (1, \infty)$ 
the space $C^1(\overline \Omega)$ is dense in $W^{1,\,q}(\Omega)$ and  there exists
a linear bounded and surjective  trace operator $\mathrm{Tr}: W^{1,\,q}(\Omega) \to W^{1- 1/q,\,q}(\partial \Omega)$
which is uniquely characterized by the condition $\mathrm{Tr} f = f|_{\partial \Omega}$ for all $f \in C^1(\overline \Omega)$.
Moreover there exists  a linear,  bounded extension operator $\mathrm{E}: W^{1- 1/q,\, q}(\partial \Omega) \to W^{1,\,q}(\Omega)$ with
$\mathrm{Tr} \circ \mathrm{E} = \mathrm{id}.$
If follows from the divergence theorem and the density of $C^1(\overline \Omega)$ that for all $i=1, 2, 3$
\begin{equation}  \label{eq:by_parts}
\int_\Omega \partial_i f  \,  \intd{x} = \int_{\partial \Omega}  \mathrm{Tr} f   \,  \nu_i  \, \mathrm{d}\mathcal{H}^2      \quad \forall f \in W^{1,\,q}(\Omega)
\end{equation}
where $\nu$ denotes the outer normal of $\partial \Omega$ (which exist $\mathcal H^2$ a.e.\  on $\partial \Omega$) and $\mathcal H^2$ the two-dimensional Hausdorff measure.
For  $p \in (1, \infty)$ we denote by $p'$ the dual exponent given by $\frac1p + \frac1{p'} = 1$. The dual of the space
$W^{1-1/p',\, p'}(\partial \Omega) = W^{1/p,\, p'}(\partial \Omega)$ is denoted by $W^{-1/p,\, p}(\partial \Omega)$.
In order to introduce a weak definition of the boundary values of $\dev \sym [P \times \nu]$ we 
assume that $r$ satisfies  \eqref{eq:assumption_r}. Then we can 
define a bounded map $S:  W^{1,\,p,\,r}(\dev \sym \Curl; \Omega; \R^{3 \times 3}) \to W^{-1/p,\, p}(\partial \Omega; \R^{3 \times 3})$ by
\begin{align}\label{eq:partIntsym_new}
   \skalarProd{SP}{Q}_{\partial \Omega} &\coloneqq   \int_{\Omega}   \skalarProd{\dev \sym\Curl P}{\mathrm{E}Q}  -   \skalarProd{P}{\Curl\dev \sym \mathrm{E}Q}\, \intd{x}
   \quad \forall\ Q\in  W^{1/p,\,p'}(\partial\Omega,\R^{3\times 3}).
\end{align}
Here the extension operator is applied componentwise. 
If  $a \in C^1(\overline \Omega; \R^3)$ and $b \in W^{1,\,q}(\Omega, \R^3)$ then
it follows from   \eqref{eq:by_parts}
that 
\begin{equation}
   \int_{\Omega}   \skalarProd{\curl a}{b}   -  \skalarProd{a}{\curl b} \, \intd{x}   =  \int_{\partial \Omega}     \skalarProd{a \times (-\nu)}{ \mathrm{Tr}\, b}  
  \, \mathrm{d}\mathcal H^2\,.
\end{equation}
Using this identity, the fact that $\Curl$ acts row-wise  and \eqref{eq:matrixCurl}, one easily deduces that 
 for $P \in C^1(\overline \Omega; \R^{3 \times 3})$ and $Q \in    W^{1/p,\,p'}(\partial\Omega,\R^{3\times 3})$ 
\begin{equation}  \skalarProd{SP}{Q}_{\partial \Omega} = \int_{\partial \Omega}     \skalarProd{\dev \sym [P \times (-\nu)]}{Q}  \, d\mathcal{H}^{2}\,.
\end{equation}
Thus, for $P \in C^1(\overline \Omega; \R^{3 \times 3})$ we have $SP = \dev \sym [P \times (-\nu)]$.
Let $\Gamma$ be a relatively open subset of $\partial \Omega$. 
We say that 
$$ \text{$SP = 0$ in $\Gamma$} \qquad \text{if}   \qquad
 \skalarProd{SP}{Q} = 0 \quad \text{$\forall Q \in (W^{1/p,\, p'} \cap C^0)(\partial \Omega; \R^{3 \times 3})$ with
 $Q=0$ on $\partial \Omega \setminus \Gamma$.}
 $$ 
 Note that  $ (W^{1/p, \,p'} \cap C^0)(\partial \Omega)$ is dense in $ W^{1/p,\, p'}(\partial \Omega)$ since it contains
 $\mathrm{Tr}(C^1(\overline \Omega))$. 
We define
\begin{equation}
 W^{1,\,p, \, r}_{0, \Gamma}(\dev\sym\Curl; \Omega,\R^{3\times3}) \coloneqq  
 \{ P \in W^{1,\,p, \, r}(\dev\sym\Curl; \Omega,\R^{3\times3}) \mid \text{$SP = 0$ in $\Gamma$} \}. 
\end{equation}
In particular 
\begin{equation}  \label{eq:smooth_zero_bc}
 T \in C^1(\overline \Omega; \R^{3 \times 3}) \cap  W^{1,\,p, \, r}_{0, \Gamma}(\dev\sym\Curl; \Omega,\R^{3\times3}) \quad \Longrightarrow
\quad \dev \sym [T \times \nu ]= 0 \quad \text{on $\Gamma$.}
\end{equation}
Since $S$ is continuous,  the space ~ $ W^{1,\,p, \, r}_{0, \Gamma}(\dev\sym\Curl; \Omega,\R^{3\times3})$ ~ is a closed subspace of \linebreak $ W^{1,\,p, \, r}(\dev\sym\Curl; \Omega,\R^{3\times3})$.
\begin{theorem} \label{thm:main2sym_new}
Let $\Omega \subset \R^3$ be a bounded Lipschitz domain, let  $1<p<\infty$ and assume that 
$r$ satisfies  \eqref{eq:assumption_r}. Let $\Gamma \subset \partial \Omega$
be relatively open and non-empty. Then there exists a constant $c = c(p, r, \Omega, \Gamma)$ such that 
 for all $P\in  W^{1,\,p,\,r}_{0, \Gamma}(\dev\sym\Curl; \Omega,\R^{3\times3})$ we have
 \begin{equation}
     \norm{ P }_{L^p(\Omega,\R^{3\times3})}\leq 
     c\,\left(\norm{\sym P }_{L^p(\Omega,\R^{3\times3})}+ \norm{ \dev\sym\Curl P }_{L^r(\Omega,\R^{3\times3})}\right).\label{eq:Korn_Lp_thm_SdSC_new}
 \end{equation}
 \end{theorem}

 \begin{remark}
Conti and Garroni \cite{conti2020sharp} and Gmeineder and Spector \cite{gmeineder2020kornmaxwellsobolev} have shown that the estimate
$$  \norm{ P }_{L^p(\Omega,\R^{3\times3})}\leq 
     c\,\left(\norm{\sym P }_{L^p(\Omega,\R^{3\times3})}+ \norm{ \Curl P }_{L^r(\Omega,\R^{3\times3})}\right)$$
     holds also in the borderline case $r=1$ and $p=\frac32$ under the normalization condition $\int_\Omega\skew P \,\intd{x}=0$ similar to \cite{Garroni10}. We do not know if Theorem~\ref{thm:main2sym_new} holds in this borderline case.
\end{remark}

\corrected{
To show Theorem \ref{thm:main2sym_new} we use the following simple fact which will be proved after the proof of Theorem \ref{thm:main2sym_new}:}
\begin{lemma} \label{le:bc_nondegenerate}
 Assume that  $\Gamma \subseteq \R^3$ has the following properties
\begin{center}\quad
\begin{enumerate*}
\item $\Gamma$ is not discrete; \hfill{}
\item $\Gamma$ is not contained in a line; \hfill{}
\item $\Gamma$ is not contained in a circle.
\end{enumerate*}
\end{center}
Let  $A \in \so(3)$, $b,d \in \R^3$, $\beta \in \R$ and consider the  function $f: \R^3 \to \R^3$
given by
$$ f(x) = A x + \beta x + b + \langle d, x \rangle x - \frac12 d \norm{x}^2.$$
Then 
\begin{equation}  \label{label:boundary_conditions_kernel}
f= 0 \quad \text{on $\Gamma$} \quad \Longrightarrow  \quad
A = 0,\,  b=d=0,\,  \beta = 0.
\end{equation}
\end{lemma}
 \begin{proof}[Proof of Theorem \ref{thm:main2sym_new}] We first show that 
 \begin{equation} \label{eq:bc_annihilates_kernel}  K_{S, dSC} \cap W^{1,\,p,\,r}_{0, \Gamma}(\dev \sym \Curl;\Omega;  \R^{3 \times 3})= \{0\}.
 \end{equation}
 Then the assertion will  follow by a standard argument  from \eqref{eq:rigid_Korn_Lp_dsC} and the fact that
 $K_{S,dSC}$ is finite-dimensional while $W^{1,\,p,\,r}_{0, \Gamma}(\dev \sym \Curl;\Omega;  \R^{3 \times 3})$ is closed.
 
 To show  \eqref{eq:bc_annihilates_kernel},   let ~$T  \in K_{S, dSC} \cap W^{1,\,p,\,r}_{0, \Gamma}$.~ Then $T$ is smooth and thus \eqref{eq:smooth_zero_bc} implies that\linebreak ~$\dev \sym [T \times \nu] = 0$~ on $\Gamma$. Since the elements of $K_{S, dSC}$ are skew-symmetric 
 it follows from 
  Observation \ref{obs:2} 
  and the formula for $K_{S, dSC}$ that there exist $\widetilde A \in \so(3)$, $b,d \in \R^3$ and $\beta \in \R$ such that
\begin{align*}
\axl T(x)=\widetilde{A}\,x+\beta\, x+b +\skalarProd{d}{x}\,x -\frac12d\norm{x}^2 = 0 \quad \text{for all $x\in\Gamma$}\,. 
\end{align*}
Now  Lemma~\ref{le:bc_nondegenerate} implies  that all coefficients vanish and hence $T \equiv 0$. This concludes the proof of 
\eqref{eq:bc_annihilates_kernel}.

Assume now that \eqref{eq:Korn_Lp_thm_SdSC_new} does not hold. Then there exists a sequence $P_k \in W^{1,\,p,\,r}_{0, \Gamma}(\dev \sym \Curl;\Omega;  \R^{3 \times 3})$
such that    
$$ \norm{\sym P_k }_{L^p(\Omega,\R^{3\times3})}+ \norm{ \dev\sym\Curl P_k }_{L^r(\Omega,\R^{3\times3})}  \to 0$$
and
$$ \| P_k \|_{L^p(\Omega,\R^{3\times3})} = 1.$$
The assumption  \eqref{eq:assumption_r} on $r$ implies that $W^{1,\,p'}_0(\Omega)$ embeds continuously into $L^{r'}(\Omega)$. Hence $L^r(\Omega)$ embeds
continuously into $W^{-1,\,p}(\Omega)$.
Thus it follows from  \eqref{eq:rigid_Korn_Lp_dsC} that  there exist $T_k \in K_{S,dSC}$ such that
$$  \| P_k - T_k  \|_{L^p(\Omega,\R^{3\times3})} \to 0.$$
In particular, the sequence $T_k$ is bounded in $L^p$ and since $K_{S,dSC}$ is finite-dimensional,  there exists a $T \in K_{S,dSC}$ and a subsequence
such that $T_k \to T$ in $L^p(\Omega, \R^{3 \times 3})$. 
Thus (for the same subsequence) $P_k  \to T$ in $L^p(\Omega, \R^{3 \times 3})$.
Moreover  $\dev \sym \Curl P_k$ converges to zero in $L^r$ and $\dev \sym \Curl T = 0$.
Since  $W^{1,\, p,\, r}_{0, \Gamma}(\dev \sym \Curl;\Omega;  \R^{3 \times 3})$ is a closed subspace of
$W^{1,\, p, \, r}(\dev \sym \Curl;\Omega;  \R^{3 \times 3})$
it follows that $T \in W^{1,\, p, \, r}_{0, \Gamma}(\dev \sym \Curl;\Omega;  \R^{3 \times 3})$.
Hence  \eqref{eq:bc_annihilates_kernel} implies that $T = 0$ and thus $P_k \to 0$ in  $L^p(\Omega, \R^{3 \times 3})$.
This contradicts the hypothesis  $\norm{P_k}_{L^p} = 1$.
\end{proof}

\begin{remark}  
Estimate \eqref{eq:Korn_Lp_thm_SdSC_new} does not hold true in other dimensions, since only  in three dimensions the matrix $\Curl$ returns a square matrix.
\end{remark}


\begin{proof}[\corrected{Proof of Lemma \ref{le:bc_nondegenerate}}] Since $\Gamma$ is not discrete there exists $\bar x \in \R^3$ and $x_k \in \Gamma \setminus \{\bar x\}$ such that
$\lim_{k \to \infty} x_k = \bar x$. The map $g(x) \coloneqq  f(\bar x + x)$ has the same form as $f$ (with different values of the parameters $A,b,d,\beta$).
Thus we may assume without loss of generality that $\bar x = 0$. Since $f$ is continuous we get $f(0) = 0$ and hence $b = 0$. 

Since 
$$ 0= \langle f(x), x \rangle  = \beta \norm{x}^2 + \frac12 \langle d, x \rangle \norm{x}^2 \quad \forall x \in \Gamma \setminus \{0\}$$
we deduce that 
$\beta + \frac12 \langle d, x\rangle = 0$ for all $x \in \Gamma \setminus \{0\}$. Considering points $x_k \in \Gamma \setminus \{0\}$
with $x_k \to 0$ we see that $\beta = 0$ and
$$ \langle d, x \rangle = 0  \quad \forall x \in \Gamma.$$

If $d = 0$ then $f(x) = Ax$. If $A \ne 0$ then the kernel of $A$ is a line since $A \in \so(3)$. Thus $\Gamma$ would be contained in a line which
contradicts our assumption. Hence for $d= 0$ we get $A=0$ and we are done. 

If $d \ne 0$ then $\Gamma$ is contained in the hyperplane perpendicular to $d$. \corrected{Since $\skalarProd{d}{f(x)}=0$ for all $x\in\Gamma$ and $A$ is skew-symmetric we get ~ $2\skalarProd{A\,d}{x}+\norm{d}^2\norm{x}^2=0$ ~ for all $x\in\Gamma$. This implies that
\begin{equation*}
 \left\lVert x+\frac{1}{\norm{d}^2}A\,d\right\rVert^2=\frac{\norm{A\,d}^2}{\norm{d}^4} \quad \forall \ x\in\Gamma.
\end{equation*}
Since $A$ is skew-symmetric, the vector $A\,d$ is contained in the plane perpendicular to $d$. It follows that $\Gamma$ is either a point (and hence discrete) or a circle ($x\in\Gamma$ and $A\,d$ lie in the same hyperplane) with center $-\frac{1}{\norm{d}^2}A\,d$ and radius $\frac{\norm{A\,d}}{\norm{d}^2}$ which contradicts our assumption.
}
\end{proof}

It is well-known, that Korn's inequality and Poincar\'{e}'s inequality are not equivalent, however, due to the presence of the $\Curl$ we get back both inequalities from our general result \eqref{eq:Korn_Lp_thm_SdSC_new}. Indeed, in the compatible case $P=\D u$ we recover a tangential Korn inequality.
\begin{corollary}\label{cor:tKorn}
  Let $\Omega \subset \R^3$ be a bounded Lipschitz domain, $1<p<\infty$ and $\Gamma$ a relatively open non-empty subset in $\partial\Omega$. There exists a constant $c=c(p,\Omega,\Gamma)>0$ such that for all $u\in  W^{1,\,p}(\Omega,\R^{3})$ with ~ $\dev\sym(\D u \times \nu) = 0$ ~ on $\Gamma$  we have
   \begin{equation}
   \norm{\D u }_{L^p(\Omega,\R^{3\times 3})} \le c\, \norm{\sym \D u}_{L^p(\Omega,\R^{3\times3})}\,.
 \end{equation}
\end{corollary}
\begin{proof} This follows from Theorem~\ref{thm:main2sym_new}  by setting $P = \D u$. 
\end{proof}
\begin{remark} \label{re:cor_tKorn}
 This boundary condition is rather weak. If  $\Gamma$ is flat, then the condition $\dev\sym(\D u \times \nu)_{|\Gamma} = 0 $ implies that $u=\alpha\,x+b$ along $\Gamma$ with $\alpha\in\R$ and $b\in\R^3$, see Appendix \ref{sec:flatbdry}. 
\end{remark}

For skew-symmetric $P=\Anti(a)$ we recover from \eqref{eq:Korn_Lp_thm_SdSC_new} a Poincar\'{e}'s inequality involving only the deviatoric (trace-free) part of the symmetrized gradient. Such a   Poincar\'{e}-type inequality can also be generalized to functions of bounded deformation, cf.~\cite{FuchsRepin2010Poincaretracefree}.
\begin{corollary}\label{cor:devsymPoin}
  Let $\Omega \subset \R^3$ be a bounded Lipschitz domain and $1<p<\infty$. Set $ W^{1,p}_{\Gamma,0}(\Omega,\R^3) \coloneqq  \{ a \in W^{1,p}(\Omega; \R^3) \mid \mathrm{Tr}\, a = 0 \text{ on $\Gamma$}\}$. There exists a constant $c=c(p,\Omega,\Gamma)>0$ such that for all $a\in  W^{1,\,p}_{\Gamma,0}(\Omega,\R^3)$, we have
 \begin{equation}
\norm{a}_{L^p(\Omega,\R^3)}\le c\, \norm{\dev\sym\D a}_{L^p(\Omega,\R^{3\times3})}\,.
 \end{equation}
 \end{corollary}
\begin{proof} This follows from Theorem~\ref{thm:main2sym_new}  by setting $P = \Anti(a)$ and the following observations:\\
 $\dev\sym(\Anti(a) \times \nu) = 0 \ \Leftrightarrow \ a=0$ on $\Gamma$, $\Curl(\Anti(a))=L(\D a)$ \corrected{and the form of $\Anti(a)$, cf.~\eqref{eq:Anti_Komp}.}
\end{proof}
The results of Theorem~\ref{thm:main2sym_new} , Corollary~\ref{cor:tKorn} and Corollary~\ref{cor:devsymPoin} can be graphically summarized as follows.
\begin{center}
 \includegraphics[width=\textwidth]{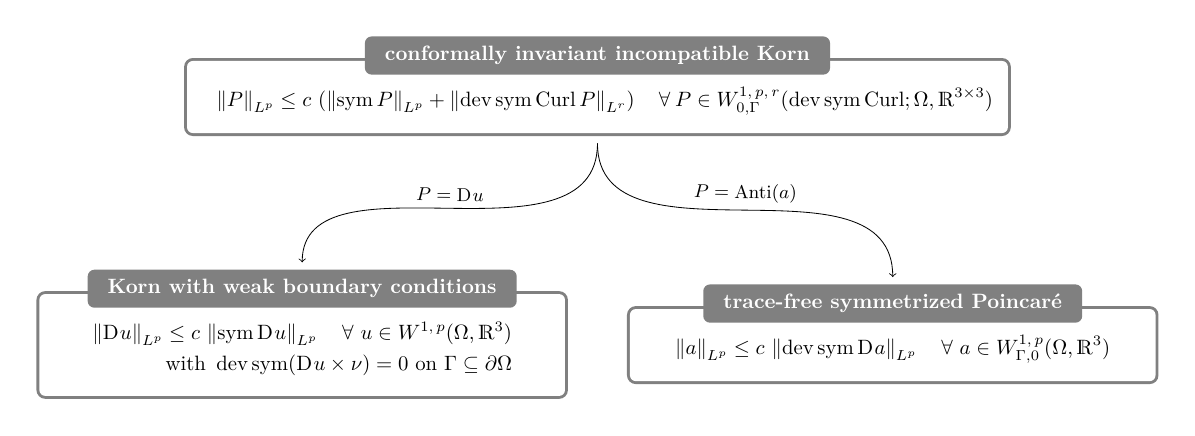}
\end{center}

\section{Comparison of the spaces  $W^{1,\,p}(\sym \Curl)$ and $W^{1,\,p}(\dev \sym \Curl)$}
Using the linear expression of the entries of $\D\Curl P$ in terms of the entries of $\D\dev\Curl P$ the authors of \cite{agn_lewintan2020KornLp_tracefree} showed  that for all $P\in\mathscr{D}'(\Omega,\R^{3\times3})$ and all $m\in\Z$ one has
 \begin{equation}\label{eq:gleicheReg}
  \Curl P\in W^{m,\,p}(\Omega,\R^{3\times3})\quad \Leftrightarrow\quad \dev\Curl P\in W^{m,\,p}(\Omega,\R^{3\times3}).
 \end{equation}
 One might, therefore, wonder whether the spaces ~ $W^{1,\,p}(\dev \sym \Curl; \Omega; \R^3)\coloneqq W^{1,\,p,\,p}(\dev \sym \Curl; \Omega; \R^3)$ ~ and ~ $W^{1,\,p}(\sym \Curl; \Omega; \R^3)$ ~ are actually identical, where 
 \[
  W^{1,\,p}(\sym \Curl; \Omega\corrected{, \R^{3\times3}})\coloneqq\{P\in L^p(\Omega,\R^{3\times 3}) \mid \sym\Curl P \in L^p(\Omega,\R^{3\times 3})\}\,.
 \]
 \corrected{We first note that clearly $W^{1,\,p}(\sym \Curl; \Omega, \R^{3\times3}) \subset W^{1,\,p}(\dev \sym \Curl;\Omega, \R^{3\times3})$ and that
the natural norm $\norm{P}_{L^p} +  \norm{\dev \sym \Curl P}_{L^p}$ on $W^{1,\,p}(\dev \sym \Curl;\Omega, \R^{3\times3})$ is weaker than the natural
norm on $W^{1,\,p}(\sym \Curl; \Omega, \R^{3\times3})$. Thus, in view of the open mapping theorem, the two spaces are identical if and only if the two natural norms are equivalent.
In view of the second estimate in \eqref{eq:normequiv-top} (which follows directly form \eqref{eq:PtbedevsymPtb} by dividing by $\norm{b}^2$) one might expect that this is really the case.
Indeed, using the reasoning in Section \ref{sec:Formalism}, we see that the algebraic identity \eqref{eq:PtbedevsymPtb}
shows that }
 \begin{equation}\label{eq:fuerLap}
 \Delta \sym \Curl P = L(\D^2\dev\sym\Curl P)
\end{equation}
\corrected{in the sense of distributions. The identity \eqref{eq:fuerLap} yields interior estimates for all compactly contained subsets $\Omega'$ of $\Omega$ of the form
\begin{equation}
  \norm{\sym \Curl P}_{L^p(\Omega')} \le  C(\Omega')\, (\norm{\dev \sym \Curl P}_{L^p(\Omega)} + \norm{ P}_{L^p(\Omega)})
\end{equation}
but we will see in the proof of  Theorem \ref{thm:spaces} assertion \ref{item:4thm:spaces} below that this is not enough to obtain equivalence of the norm on the
full set $\Omega$ because we do not impose boundary conditions on $P$.

To illustrate the obstruction to a global estimate, let use consider the following example. Let  $D$ be the unit ball in $\R^2$ and consider the spaces
$W^{2,2}(D)$ and $W^{2,2}_\Delta(D)\coloneqq \{ u \in W^{1,2}(D) \mid \Delta u = 0\}$ with norms $\norm{u}_{W^{2,2}} = \norm{u}_{L^2(\Omega)} + \norm{\D u}_{L^2(\Omega)} +
 \norm{\D^2 u}_{L^2(\Omega)}$
 and $\norm{u}_{\Delta} =  \norm{u}_{L^2(\Omega)} + \norm{\D u}_{L^2(\Omega)} + \norm{\Delta u}_{L^2(\Omega)}$, respectively. Since $\Delta$ is an elliptic operator, we have
 interior estimates $\norm{u}_{W^{2,2}(\Omega')} \le C(\Omega') \,\norm{u}_{\Delta}$, but the norms are not equivalent since for the harmonic functions
 $f_k(x) \coloneqq \Re(e^{k(x_1 + \komplexI\, x_2)})$ we get $\lim_{k \to \infty} \|u\|_{W^{2,2}}/ \|u\|_{\Delta} = \infty$. 
  The reason: while  the symbol $\sigma(\xi) = -(\xi_1^2 + \xi_2^2)$ of the operator $\Delta$ has no non-trivial real zeroes (this is ellipticity), it does 
 have the non-trivial complex zeroes $\xi_1 = k$, $\xi_2 = \komplexI\, k$. This allows us to construct the ‘bad’  functions $f_k$.  A similar analysis of the action of the matrix-valued symbols of the operators $\sym \Curl$ and $\dev \sym \Curl$ on $\C^3$ will allow us below
 to construct maps $P_k$ which show that the norms $\norm{\dev \sym  \Curl P}_{L^p(\Omega)} + \norm{P}_{L^p(\Omega)}$ and $\norm{\sym \Curl P}_{L^p(\Omega)} + \norm{ P}_{L^p(\Omega)}$
 are not equivalent if $\Omega$ is a bounded domain. By contrast, one can use Fourier transform to show that  the norms are equivalent for periodic $P$ or $P \in L^p(\R^3,\R^{3 \times 3})$, which we show for the convenience of the reader in Appendix \ref{App:Fourier}.
 For the latter purposes we start with the following proposition.}

 \begin{proposition}   \label{pr:miklin_hoermander}  Let $V$ be a finite-dimensional vectorspace and denote by $\mathrm{Lin}(V,V)$ the space of linear maps
 from $V$ to $V$.
 Let $\mathbb A$ and $\mathbb{\widetilde A}$ be linear maps from $\R^n$ to $\mathrm{Lin}(V,V)$. 
 Assume that 
 \begin{equation}  \label{eq:equal_kernel}
\mathbb{\widetilde A}(\xi)  a= 0 \quad  \forall\, \xi \in \R^n \setminus \{0\} \quad  \forall\, a \in \ker \mathbb{A}(\xi) 
 \end{equation}
 and
  \begin{equation} \label{eq:constant_rank}
\dim \ker \mathbb{A}  \quad \text{is constant on $\R^n \setminus \{0\}.$}
\end{equation}
Define differential operators by 
$$ \mathcal A = \mathbb{A}(\nabla) \coloneqq \sum_{j=1}^n \mathbb{A}(e_j) \partial_j \quad \text{and} \quad    \widetilde{ \mathcal A} = \widetilde{\mathbb{A}}(\nabla) \coloneqq\sum_{j=1}^n \widetilde{\mathbb{A}}(e_j) \partial_j. $$
Then for each $p \in (1, \infty)$ there exists a constant  $c = c(p)$
such that 
\begin{align}  \label{operator_bound_Rn}  
\norm{\widetilde { \mathcal A} f}_{L^p(\R^n,V)}& \le c\, \norm{ \mathcal A f}_{L^p(\R^n, V)}  \quad \forall\, f \in L^p(\R^n, V)
\intertext{and}
\label{operator_bound_Tn}  
\norm{\widetilde{ \mathcal A} f}_{L^p(\T^n,V)} &\le c\,  \norm{ \mathcal A f}_{L^p(\T^n, V)}  \quad \forall\, f \in L^p(\T^n, V)
\end{align}
in the distributional sense.
 \end{proposition}
 
 \begin{proof} This is well-known, cf.~e.g.~\cite[pp.~1362--1365]{FonsecaMueller1999quasiconvexity} or \cite[Section IV.3]{Stein1970singints}. We recall the argument for the convenience of the reader.   We focus on \eqref{operator_bound_Rn}, the proof
 of  \eqref{operator_bound_Tn}  is analogous. If suffices to show \eqref{operator_bound_Rn}  for $f \in C_c^\infty(\R^n, V)$. Then the general case
 follows by approximation.  We will construct a linear bounded operator $\mathcal M: L^p(\R^n, V) \to L^p(\R^n, V)$ such that
 \begin{equation}  \label{eq:equality_tilde_A_smooth}
  \widetilde{\mathcal A} f = \mathcal M \mathcal A f  \quad \forall f \in C_c^\infty(\R^n, V).
  \end{equation}
For $ \xi \in \R^n\setminus \{0\}$ we define
\begin{align} \mathbb P(\xi): V \to V \quad \text{as the orthogonal projection onto $\ker \mathbb A(\xi)$}
\intertext{and we define $\mathbb Q(\xi):  V\to V$ by}
 \mathbb Q(\xi) \mathbb{A}(\xi)  =   \mathrm{Id} - \mathbb P(\xi), \quad \mathbb Q \equiv 0 \quad \text{on $(\mathrm{range}\, \mathbb A(\xi))^\perp$.}
\end{align}
It follows from  \eqref{eq:constant_rank}
that $ \xi \mapsto \mathbb P(\xi)$ is smooth and homogeneous of degree zero on $\R^n \setminus \{0\}$, while
$ \xi \mapsto \mathbb Q(\xi)$ is smooth and homogeneous of degree $-1$.  
For $\xi \in \R^{n}\setminus \{0\}$ define
\begin{equation} \mathbb M(\xi)\coloneqq\widetilde{\mathbb{A}}(\xi) \mathbb Q(\xi).\end{equation}Then $\mathbb M$ is homogeneous of degree zero and smooth on the unit sphere $\mathbb{S}^{n-1}$ of $\R^n$. 
For $f \in C_c^\infty(\R^n, V)$ define 
\begin{equation}\mathcal M f = (\mathcal F)^{-1} \mathbb M \mathcal F f\end{equation}
where $\mathcal F$ denotes the Fourier transform. By the Mikhlin-H\"ormander multiplier theorem $\mathcal M$ has a unique extension to
a bounded operator on $L^p(\R^n,V)$. Moreover we have 
\begin{equation} \mathbb M(\xi) \mathbb A(\xi) =\widetilde{ \mathbb A}(\xi) \mathbb Q(\xi) \mathbb A(\xi) = \widetilde{ \mathbb A}(\xi) - \widetilde{\mathbb A}(\xi) \mathbb P(\xi)
=  \widetilde{ \mathbb A}(\xi).
\end{equation}
Here we used the assumption  \eqref{eq:equal_kernel} in the last identity. 
Now 
 \eqref{eq:equality_tilde_A_smooth} follows from the definition of $\mathcal M$. 
 \end{proof}

\corrected{On bounded sets we will make use of the following behavior.}
 
   \begin{proposition} \label{pr:k-growth} Let $\Omega \subset \R^3$ be bounded, open and non-empty.  Let $z = x_1 + \komplexI\, x_2$ and
 let $q_k(x) = z^k$. 
 Then  
 $$ \lim_{k \to \infty} \frac{\norm{k  q_{k-1}}_{L^p(\Omega\corrected{,\C})}}{ \norm{q_k}_{L^p(\Omega\corrected{,\C})}} = \infty.$$ 
 \end{proposition}
 
 \begin{proof} For $\delta > 0$ define $U_\delta := \{ x  \in \R^3 \mid x_1^2 + x_2^2  < \delta^2\}$.  
  Since $\corrected{|}{q_{k-1}/ q_k}\corrected{|} \le \delta^{-1}$ on $\Omega \setminus  U_\delta$
 we have
\begin{equation}  \lim_{k \to \infty} \frac{\norm{k  q_{k-1}}_{L^p(\Omega \setminus U_\delta\corrected{,\C})}}{ \norm{q_k}_{L^p(\Omega \setminus U_\delta\corrected{,\C})}} = \infty.\end{equation}                                                                                                                                                    
 Now the assertion follows from the fact that
 $$   \lim_{k \to  \infty}  \frac{ \norm{q_k}_{L^p(\Omega \setminus U_\delta\corrected{,\C})}}{  \norm{q_k}_{L^p(\Omega\corrected{,\C})}} = 1$$
 whenever $\delta > 0$ is so small that $\Omega \setminus U_{2\delta}$ has positive measure.
 \end{proof}
\corrected{
 With these preparations in hand we arrive at our final result.
}

 \begin{theorem}\label{thm:spaces}  The following assertions hold for $p \in (1, \infty)$. 
\begin{enumerate}
\item (whole space $\R^3$)  \label{it:comparison1} There exists a constant $c = c(p)$ such that for $P \in L^p(\R^3, \R^{3 \times 3})$
\begin{equation}
\| \sym  \Curl P\|_{L^p(\R^3, \R^{3 \times 3})} \le c\, \|  \dev \sym  \Curl P\|_{L^p(\R^3, \R^{3 \times 3})};
\end{equation}
\item  (periodic functions)   \label{it:comparison2}  if $\T^3 = \R^3/ \Z^3$ then for all $P \in L^p(\T^3, \R^{3 \times 3})$
\begin{equation}
\norm{ \sym  \Curl P}_{L^p(\R^3, \R^{3 \times 3})} \le c\, \norm{  \dev \sym  \Curl P}_{L^p(\R^3, \R^{3 \times 3})};
\end{equation}
\item  (half-spaces)   \label{it:comparison3}  if $\Omega$ is a half-space then for $P \in L^p(\Omega,\R^{3 \times 3})$ the  seminorms 
$\| \sym  \Curl P\|_{L^p(\Omega, \R^{3 \times 3})}$ and   $ \|  \dev \sym  \Curl P\|_{L^p(\Omega, \R^{3 \times 3})}$
are not equivalent;
\item \label{item:4thm:spaces} (bounded sets)    \label{it:comparison4} If $\Omega\subset\R^3$ is a bounded, open, non-empty set then
$$W^{1,\,p}(\sym \Curl; \Omega, \R^3) \ne W^{1,\,p}(\dev \sym \Curl; \Omega, \R^3).$$
\end{enumerate}
\end{theorem}

\paragraph{Notation}  \quad
In this subsection we  use the notation
$$ \skalarProd{a}{b} \coloneqq \sum_{j=1}^3  a_j b_j  \quad \text{for $a, b \in \C^3$}.$$
Note that this is different from the usual sesquilinear form $\sum_j a_j \overline b_j$ where $\overline z$ denotes the complex conjugate of a complex number $z$.  In particular $\skalarProd{a}{a}$ is \textbf{not} nonnegative on $\C^3$.

 \begin{proof}
The first and second  assertion for the whole space and periodic functions follow from the estimate \eqref{eq:normequiv-top}
\begin{equation}  \label{eq:bound_real}
 \forall\,  \xi \in \R^3   \quad  \forall\,  \widehat P \in \R^{3 \times 3} \quad  \norm{\sym (\widehat P \times \xi)} \le (1+ \sqrt 3)  
 \norm{\dev \sym (\widehat P \times \xi)},
\end{equation}
 the fact that  $\dim \{ \widehat P \mid \dev \sym (\widehat P \times \xi) = 0\}=4$ is independent of $\xi$ for $\xi \in \R^3 \setminus \{0\}$
 and Proposition~\ref{pr:miklin_hoermander}, applied to the operators ~ $\mathbb A(\xi) P  = \dev \sym (P \times \xi)$ ~ and ~ $\widetilde{\mathbb A}(\xi) P = \sym (P \times \xi)$.
 
 To prove the third and fourth  assertion we first show that
\begin{equation}  \label{eq:nonbound_complex}
\exists\, \xi \in \C^3   \quad  \exists\, \widehat P \in \C^{3 \times 3} : \quad  \dev \sym (\widehat P \times \xi) = 0 \quad \text{and} \quad  \sym (\widehat P \times \xi) \ne 0.
\end{equation}
Then the assertion will follow by standard arguments. One such example is given by
\begin{subequations}
\begin{align}\label{eq:komplexesBsp}
 \widehat P=\begin{pmatrix} 0 & 0 & -1 \\ 0 & 0 & \komplexI \\ 0 & -\komplexI & 0\end{pmatrix} \qquad &\text{and}\qquad \xi=\begin{pmatrix}1 \\ \komplexI \\ 0\end{pmatrix} \quad \Rightarrow \quad \widehat P\times \xi = \begin{pmatrix} \komplexI & -1 & 0 \\ 1 & \komplexI & 0 \\ 0 & 0 & \komplexI\end{pmatrix}
\intertext{so that}
 \sym(\widehat P\times \xi)=\komplexI \cdot \id  \qquad&\corrected{\text{but}}\qquad \dev\sym(\widehat P\times \xi)= 0.\label{eq:symunddevsym}
\end{align}
\end{subequations}
Further examples which fulfill \eqref{eq:nonbound_complex} can be found splitting $\widehat P$ into the symmetric and skew-symmetric part:  $\widehat P = \widehat S + \Anti(\widehat a)$. 
By \eqref{eq:firststepdevsymPtimesb}
\begin{align}   \label{eq:devsym_minus_sym}
 \dev\sym( \widehat P\times \xi)  =  \sym(\widehat P\times \xi) +\frac23\skalarProd{\widehat a}{\xi}\,\id.
\end{align}
Thus it suffices to find $\widehat P \in \C^{3 \times 3}$ and $\xi \in \C^3$ such that
~ $\dev \sym (\widehat P \times \xi) = 0$ ~ and ~ $\skalarProd{\widehat a}{\xi} \ne 0$.
 Indeed, the example in \eqref{eq:komplexesBsp} satisfies these conditions.
 
 Now we show that in a half-space the seminorms $\norm{\sym \Curl  \cdot}_{L^p(\Omega, \R^{3 \times 3})}$ and 
$\norm{ \dev \sym \Curl \cdot }_{L^p(\Omega, \R^{3 \times 3})}$ are not equivalent.
Since the operators $\dev \sym \Curl$ and $\sym \Curl$ interact naturally with rotations it suffices to consider the 
half-space 
$$ \Omega = \{ x \in \R^3 \mid x_1 < 0 \}.$$
Note that the norms are equivalent for real-valued fields $P$ if and only if they are equivalent for complex-valued fields $P$. 
Let $\xi$ and $\widehat P$ be as in \eqref{eq:komplexesBsp}. 
For a constant vector $b \in \R^3$ and a scalar function $\zeta$ we have $\curl (b\,\zeta) = b \times (-\nabla \zeta)$. Since $\Curl$ acts row-wise
we have for a constant matrix
$\widehat P$ the identity $\Curl (\widehat P\, \zeta) =- \widehat  P \times \nabla \zeta$. Thus for all $k\in\N$
\begin{align}\label{eq:curlrechnung2} \Curl  \left(\widehat P\mathrm{e}^{k \skalarProd{\xi}{x}} \right) = 
- \widehat P \times \nabla \mathrm{e}^{k \skalarProd{\xi}{x}} = - k \mathrm{e}^{k \skalarProd{\xi}{x}} (\widehat P \times \xi)
\end{align}
so that with \eqref{eq:symunddevsym} we have
\begin{equation}\label{eq:desymCurl0}
 \sym\Curl  \left(\widehat P\mathrm{e}^{k \skalarProd{\xi}{x}} \right) = -\komplexI\,k\,\mathrm{e}^{k \skalarProd{\xi}{x}}\cdot\id \quad \text{and}\quad \dev\sym\Curl\left(\widehat P\mathrm{e}^{k \skalarProd{\xi}{x}} \right) =0.
\end{equation}
Let $\eta \in C_c^\infty(B(0,2))$ be a cut-off function such that $\eta =1$ in $B(0,1)$ and consider the functions
\begin{equation}
  P_k(x) = \frac1k \widehat P \mathrm{e}^{k \skalarProd{\xi}{x}}  \eta(x).
\end{equation}
Then 
\[
 \Curl P_k(x)=\frac{\eta(x)}{k}\Curl  \left(\widehat P\mathrm{e}^{k \skalarProd{\xi}{x}} \right)-\frac1k\mathrm{e}^{k \skalarProd{\xi}{x}}\widehat P \Anti(\nabla\eta)\overset{\eqref{eq:curlrechnung2}}{=} -\mathrm{e}^{k \skalarProd{\xi}{x}}\left(\widehat P \times \xi+ \frac1k \widehat P \Anti(\nabla\eta)\right)
\]
and with \eqref{eq:symunddevsym} we obtain
\begin{subequations}
\begin{align}
 \norm{\dev \sym \Curl P_k(x)} \le C  \frac1k  e^{kx_1} \sup \norm{\nabla \eta}
 \intertext{and}
\sym \Curl P_k =  -\komplexI\,\mathrm{e}^{k\, x_1 + \komplexI\, k\, x_2}\cdot\id     \quad   \text{in $B(0,1)$}.
\end{align}
 \end{subequations}
From this we easily conclude that $\norm{\sym \Curl P_k}_p/ \norm{\dev \sym \Curl P_k}_p \to \infty$ ~ which shows claim \ref{it:comparison3}.

Finally, we prove the last assertion \ref{it:comparison4}. Let $\Omega$ be a bounded, open, non-empty set.
It suffices to show that in $W^{1,\, p}(\sym\Curl; \Omega, \C^{3 \times 3})$ the \corrected{norms} ~
$\corrected{\norm{\cdot}_{L^p(\Omega, \C^{3 \times 3})}+}\norm{\sym\Curl  \cdot}_{L^p(\Omega, \C^{3 \times 3})}$ ~ and ~
$\corrected{\norm{\cdot}_{L^p(\Omega, \C^{3 \times 3})}+}\norm{ \dev \sym \Curl \cdot }_{L^p(\Omega, \C^{3 \times 3})}$ ~ are not equivalent. Indeed, this implies that also in ~ $W^{1,\, p}(\sym\Curl; \Omega, \R^{3 \times 3})$ ~ the norms 
$\corrected{\norm{\cdot}_{L^p(\Omega, \R^{3 \times 3})}+}\norm{\sym\Curl  \cdot}_{L^p(\Omega, \R^{3 \times 3})}$  and 
$\corrected{\norm{\cdot}_{L^p(\Omega, \R^{3 \times 3})}+}\norm{ \dev \sym \Curl \cdot }_{L^p(\Omega, \R^{3 \times 3})}$ are not equivalent. Thus, since the identity map $$i : W^{1,\, p}(\sym\Curl; \Omega, \R^{3 \times3}) \to W^{1,\, p}(\dev \sym\Curl; \Omega, \R^{3 \times3})$$ is continuous it then follows from the open mapping theorem that  $$W^{1,\, p}(\dev \sym\Curl; \Omega, \R^{3 \times3})\ne W^{1,\, p}( \sym\Curl; \Omega, \R^{3 \times3}).$$

Let $\xi$ and $\hat P$ be again as in \eqref{eq:komplexesBsp}. Set  $z = x_1 + \komplexI x_2$ and 
\begin{equation}  P_t(x) := \widehat P \mathrm{e}^{t\skalarProd{\xi}{x}} = \widehat P \mathrm{e}^{t z}.
\end{equation}
Then as in \eqref{eq:desymCurl0}
\begin{equation}  \label{eq:devsym_kernel_t}  \dev \sym \Curl P_t = 0
\end{equation}
and
\begin{equation}   \label{eq:sym_kernel_t} 
 \sym \Curl P_t = -\komplexI\, t\,  \mathrm{e}^{t z}\cdot \id.
\end{equation}
Let $Q_k(x) =  \widehat P z^k$. Taking the $k$-th derivative of  \eqref{eq:devsym_kernel_t}  and 
 \eqref{eq:sym_kernel_t}  and evaluating at $t=0$ we get, for all $k \in \N$,
 \begin{equation}
 \dev \sym \Curl Q_k = 0,   \quad \sym \Curl Q_k =-\komplexI\,k\,z^{k-1}\cdot \id.
 \end{equation}
 It follows from Proposition   \ref{pr:k-growth} that 
 $$ \lim_{k \to \infty} \frac{ \norm{\sym \Curl  Q_k}_{L^p(\Omega, \C^{3 \times 3})} }
 {   \norm{  Q_k}_{L^p(\Omega, \C^{3 \times 3})}  +  \norm{\dev \sym \Curl  Q_k}_{L^p(\Omega, \C^{3 \times 3})    }}         
 = \infty.
 $$
 This concludes the proof of the theorem.
\end{proof}

\corrected{
\begin{remark}
 Assertion \ref{item:4thm:spaces} of Theorem \ref{thm:spaces} is complemented by the following two strict inclusions:
 \begin{equation}
 W^{1,\,p}(\Omega,\R^{3\times3})\subsetneqq W^{1,\,p}(\Curl;\Omega,\R^{3\times3})\subsetneqq W^{1,\,p}(\sym\Curl;\Omega,\R^{3\times3}).
\end{equation}
 To see that the first inclusion is strict,  we may use functions of the form $P_k = \D u_k$
 where $u_k = w(kx)$ and $w: \R^3 \to \R^3$ is periodic, to see that the corresponding norms
 are not equivalent. To see that the second inclusion is strict we can use functions of the form
 $P_k(x) = \zeta(kx)\cdot \id$ where $\zeta: \R^3 \to \R$ is periodic, and observe that $\sym \Curl P_k =k\,  \sym (\Anti(\nabla \zeta)(kx)) = 0$.
\end{remark}
}

\subsubsection*{Acknowledgment} The authors thank Ionel-Dumitrel Ghiba, University of Iasi, Romania, for helpful discussions \corrected{and also the anonymous referee for his valuable comments and suggestions}. This work was initiated in the framework of the  Priority Programme SPP 2256 'Variational Methods for Predicting Complex Phenomena in Engineering Structures and Materials'  funded by the Deutsche Forschungsgemeinschaft (DFG, German research foundation), Project-ID  422730790, 
by a collaboration of projects  'Mathematical analysis of microstructure in supercompatible alloys'  (Project-ID 441211072) and  'A variational scale-dependent transition scheme - from Cauchy elasticity to the relaxed micromorphic continuum'   (Project-ID   440935806).
Peter Lewintan and Patrizio Neff were supported by the Deutsche Forschungsgemeinschaft (Project-ID 415894848). Stefan M\"uller has also been supported by the Deutsche Forschungsgemeinschaft through the Hausdorff Center for Mathematics (GZ EXC 2047/1, Projekt-ID 390685813) and the collaborative research centre 'The mathematics of emerging effects' (CRC 1060, Projekt-ID 211504053).

 \printbibliography

  {\footnotesize
 \begin{alphasection}
\section{Appendix}

\subsection{Geometrical interpretation of tangential boundary conditions}
\subsubsection{The case $P\times \nu= 0$  following \cite{Gurtin2005gradientplastBurger}}
In this appendix we provide the reader with the development of Gurtin and Needleman \cite{Gurtin2005gradientplastBurger} adapted to our notation.
Let $\nu\in\R^3$ be a unit vector. For the projection onto the plane perpendicular to $\nu$ we
can consider one of the following matrix representations:
\begin{equation}\label{eq:projektion}
 \mathbb{P}_\nu = \id - \nu\otimes \nu \overset{\eqref{eq:prod_id}}{=} -\Anti(\nu)\times \nu \overset{\eqref{eq:anti^2}}{=} -\Anti(\nu)\Anti(\nu) = \Anti(\nu)^T\Anti(\nu)\,.
\end{equation}
The last expression shows directly that the vector product of $\mathbb{P}_\nu$ and $\nu$ commutes:
\begin{equation}\label{eq:etimesPe}
 \nu\times\mathbb{P}_\nu = -\Anti(\nu)\Anti(\nu)\Anti(\nu)= \mathbb{P}_\nu\Anti(\nu)=\mathbb{P}_\nu\times \nu\underset{\norm{\nu}=1}{\overset{\eqref{eq:anti^3}}{=}}\Anti(\nu)\in\so(3)
\end{equation}
and we have moreover for any $(3\times 3)$-matrices $P$ and $H$:
\begin{equation}
\begin{split}
\skalarProd{P\times \nu}{H} = \skalarProd{P\Anti(\nu)}{H}&= -\skalarProd{P}{H\Anti(\nu)} = -\skalarProd{P}{H\times \nu}
 \overset{\mathclap{\eqref{eq:etimesPe}}}{=}\ \skalarProd{P\,\mathbb{P}_\nu\Anti(\nu)}{H} = - \skalarProd{P\,\mathbb{P}_\nu}{H\times \nu}.
\end{split}
\end{equation}
and also
\begin{align}
 \norm{P \,\mathbb{P}_\nu}^2&=\skalarProd{P \,\mathbb{P}_\nu}{P \,\mathbb{P}_\nu}\overset{\eqref{eq:projektion}}{=} \skalarProd{P\Anti(\nu)\Anti(\nu)}{P\Anti(\nu)\Anti(\nu)} = -\skalarProd{P\Anti(\nu)}{P\Anti(\nu)\Anti(\nu)\Anti(\nu)}\\
 &\underset{\norm{\nu}=1}{\overset{\eqref{eq:anti^3}}{=}}\skalarProd{P\Anti(\nu)}{P\Anti(\nu)} = \norm{P\times \nu}^2\\
 & = -\skalarProd{P}{P\Anti(\nu)\Anti(\nu)}\overset{\eqref{eq:projektion}}{=} \skalarProd{P}{P(\id-\nu\otimes \nu)}= \skalarProd{P}{P}-\skalarProd{P}{P\,\nu\otimes \nu}= \norm{P}^2-\skalarProd{P\,\nu}{P\,\nu}\\
 & = \norm{P}^2-\norm{P\,\nu}^2\,.
\end{align}
Thus, ~ $P \,\mathbb{P}_\nu=0$ ~ if and only if ~ $P\times \nu=0$. The latter condition can be tested by applying the scalar product with deviatoric (trace-free) matrices, it holds:
\begin{equation}\label{eq:gurtin_dev}
 \skalarProd{P\times \nu}{D}= 0 \ \forall D \text{ with $\tr D=0$} \quad \Leftrightarrow \quad P\times \nu =0.
\end{equation}
This implies ~ $\dev(P\times \nu)=0$ ~ if and only if ~ $P\times \nu =0$~ and extends, of course, to the case of arbitrary non-zero vector $\nu\in\R^3$, cf. also our Observation 2.2 in \cite{agn_lewintan2020KornLp_tracefree} and shows
\begin{equation}
 \norm{\dev(P\times \nu)}\le\norm{P\times \nu}\le C\cdot\norm{\dev(P\times \nu)}.
\end{equation}

To establish \eqref{eq:gurtin_dev}, let $H$ be an arbitrary matrix and consider the trace-free matrix $D\coloneqq H -\tr(H)\nu\otimes \nu$. By the assumption we have
\begin{equation}
 0 = \skalarProd{P\times \nu}{D}= \skalarProd{P\times \nu}{H -\tr(H)\nu\otimes \nu}= \skalarProd{P\times \nu}{H}+\tr(H)\skalarProd{P}{(\nu\otimes \nu)\times \nu}= \skalarProd{P\times \nu}{H}.
\end{equation}
Since $H$ is arbitrary, it follows $P\times \nu=0$.
\subsubsection{The compatible case $\D u\times \nu =0$}
Let $\Gamma$ be a relatively open (non-empty) connected subset of the boundary $\partial \Omega$ and assume that $P=\D u$ is compatible. The condition $\D u\times\nu_{|\Gamma}\equiv0$ is equivalent to $\D u\, \mathbb{P}_\nu{}_{|\Gamma}\equiv0$ which can also be written as $\D u\,\tau{}_{|\Gamma}\equiv0$ for all tangential directions on $\Gamma$, meaning that all tangential derivatives of $u$ along $\Gamma$ are vanishing. Thus, $u$ has to be constant along $\Gamma$, since for any curve $\gamma:[0,1]\to\Gamma$ on $\Gamma$ we have $\frac{\mathrm{d}}{\mathrm{d} s}u(\gamma(s))=\D u(\gamma(s))\,\gamma'(s)=0$, cf.~\cite[p.~35]{Girault1986FEM}.

\subsubsection{The case $\sym(P\times \nu)=0$}
If ~ $\sym(P\times \nu)=0$, ~ then there exists a vector $a\in\R^3$ so that  ~$P\times \nu =\Anti(a)$. ~ Hence,
\begin{equation}
 a\times \nu = \Anti(a)\,\nu = (P\times \nu)\, \nu= P \Anti(\nu)\, \nu = P(\nu\times \nu)=0
\end{equation}
and $a$ has to be of the form $a=\alpha\cdot \nu$ with $\alpha\in\R$. Thus we have $P\times \nu=\alpha\cdot\Anti(\nu)=\alpha\cdot\id\times \nu$ and we conclude
\begin{equation}
 \sym(P\times \nu)=0 \quad \Leftrightarrow \quad P\times \nu = \alpha\cdot\Anti(\nu) \quad \Leftrightarrow \quad (P-\alpha\cdot\id)\times \nu=0 \quad \Leftrightarrow\quad (P-\alpha\cdot\id)\,\mathbb{P}_\nu=0 \quad \text{for an $\alpha\in\R$}.
\end{equation}
In a similar way to \eqref{eq:gurtin_dev}, it follows that
\begin{equation}\label{eq:gurtin_dev_sym}
 \skalarProd{\sym(P\times \nu)}{D}= 0 \ \forall D \text{ with $\tr D=0$} \quad \Leftrightarrow \quad \sym(P\times \nu) =0
\end{equation}
so that again we deduce that ~ $\dev\sym(P\times \nu)= 0$ ~ if and only if ~ $\sym(P\times \nu)=0$, ~ cf.~our Observation \ref{obs:3}.

\subsubsection{The compatible case $\sym(\D u \times \nu)=0$}\label{sec:flatbdry}
Let $\Gamma$ be a relatively open (non-empty) connected subset of the boundary $\partial \Omega$ and assume that $P=\D u$ is compatible.  By the previous observation the condition $\sym(\D u \times \nu)_{|\Gamma}=0$ is fulfilled if and only if there exists a function $\zeta:\Gamma\to\R$ such that ~ $(\D u-\zeta\cdot\id)\,\mathbb{P}_{\nu}{}_{|\Gamma}=0$. If $u=\alpha\,x+b$, then it is clear that this boundary condition is satisfied. On a flat portion of the boundary we will establish also a converse statement. Indeed, let (after possible rotation) $\Gamma\subseteq\R^2\times\{0\}$ be simply connected and $\gamma:[0,1]\to\Gamma$. By the previous observation we have $\frac{\mathrm{d}}{\mathrm{d} s}u(\gamma(s))=\D u(\gamma(s))\,\gamma'(s)=\zeta(\gamma(s))\gamma'(s)$, so that
$u(\gamma(1))=u(\gamma(0))+\int_0^1\zeta(\gamma(s))\gamma'(s)\intd{s}$ and for a closed curve $\gamma$ we deduce
\begin{equation}\label{eq:condition_gamma}
 0 = \int_0^1\zeta(\gamma(s))\gamma'(s)\intd{s} = \int_0^1
 \begin{pmatrix}\skalarProd{\zeta(\gamma(s))\,e_1}{\gamma'(s)}\\
 \skalarProd{\zeta(\gamma(s))\,e_2}{\gamma'(s)}\\
 \skalarProd{\zeta(\gamma(s))\,e_3}{\gamma'(s)}                                                   \end{pmatrix}
\intd{s} = \int_0^1\begin{pmatrix}\skalarProd{(\widehat{\zeta}(\gamma_1(s),\gamma_2(s)),0)}{(\gamma_1'(s),\gamma_2'(s))}_{\R^2}\\ \skalarProd{(0,\widehat{\zeta}(\gamma_1(s),\gamma_2(s)))}{(\gamma_1'(s),\gamma_2'(s))}_{\R^2} \\ 0\end{pmatrix}\intd{s}
\end{equation}
where in the last step we have used, that $\gamma\subset\R^2\times\{0\}$ has vanishing third component and we have set $\widehat{\zeta}(x,y)=\zeta(x,y,0)$. Since \eqref{eq:condition_gamma} is valid for all connected curves $\gamma\subset\Gamma\subseteq\R^2\times\{0\}$  the vector fields $(\widehat{\zeta},0)^T$ and $(0,\widehat{\zeta})^T$ have to be conservative. Thus,
\begin{equation}
 0 = \curl_{2\D}\begin{pmatrix} \widehat{\zeta}\\0\end{pmatrix} = - \widehat{\zeta,}_{y} \quad \text{and} \quad 0 = \curl_{2\D}\begin{pmatrix} 0\\ \widehat{\zeta}\end{pmatrix} = \widehat{\zeta,}_{x}
\end{equation}
where $\curl_{2\D}v=v_2,x-v_1,y$, so that we conclude $\zeta=\widehat{\zeta}\equiv\operatorname{const}$ and set $\zeta(x,y,0)=\alpha$. The previous observation imply $\frac{\mathrm{d}}{\mathrm{d} s}[u(\gamma(s))-\alpha\gamma(s)]\equiv 0$ for all admissible curves $\gamma$ meaning that along $\Gamma$ the function $u$ has to be of a form $u(x,y,0)=\alpha\cdot(x,y,0)^T +(b_1,b_2,0)^T$.

\subsection{Some basic identities}

We outline some basic identities which played useful roles in our considerations:
\bigskip

\begin{tabular}{:l:l:}
\hdashline
&\\
 1. from linear algebra: & 2. and their formal equivalents from calculus:\\[1ex]
 \begin{minipage}{6.5cm}
  \ (a) \ $a\otimes b$ \quad dyadic product,\\
  \hphantom{ (a) } $\skalarProd{a}{b}=\tr(a\otimes b)$ \quad scalar product,\\
  \hphantom{ (a) } $a\times b=\axl(b\otimes a - a\otimes b )$ \quad vector product,\\
  \hphantom{ (a) } $b\times b = 0$,\\
  \hphantom{ (a) } $\skalarProd{a\times b}{b}=0$,\\
  \hphantom{ (a) } $2\,\skew(a\otimes b) = -\Anti(a\times b)$,
 \end{minipage}&
 \begin{minipage}{7cm}
 \ (a) \ $\D a= a\otimes \nabla$,\\
 \hphantom{ (a) } $\div a = \skalarProd{a}{\nabla}=\tr(\D a)$,\\
 \hphantom{ (a) } $\curl a=a\times (-\nabla)=2\axl\skew (\D a)$,\\
 \hphantom{ (a) } $\curl\nabla\zeta\equiv0$,\\
 \hphantom{ (a) } $\div\curl a \equiv0$,\\
 \hphantom{ (a) } $2\,\skew(\D a) = \Anti(\curl a)$,
 \end{minipage}\\[7.5ex]
 \begin{minipage}{6.5cm}
  \ (b) \ $P\, b$,\\
   \hphantom{ (b) }   $\id\, b = b$,\\
 \hphantom{ (b) }  $\Anti(a)\,b= a\times b = -\Anti(b)\,a$,\\
 \hphantom{ (b) } $ A\,b = (\axl A) \times b$,\\
 \hphantom{ (b) } $(a\otimes b)b = \norm{b}^2a$,\\
 \hphantom{ (b) } $(b\otimes a)b = \skalarProd{a}{b}b = \norm{b}^2a+(a\times b) \times b$,
 \end{minipage}&
 \begin{minipage}{7cm}
 \ (b) \ $\Div P = P\,\nabla$,\\
 \hphantom{ (b) } $\Div(\zeta\cdot\id)=\nabla\zeta$,\\
 \hphantom{ (b) } $\Div (\Anti(a))= -\curl a = \Anti(\nabla)\, a$,\\
 \hphantom{ (b) } $\Div A = -\curl\axl A$,\\
 \hphantom{ (b) } $\Div(\D a)= \Delta a$,\\
 \hphantom{ (b) } $\Div((\D a)^T)= \nabla\div a=\Delta a + \curl \curl a$,
 \end{minipage}\\[7.5ex]
 \ (c) \ $P\times b=P\,\Anti(b)=-(\Anti(b)\,P^T)^T$, &
 \ (c) \ $\Curl P = P \times(-\nabla)=-P\,\Anti(\nabla)$,\\[0.8ex]
 \ (d) \ $\id\times b = \Anti(b)\in\so(3)$, &
 \ (d) \ $\Curl(\zeta\cdot\id)=-\Anti(\nabla \zeta)\in\so(3)$,\\[1.4ex]
  \begin{minipage}{7cm}
 \ (e) \ $(a\otimes b)\times b = 0$,\\
 \hphantom{ (e) } $\frac12(b\otimes a)\times b = \sym(a\otimes b)\times b = -\skew(a\otimes b)\times b $\\
\hphantom{ (e) $\frac12(b\otimes a)\times b$} $= -b \otimes \axl\skew(a\otimes b)$,
  \end{minipage}&
 \begin{minipage}{7.4cm}
 \ (e) \ $\Curl(\D a) \equiv 0$, \\
 \hphantom{ (e) }$\frac12\Curl((\D a)^T) =\Curl(\sym \D a) = - \Curl(\skew \D a)$\\
 \hphantom{ (e) $ \frac12\Curl((\D a)^T)$} $= (\D \axl\skew \D a)^T=  \frac12(\D\,\curl a)^T$,
  \end{minipage}\\[3.5ex]
 \begin{minipage}{6.5cm}
 \ (f) \ Room's formulas:\\
 \hphantom{ (f) } $(\Anti(a))\times b = b \otimes a -\skalarProd{b}{a}\cdot\id$,\\
 \hphantom{ (f) } $A\times b = b \otimes \axl A - \skalarProd{b}{\axl A}\cdot \id$,\\
 \hphantom{ (f) } $(\axl A)\otimes b = (A\times b)^T -\frac12\tr(A\times b)\cdot\id$,\\
 \hphantom{ (f) } $\tr(A\times b) = -2\skalarProd{\axl A}{b}$,
 \end{minipage}
 &
 \begin{minipage}{7cm}
 \ (f) \ Nye's formulas:\\
 \hphantom{ (f) } $\Curl (\Anti(a)) = \div a\cdot\id- (\D a)^T$,\\
 \hphantom{ (f) } $\Curl A = \tr(\D \axl A )\cdot\id-(\D \axl A)^T$,\\
 \hphantom{ (f) } $\D \axl(A) = \frac12\tr(\Curl A)\cdot\id-(\Curl A)^T$,\\
 \hphantom{ (f) } $\tr(\Curl A) = 2\,\div\axl A$,
 \end{minipage}\\[7ex]
 \ (g) \ $\tr(S\times b)=0$,&
 \ (g) \ $\tr(\Curl S)\equiv0$,\\[0.8ex]
 \ (h) \ $(P\times b)^T\times b = -\Anti(b)\,P^T\Anti(b) = -b\times P^T\times b$, &
 \ (h) \ $\inc(P)=\Curl[(\Curl P)^T] = -\nabla\times P^T\times\nabla$,\\[0.8ex]
 \ (i) \ $\big(\id\times b\big)^T\times b = \norm{b}^2\cdot\id - b\otimes b\in\Sym(3)$,&
 \ (i) \  $\inc(\zeta\cdot\id)= \Delta\zeta\cdot \id-\D^2\zeta\in\Sym(3)$,\\[1.4ex]
 \begin{minipage}{6.5cm}
 \ (j) \ $\big((b\otimes a)\times b\big)^T\times b =0$,\\
 \hphantom{ (j) } $\big(\sym(a\otimes b)\times b\big)^T\times b =0$,\\
 \hphantom{ (j) } $\big(\skew(a\otimes b)\times b\big)^T\times b =0$,
 \end{minipage}
 &
 \begin{minipage}{7cm}
 \ (j) \ $\inc((\D a)^T)\equiv0$,\\
 \hphantom{ (j) } $\inc(\sym\D a)\equiv 0$,\\
 \hphantom{ (j) } $\inc(\skew\D a)\equiv 0$,
 \end{minipage}\\[4.5ex]
 \ (k) \ $\big((\Anti(a))\times b\big)^T\times b = -\skalarProd{b}{a}\Anti(b)\in\so(3)$,\hspace*{2ex}&
 \ (k) \ $\inc(\Anti(a)) = -\Anti(\nabla \div a)\in\so(3)$,\\[1.4ex]
  \begin{minipage}{6.5cm}
 \ (l) \ $\big(S\times b\big)^T\times b \in\Sym(3)$,\\
 \hphantom{ (l) } $\tr(\big(S\times b\big)^T\times b)=\norm{b}^2\tr(S)-\skalarProd{S}{b\otimes b}_{\R^{3\times3}}$,
  \end{minipage}
 &
 \begin{minipage}{7cm}
 \ (l) \ $\inc S \in\Sym(3)$,\\
  \hphantom{ (l) } $\tr(\inc S) = \Delta \tr(S)-\div\Div S$,
  \end{minipage}\\[2.5ex]
 \ (m) \ $\dev(P\times b)= P\times b + \frac23\skalarProd{\axl\skew P}{b}\cdot\id$,&
 \ (m) \ $\dev\Curl P = \Curl P -\frac23\div\axl\skew P\cdot \id$,\\[1.4ex]
\begin{minipage}{6.5cm}
 \ (n) \ $[(P\times b)^T\times b]^T=(P^T\times b)^T\times b$,\\
 \hphantom{ (n) } $\sym[(P\times b)^T\times b] = ((\sym P)\times b)^T\times b$, \\
  \hphantom{ (n) } $\skew[(P\times b)^T\times b] = ((\skew P)\times b)^T\times b$,
\end{minipage}
&
\begin{minipage}{7cm}
 \ (n) \ $[\inc(P)]^T = \inc(P^T)$,\\
  \hphantom{ (n) }  $\sym\inc P = \inc\sym P$,\\
  \hphantom{ (n) } $\skew\inc P = \inc\skew P$,
\end{minipage}\\[4ex]
 \ (o) \  $\tr[((S\times b)\times b)^T\times b]=0$, &
 \ (o) \ $\tr(\inc \Curl S)\equiv 0$, \\[0.8ex]
\ (p) \ $\skalarProd{a\times b}{c}=-\skalarProd{a}{c\times b}$, & \\[1.4ex]
   \begin{minipage}{6.5cm}
              \ (q) \ $a\otimes b=0 \quad \Leftrightarrow\quad \dev\sym(a\otimes b)=0$, \\
             \hphantom{ \ (q) \  $a\otimes b=0 \quad$}$\Leftrightarrow \quad \dev\sym(\Anti(a)\times b)=0$,
            \end{minipage}
&  \begin{minipage}{7cm} for ~$\zeta\in\mathscr{D}'(\Omega,\R)$, $a\in\mathscr{D}'(\Omega,\R^3)$, $A\in\mathscr{D}'(\Omega,\so(3))$\\ \hphantom{for ~}$S\in\mathscr{D}'(\Omega,\Sym(3))$ and $P\in\mathscr{D}'(\Omega,\R^{3\times 3})$. \end{minipage} \\[4ex]
\ (r) \ $\dev(P\times b) = 0 \quad \Leftrightarrow \quad P\times b = 0$, & \\[0.8ex]
\ (s) \ $\dev\sym(P\times b) = 0 \quad \Leftrightarrow \quad \sym(P\times b) = 0$, & \\[2ex]
for $a,b \in\R^3$, $S\in\Sym(3)$, $A\in\so(3)$  and $P\in\R^{3\times 3}$, &\\
&\\
\hdashline
\end{tabular}
\bigskip

The expression in (l) reads in more details
\begin{align}
 (S\times b)^T\times b &= -b \times S \times b = - \Anti(b)\, S \, \Anti(b)\notag\\
 & = S(b\otimes b)+(b\otimes b)S-\norm{b}^2S -\tr(S)b\otimes b +(\norm{b}^2\tr(S)-\skalarProd{S}{b\otimes b}_{\R^{3\times3}})\cdot\id\,,
 \intertext{so that the formal equivalent for $\inc S$ has the form}
 \inc S & = \D \Div S +(\D \Div S)^T-\Delta S -\D^2\tr(S)+(\Delta \tr(S)-\div\Div S)\cdot\id\,.
\end{align}

\subsection{The kernel of $\Curl$, $\dev\Curl$ and $\sym\Curl$}
\begin{lemma}
 Let $\Omega\subset\R^3$ be a simply connected open set and $P\in\mathscr{D}'(\Omega,\R^{3\times 3})$. Then we have
 \begin{thmenum}
  \item $\Curl P \equiv 0 $ if and only if $P=\D u$, \label{part_a}
  \item $\dev\Curl P \equiv 0 $ if and only if $P=\alpha\cdot\Anti(x)+\D u$, \label{part_b}
  \item $\sym\Curl P \equiv 0 $ if and only if $P=\zeta\cdot\id+\D u$. \label{part_c}
 \end{thmenum}
 where $\alpha\in\R$, $\zeta\in\mathscr{D}'(\Omega,\R)$ and $u\in\mathscr{D}'(\Omega,\R^3)$.
 \end{lemma}
\begin{proof}
Part \ref{part_a} follows by the definition of the matrix $\Curl$ which acts row-wise on matrices and the fact that $\Omega$ is simply connected from the classical Poincar\'{e} lemma, cf.~e.g.~\cite[Theorem 6.17-2]{Ciarlet2013FAbook} but also the historical remarks therein.

Now, let ~ $\dev\Curl P \equiv 0 $ ~ then ~ $\Curl P = 2\,\alpha\cdot\id$ ~ with a scalar field $\alpha$. Since ~ $\Div \Curl P \equiv0$ ~ (in the sense of distributions) we get, taking the matrix $\Div$ on both sides, that
\begin{equation}
  2\,\nabla\alpha = 2\,\Div (\alpha \cdot \id)=\Div\Curl P \equiv0 \quad \Rightarrow \quad \alpha\equiv \operatorname{const}.
\end{equation}
Hence, ~ $\Curl(P-\alpha\cdot\Anti(x))= \Curl P -\alpha\cdot\Curl(\Anti(x))=\Curl P -2\,\alpha\cdot\id\equiv 0$, ~ so that there exists a vector field $u$ such that
\begin{equation}
 P -\alpha\cdot\Anti(x) = \D u.
\end{equation}
Conversely, we have ~ $\dev\Curl(\alpha\cdot\Anti(x)+\D u)=\alpha\cdot\dev\Curl(\Anti(x))=\alpha\cdot\dev(2\cdot\id)\equiv0$, so that part \ref{part_b} follows.

The conclusion of part \ref{part_c} is obtained in a similar way. Indeed, if ~ $\sym\Curl P \equiv 0$ ~ then ~ $\Curl P =\Anti(a)$ ~ for a vector field $a$. Taking the matrix $\Div$ on both sides we obtain
\begin{equation}
-\curl a =\Div \Anti(a) = \Div\Curl P\equiv 0.
\end{equation}
Hence, ($\Omega$ is a simply connected)
\begin{equation}
 a = -\nabla \zeta
\end{equation}
for a scalar field $\zeta$. Moreover, we have
\begin{equation}
 \Curl(\zeta\cdot\id)=-\Anti(\nabla\zeta)=\Anti(a)=\Curl P.
\end{equation}
Thus, ~ $\Curl(P-\zeta\cdot\id)\equiv0$ ~ and therefore there exists a vector field $u$ such that
\begin{equation}
 P - \zeta\cdot \id = \D u.
\end{equation}
Conversely, we have
\begin{equation}
 \sym\Curl(\zeta\cdot\id+\D u)=\sym\Curl(\zeta\cdot\id)=-\sym(\Anti(\nabla\zeta))\equiv0,
\end{equation}
which establishes part \ref{part_c}.
\end{proof}
\begin{remark}
 [Concerning the kernel of $\dev\sym\Curl$] It is clear that, $\Anti(\varphi_C)+\zeta\cdot\id+\D u$ belongs to the kernel of $\dev\sym\Curl$, denotig by $\varphi_C$ the infinitesimal conformal maps, cf.~\eqref{eq:infinitesimalconfis}. However, it is not clear wether these functions already represent the whole class.
\end{remark}

\subsection{The kernel of  $\operatorname{inc} \operatorname{sym}$, $\operatorname{inc} \operatorname{skew}$ and $\operatorname{inc}$}
Let $\Omega\subset\R^3$ be a bounded domain and $P\in\mathscr{D}'(\Omega,\R^{3\times 3})$. Then the Saint-Venant compatibility conditions give
\begin{equation}\label{eq:expression_incsym}
 \inc \sym P \equiv 0 \quad \Leftrightarrow \quad \sym P = \sym \D u,
\end{equation}
where $u\in\mathscr{D}'(\Omega,\R^3)$. Moreover, we have
\begin{equation}\label{eq:expression_incskew}
 \inc \skew P \equiv 0 \quad \Leftrightarrow \quad \skew P = \alpha\cdot\Anti(x) +\skew\D v
\end{equation}
where $\alpha\in\R$ and $v\in\mathscr{D}'(\Omega,\R^3)$. However, it would be desirable to obtain in \eqref{eq:expression_incskew} an expression as in \eqref{eq:expression_incsym} without an additional term in $\Anti(x)$. This can be achieved, e.g.,~assuming additional (boundary) conditions. In a first observation, for ~$\skew P = \skew \D v$~ we obtain ~$\axl\skew P = \axl\skew\D v = \frac12\curl v$,~ where ~$\curl v = 2\, \axl \skew P$~ always has a solution provided that ~$\div\axl\skew P \equiv0$.~ However, we will see, that ~$\inc(\skew P)\equiv0$~ implies only ~$\div\axl\skew P\equiv\operatorname{const}$.~ Indeed, to establish \eqref{eq:expression_incskew} we make use of the expression 
\begin{equation*}
 \inc \Anti(a) = -\Anti(\nabla\div a)
\end{equation*}
valid for all $a\in\mathscr{D}'(\Omega,\R^3)$. Thus,
\begin{equation*}
 \inc \Anti(a)\equiv0 \quad \Leftrightarrow \quad \nabla\div a\equiv 0  \quad \Leftrightarrow \quad \div a \equiv\operatorname{const}.
\end{equation*}
Therefore, there exists an $\alpha\in\R$ such that the vector field $a(x)-\alpha\cdot x$ is solenoidal, i.e.,~ $\div(a(x)-\alpha\cdot x)\equiv0$. ~ Hence, there exists a vector potential $v\in\mathscr{D}'(\Omega,\R^3)$ such that ~ $a(x)-\alpha\cdot x=\curl \frac{v(x)}{2}$. Since ~ $\Anti(\curl v) = 2\, \skew(\D v)$ ~ we obtained the expression from \eqref{eq:expression_incskew} where we have used $a = \axl\skew P$. Conversely, we have $\inc ( \alpha\cdot\Anti(x) +\skew\D v)\equiv 0$, which establishes the relation from \eqref{eq:expression_incskew}.

Furthermore,
\begin{align}
 \inc P \equiv 0 \quad &\Leftrightarrow\quad \sym \inc P \equiv 0 \ \ \wedge \ \ \skew\inc P \equiv 0 \quad \Leftrightarrow\quad \inc\sym P \equiv 0 \ \ \wedge \ \ \inc\skew P \equiv 0 \notag\\
 &\underset{\mathclap{\eqref{eq:expression_incskew}}}{\overset{\mathclap{\eqref{eq:expression_incsym}}}{\Leftrightarrow}} \quad \sym P = \sym \D u \ \ \wedge \ \ \skew P= \alpha\cdot\Anti(x) +\skew\D v\notag \\
 & \Leftrightarrow \quad P = \alpha\cdot\Anti(x) +\sym \D u + \skew\D v 
\end{align}
where  $u,v\in\mathscr{D}'(\Omega,\R^3)$ and $\alpha\in\R$.

\subsection{Rotations and the cross product}
By definition of $\Anti$ it holds $\Anti(a)\,b=a\times b$. Thus, for any rotation $R\in\SO(3)$ we have
\begin{equation}
\Anti(R\,a)\,R\,b=(R\,a)\times(R\,b) = R(a\times b) = R\Anti(a)\,b \qquad \forall a, b \in\R^3
\end{equation}
and therefore
\begin{equation}
 \Anti(R\,a)\,R = R\Anti(a) \quad\Rightarrow \quad \Anti(R\,a) = R\Anti(a)R^T.
\end{equation}
It follows that
\begin{subequations}
\begin{align}
 \Anti(R\,a)\Anti(R\,b)&= R\Anti(a)\Anti(b)R^T
\shortintertext{but also}
 \dev(\Anti(R\,a)\Anti(R\,b))&= R[\dev(\Anti(a)\Anti(b))]R^T,\\ \sym(\Anti(R\,a)\Anti(R\,b))&= R[\sym(\Anti(a)\Anti(b))]R^T,\\
 \dev\sym(\Anti(R\,a)\Anti(R\,b))&= R[\dev\sym(\Anti(a)\Anti(b))]R^T,\\
 \tr(\Anti(R\,a)\Anti(R\,b))&= \tr(\Anti(a)\Anti(b))\,.
 \end{align}
 \end{subequations}
Consequently, we have
\begin{subequations}
\begin{align}
\norm{\Anti(R\,a)\Anti(R\,b)}&=\norm{R\Anti(a)\Anti(b)R^T}=\norm{\Anti(a)\Anti(b)}
\intertext{and, due to the isotropy of $\dev$, $\sym$ and $\dev\sym$, also}
 \norm{\dev(\Anti(R\,a)\Anti(R\,b))}&= \norm{\dev(\Anti(a)\Anti(b))},\\ \norm{\sym(\Anti(R\,a)\Anti(R\,b))}&= \norm{\sym(\Anti(a)\Anti(b))},\\
 \norm{\dev\sym(\Anti(R\,a)\Anti(R\,b))}&= \norm{\dev\sym(\Anti(a)\Anti(b))}\,.
\end{align}
 \end{subequations}
Since $\Anti(\alpha\, a)=\alpha\Anti(a)$ for all $\alpha\in\R$ and in regard with the invariance relations above and the quadratic homogeneity, the considerations in section \ref{sec:3D} can also be obtained with a fixed skew-symmetric matrix, say for $a=e_3$ or, equivalently $A=\Anti(e_3)=e_2\otimes e_1-e_1\otimes e_2$. We demonstrate it in re-proving Observation \ref{obs:2}, i.e.
\begin{equation}\label{eq:normequiv_appendix}
\boxed{
 \frac12\norm{a}^2\norm{b}^2\le\norm{\dev\sym(\Anti(a)\times b)}^2\le\frac23\norm{a}^2\norm{b}^2 \quad \forall a, b\in\R^3.
 }
\end{equation}
By the previous discussion it suffices to establish this estimates already for fixed $a=e_3$. Indeed, we have
\begin{equation}
 \Anti(e_3)\times b = \begin{pmatrix} 0 & -1 & 0 \\ 1 & 0 & 0 \\ 0 & 0 & 0 \end{pmatrix} \begin{pmatrix}0 & -b_3 & b_2 \\ b_3 & 0 & -b_1 \\ -b_2 & b_1 & 0\end{pmatrix} = \begin{pmatrix}-b_3 & 0 & b_1 \\ 0 & -b_3 & b_2\\ 0 & 0 & 0\end{pmatrix}
\end{equation}
Hence,
\begin{equation}
 \norm{\dev\sym(\Anti(e_3)\times b)}^2 = \frac23 b_3^2 +\frac12 b_1^2+ \frac12 b_2^2
\end{equation}
and therefore
\begin{equation}
 \frac12\norm{b}^2\le\norm{\dev\sym(\Anti(e_3)\times b)}^2\le\frac23\norm{b}^2 \,.
\end{equation}
From the estimate \eqref{eq:normequiv_appendix} it follows that for fixed $b\neq0$ the linear map $a\mapsto\dev\sym(\Anti(a)\times b)$ is invertible. A specific inverse map can be obtained as follows. Set
\begin{equation*}
 M\coloneqq \dev\sym(\Anti(a)\times b) \overset{\eqref{eq:iii}}{=} \frac12(a\otimes b + b \otimes a) - \frac13\skalarProd{a}{b}\cdot\id.
\end{equation*}
Thus,
\begin{equation*}
 \tr M = 0, \quad M\,b=\frac12\norm{b}^2a+\frac16\skalarProd{a}{b}\,b, \quad \text{and}\quad \skalarProd{Mb}{b}=\frac23\skalarProd{a}{b}\norm{b}^2
\end{equation*}
and we have
\begin{equation}\label{eq:inversemap}
 a= L_b\,M, \quad \text{where }\quad L_b\,M\coloneqq \frac{2}{\norm{b}^2}\left(M\,b - \frac14\frac{\skalarProd{Mb}{b}}{\norm{b}^2}\, b\right).
\end{equation}

\corrected{
\subsection{Fourier transformation and equivalence of spaces}\label{App:Fourier}
The Fourier transform of $f$ on $\R^3$ is given by
\begin{equation}
 \mathcal{F}f(\xi)=\widehat{f}(\xi)=(2\pi)^{-\frac32}\int_{\R^3} \mathrm{e}^{-\mathrm{i}\langle\xi,x\rangle}f(x)\,\intd{x}, \quad \xi\in\R^3\,.
\end{equation}
If $f$ is sufficiently regular, then $\widehat{\partial_j f}(\xi)= \mathrm{i}\,\xi_j \widehat{f}(\xi)$. Thus, for a sufficiently regular vector field $v:\R^3\to\R^3$ we have
\begin{equation}
 \widehat{\curl v} (\xi)= \mathrm{i}\,\xi\times \widehat{v}(\xi) = \mathrm{i}\,\widehat{v}( \xi)\times (-\xi)
\end{equation}
and for a sufficiently regular matrix field $P:\R^3\to\R^{3\times3}$ it follows
\begin{equation}
 \widehat{\Curl P}(\xi)= -\mathrm{i}\,\widehat{P}(\xi)\times \xi=-\mathrm{i}\,\widehat{P}(\xi)\Anti(\xi).
\end{equation}
Consequently,
\begin{equation}\label{eq:fouriers}
\mathcal{F}\sym\Curl P(\xi)=-\mathrm{i}\,\sym(\widehat{P}(\xi)\times \xi) \quad \text{and}\quad \mathcal{F}\dev\sym\Curl P(\xi)=-\mathrm{i}\,\dev\sym(\widehat{P}(\xi)\times \xi)\,.
\end{equation}
Recall, that by \eqref{eq:normequiv-top} we have the estimate
\begin{equation}
\forall \xi\in\R^3 : \quad
 \norm{\dev\sym (P\times \xi)}\le \norm{\sym(P\times \xi)} \le (1+\sqrt{3})\norm{\dev\sym(P\times \xi)}\,
\end{equation}
which in regard with \eqref{eq:fouriers} gives
\begin{equation}
 c\,\norm{\mathcal{F}\dev\sym \Curl P}_{L^2(\R^3)}\le \norm{\mathcal{F}\sym\Curl P}_{L^2(\R^3)} \le C\,\norm{\mathcal{F}\dev\sym\Curl P}_{L^2(\R^3)}\,.
\end{equation}
Since by Plancherel's theorem the Fourier transformation is an isometry of spaces, i.e.~the $L^2$-norm satisfies $\norm{f}_{L^2(\R^3)}=\norm{\widehat{f}}_{L^2(\R^3)}$ we conclude
\begin{equation}
  c\,\norm{\dev\sym \Curl P}_{L^2(\R^3,\R^{3\times3})}\le \norm{\sym\Curl P}_{L^2(\R^3,\R^{3\times3})} \le C\,\norm{\dev\sym\Curl P}_{L^2(\R^3,\R^{3\times3})}\,.
\end{equation}
In other words ~ $\dev\sym \Curl P\in L^2(\R^3,\R^{3\times3})$ ~ if and only if ~ $\sym \Curl P\in L^2(\R^3,\R^{3\times3})$, ~ thus, establishing the equivalence of spaces without boundary conditions
$$W^{1,2}(\dev\sym\Curl;\R^3,\R^{3\times3})=W^{1,2}(\sym\Curl;\R^3,\R^{3\times3}),$$
as well as the norm equivalence
\begin{equation*}
  \norm{P}_{L^2(\R^3,\R^{3\times3})}+ \norm{\sym\Curl P}_{L^2(\R^3,\R^{3\times3})} \le C(\norm{P}_{L^2(\R^3,\R^{3\times3})}+\norm{\dev\sym\Curl P}_{L^2(\R^3,\R^{3\times3})})\,.
\end{equation*}
Hence, we conclude
\begin{align*}
 \widehat{\skew P}(\xi) &\overset{\eqref{eq:inversemap}}{=} \Anti\left(L_{\xi/\norm{\xi}}\left[\dev\sym\left((\widehat{\skew P}(\xi))\times\frac{\xi}{\norm{\xi}}\right)\right] \right) = \Anti\left(L_{\xi/\norm{\xi}}\left[\dev\sym\left((\skew\widehat{P}(\xi))\times\frac{\xi}{\norm{\xi}}\right)\right] \right)\\
 &\ \ = \ \ \Anti\left(L_{\xi/\norm{\xi}}\left[\dev\sym\left(((\widehat{P}-\sym\widehat{P})(\xi))\times\frac{\xi}{\norm{\xi}}\right)\right] \right) \\
 & \overset{\eqref{eq:fouriers}}{=}  \Anti\left(L_{\xi/\norm{\xi}}\left[\frac{-\mathrm{i}}{\norm{\xi}}\mathcal{F}\dev\sym\Curl P(\xi)\right] - L_{\xi/\norm{\xi}}\left[\dev\sym\left((\widehat{\sym P}(\xi))\times\frac{\xi}{\norm{\xi}}\right)\right]\right).
\end{align*}
Thus, standard multiplier estimates will give $L^p$ estimates (for $1<p<\infty$) in the periodic setting and in the whole space $\R^3$.
}

\end{alphasection}

\end{document}